\documentclass[12pt,psamsfonts,leqno,oneside,letterpaper]{amsart}
\usepackage[dvips,text={6.5truein,9truein},left=1truein,top=1truein]{geometry}
\usepackage{amssymb,amsmath,amscd,enumerate}
\usepackage[pdftex]{graphicx}
\usepackage{url}

\usepackage[colorlinks,linkcolor=blue,citecolor=blue,pdfstartview=FitH]{hyperref}
\input xy
\xyoption{all}
\SelectTips{cm}{12}

\usepackage{color}

\theoremstyle{definition}
\newtheorem{para}{}[section]

\newtheorem{remark}[para]{Remark}
\newtheorem{remarks}[para]{Remarks}
\newtheorem{notation}[para]{Notation}
\newtheorem{convention}[para]{Convention}
\newtheorem{definition}[para]{Definition}
\newtheorem{definitions}[para]{Definitions}
\newtheorem{claim}[equation]{Claim}
\newtheorem{fact}[para]{Fact}
\newtheorem*{facts}{Facts}
\newtheorem{example}[para]{Example}

\newcommand\Alternatives{\begin{enumerate}[(i)]}
\newcommand\EndAlternatives{\end{enumerate}}
\newcommand\Conditions{\begin{enumerate}[(1)]}
\newcommand\EndConditions{\end{enumerate}}

\theoremstyle{plain}
\newtheorem{theorem}[para]{Theorem}
\newtheorem{lemma}[para]{Lemma}
\newtheorem{proposition}[para]{Proposition}
\newtheorem{corollary}[para]{Corollary}
\newtheorem{conjecture}[para]{Conjecture}

\numberwithin{equation}{para}
\numberwithin{figure}{section}

\newcommand\Number{\begin{para}}
\newcommand\EndNumber{\end{para}}
\newcommand\Definition{\begin{definition}}
\newcommand\EndDefinition{\end{definition}}
\newcommand\Definitions{\begin{definitions}}
\newcommand\EndDefinitions{\end{definitions}}
\newcommand\Theorem{\begin{theorem}}
\newcommand\EndTheorem{\end{theorem}}
\newcommand\Conjecture{\begin{conjecture}}
\newcommand\EndConjecture{\end{conjecture}}
\newcommand\Remark{\begin{remark}}
\newcommand\EndRemark{\end{remark}}
\newcommand\Remarks{\begin{remarks}}
\newcommand\EndRemarks{\end{remarks}}
\newcommand\Convention{\begin{convention}}
\newcommand\EndConvention{\end{convention}}
\newcommand\Notation{\begin{notation}}
\newcommand\EndNotation{\end{notation}}
\newcommand\Lemma{\begin{lemma}}
\newcommand\EndLemma{\end{lemma}}
\newcommand\Proposition{\begin{proposition}}
\newcommand\EndProposition{\end{proposition}}
\newcommand\Corollary{\begin{corollary}}
\newcommand\EndCorollary{\end{corollary}}
\newcommand\Claim{\begin{claim}}
\newcommand\EndClaim{\end{claim}}
\newcommand\Proof{\begin{proof}}
\newcommand\EndProof{\end{proof}}
\newcommand\Equation{\begin{equation}}
\newcommand\EndEquation{\end{equation}}

\newcommand\Bullets{\begin{itemize}}
\newcommand\EndBullets{\end{itemize}}

\newcommand\PtsToPoly{\cite[Lemma 1.4]{DeB_cyclic_geom}}  
\newcommand\OneSide{\cite[Lemma 1.5]{DeB_cyclic_geom}} 
\newcommand\IsoscelesDecomp{\cite[Lemma 1.6]{DeB_cyclic_geom}} 
\newcommand\LongestSide{\cite[Corollary 1.11]{DeB_cyclic_geom}} 
\newcommand\CenteredSpace{\cite[Definition 3.1]{DeB_cyclic_geom}} 
\newcommand\BCn{\cite[Lemma 3.4]{DeB_cyclic_geom}} 
\newcommand\RadiusFunction{\cite[Lemma 3.6]{DeB_cyclic_geom}} 
\newcommand\InCenteredClosure{\cite[Lemma 3.9]{DeB_cyclic_geom}} 
\newcommand\HCn{\cite[Lemma 6.4]{DeB_cyclic_geom}} 
\newcommand\VerticalRay{\cite[Corollary 6.5]{DeB_cyclic_geom}} 
\newcommand\ShrinkingParameters{\cite[Lemma 6.12]{DeB_cyclic_geom}}
\newcommand\FullSector{\cite[Lemma 5.2]{DeB_cyclic_geom}} 
\newcommand\DefectDerivative{\cite[Proposition 5.5]{DeB_cyclic_geom}} 

\renewcommand\Re{\mathop{\rm Re}}
\renewcommand\Im{\mathop{\rm Im}}

\newcommand\calp{{\mathcal P}}

\newcommand\calAC{\mathcal{AC}}
\newcommand\calBC{\mathcal{BC}}
\newcommand\calHC{\mathcal{HC}}
\newcommand\calc{{\mathcal C}}
\newcommand\cale{\mathcal{E}}
\newcommand\calf{\mathcal{F}}

\newcommand\calh{{\mathcal H}}

\newcommand\cals{\mathcal{S}}

\newcommand\co{\colon\thinspace}

\newcommand\bd{\mathbf{d}}

\newcommand\bs{\mathbf{s}}
\newcommand\bt{\mathbf{t}}
\newcommand\bu{\mathbf{u}}
\newcommand\bv{\mathbf{v}}
\newcommand\bw{\mathbf{w}}
\newcommand\bx{\mathbf{x}}
\newcommand\by{\mathbf{y}}
\newcommand\bz{\mathbf{z}}

\begin{document}

\title[Maximal injectivity radius of hyperbolic surfaces]{The centered dual and the maximal injectivity radius of hyperbolic surfaces}

\author{Jason DeBlois}
\address{Department of Mathematics\\University of Pittsburgh}
\email{jdeblois@pitt.edu}
\thanks{Partially supported by NSF grant DMS-1240329}

\begin{abstract}  We give sharp upper bounds on the maximal injectivity radius of finite-area hyperbolic surfaces and use them, for each $g\geq 2$, to identify a constant $r_{g-1,2}$ such that the set of closed genus-$g$ hyperbolic surfaces with maximal injectivity radius at least $r$ is compact if and only if $r>r_{g-1,2}$.  The main tool is a version of the ``centered dual complex'' that we introduced earlier, a coarsening of the Delaunay complex.  In particular, we bound the area of a compact centered dual two-cell below given lower bounds on its side lengths.\end{abstract}

\maketitle

This paper analyzes the ``centered dual complex" of a locally finite subset $\cals$ of $\mathbb{H}^2$, first introduced in our prior preprint \cite{DeB_tessellation}, and applies it to describe the maximal injectivity radius of hyperbolic surfaces.  The centered dual complex is a cell decomposition with vertex set $\cals$ and totally geodesic edges.  Its underlying space contains that of the geometric dual to the Voronoi tessellation.  We regard it as a tool for understanding the geometry of packings.

The rough idea behind the construction is that geometric dual $2$-cells that are not ``centered'' (see Definition \ref{first centered}) are hard to analyze individually but naturally group into larger cells that can be treated as units.  Our first main theorem bears the fruit of this approach, turning a lower bound on edge lengths into a good lower bound on area for centered dual $2$-cells.

\newcommand\maintheorem{Let $C$ be a compact two-cell of the centered dual complex of a locally finite set $\cals\subset\mathbb{H}^2$, such that for some fixed $d>0$ each edge of $\partial C$ has length at least $d$.  If $C$ is a triangle then its area is at least that of an equilateral hyperbolic triangle with side lengths $d$.  If $\partial C$ has $k > 3$ edges then:  
$$ \mathrm{Area}(C) \geq (k-2) A_{\mathit{m}}(d) $$
Here $A_{\mathit{m}}(d)$ is the maximum of areas of triangles with two sides of length $d$, that of a triangle whose third side is a diameter of its circumcircle.}
\newtheorem*{MainTheorem}{Theorem \ref{main}}
\begin{MainTheorem}\maintheorem\end{MainTheorem}

The bounds of Theorem \ref{main} do not hold for arbitrary Delaunay or geometric dual cells, even triangles.  The theorem further gives explicit form to the assertion that $\cals$ has low density in a centered dual two-cell of high combinatorial complexity.  We prove an analog of Theorem \ref{main} for centered dual $2$-cells of finite complexity that are not compact, in Theorem \ref{main for noncompact}.  

Our next main theorem, which uses Theorems \ref{main} and \ref{main for noncompact}, illustrates the sort of application we have in mind for the centered dual complex.  Below let $\mathit{injrad}_x F$ denote the \textit{injectivity radius} of a hyperbolic surface $F$ at $x\in F$, half the length of the shortest non-constant geodesic arc in $F$ that begins and ends at $x$.

\newcommand\mainapp{For $r > 0$, let $\alpha(r)$ be the angle of an equilateral hyperbolic triangle with sides of length $2r$, and let $\beta(r)$ be the angle at either endpoint of the finite side of a horocyclic ideal triangle with one side of length $2r$:\begin{align*}
 & \alpha(r) = 2\sin^{-1}\left(\frac{1}{2\cosh r}\right) &
 & \beta(r) = \sin^{-1}\left(\frac{1}{\cosh r}\right) \end{align*}
A complete, oriented, finite-area hyperbolic surface $F$ with genus $g\geq 0$ and $n\geq 0$ cusps has injectivity radius at most $r_{g,n}$ at any point, where $r_{g,n}>0$ satisfies:\begin{align*}
  (4g+n-2)3\alpha(r_{g,n}) + 2n\beta(r_{g,n}) = 2\pi \end{align*}
Moreover, the collection of such surfaces with injectivity radius $r_{g,n}$ at some point is a non-empty finite subset of the moduli space $\mathfrak{M}_{g,n}$ of complete, oriented, finite-area hyperbolic surfaces of genus $g$ with $n$ cusps.}
\newtheorem*{MainApp}{Theorem \ref{main app}}

\begin{MainApp}\mainapp\end{MainApp}

The closed (i.e.~$n=0$) case of Theorem \ref{main app} was proved by Christophe Bavard \cite{Bavard}.  It follows from ``B\"or\"oczky's theorem'' \cite{Bor}, which bounds the local density of constant-radius packings of $\mathbb{H}^2$, since a disk embedded in a hyperbolic surface has as its preimage a packing of the universal cover $\mathbb{H}^2$ with constant local density.  We reproduce this argument in Lemma \ref{Boroczky}.

The general case does not follow in the same way, since the preimage of a maximal-radius embedded disk on a noncompact hyperbolic surface is not a maximal-density packing of $\mathbb{H}^2$.

By basic calculus, $\alpha$ and $\beta$ are decreasing functions of $r$ with $\alpha(r)< \beta(r)$ for each $r>0$.  Thus if $g'\leq g$ and $n'\leq n$ then $r_{g',n'}\leq r_{g,n}$.  It also happens that $2\beta(r) < 3\alpha(r)$ for each $r>0$, see Corollary \ref{special monotonicity}; whence $r_{g-1,n+2} < r_{g,n}$ for any $g>0$ and $n\geq 0$.  Therefore:
$$ r_{0,2g} < r_{1,2g-2} <\cdots < r_{g-1,2} < r_{g,0}, $$
for any $g \geq 2$.  This relates Theorem \ref{main app}'s upper bounds on maximal injectivity radius of surfaces with a fixed even Euler characteristic.  It implies compactness results for certain subsets of moduli space.  Below we use the topology of geometric convergence on $\mathfrak{M}_g\doteq \mathfrak{M}_{g,0}$ (see \cite[\S E.1]{BenPet}).  This is the usual, ``algebraic'' topology on $\mathfrak{M}_g$ (compare eg.~\cite[\S 10.3]{FaMa}).

\begin{corollary}\label{compactness}For $g\geq 2$, the collection of surfaces of maximal injectivity radius at least $r$,
$$\mathfrak{C}_{\geq r,g}\doteq \{F\ \mbox{orientable, closed, and hyperbolic}\,|\,\mathit{injrad}_x F \geq r\ \mbox{for some}\ x\in F\},$$ 
is a compact subset of $\mathfrak{M}_g$ if and only if $r >  r_{g-1,2}$.\end{corollary}

Corollary \ref{compactness} contrasts with ``Mumford's compactness criterion'' \cite{Mumford} which asserts compactness, for any $\epsilon >0$, of the subset of $\mathfrak{M}_g$ consisting of surfaces with \textit{minimal} injectivity radius at least $\epsilon$.  However it is a standard consequence of the Margulis Lemma that $\mathfrak{C}_{\geq \epsilon_2,g} = \mathfrak{M}_g$ (and hence is noncompact), where $\epsilon_2$ is the $2$-dimensional Margulis constant.  On the other hand, by Theorem \ref{main app} $\mathfrak{C}_{\geq r_g,0}$ is finite and hence compact.

We will sketch a proof of Corollary \ref{compactness} below that uses Theorem \ref{main app} and standard results on geometric convergence (eg.~from \cite[Ch.~E]{BenPet}).  Details can be easily filled in.

\begin{proof}[Proof of Corollary \ref{compactness}]  It is a key fact that if $(F,x)$ is a pointed geometric limit of $\{(F_n,x_n)\}$, then $\mathit{injrad}_x F = \lim_{n\to\infty} \mathit{injrad}_{x_n} F_n$.  This implies that $\mathfrak{C}_{\geq r,g}$ is closed in $\mathfrak{M}_g$.  For $r > r_{g-1,2}$ we will show that it is also bounded; i.e.~contained in one of the ``Mumford sets'' above.

Let $F_n$ be a sequence of closed, oriented, genus-$g$ hyperbolic surfaces with (minimal) injectivity radius approaching $0$, and for each $n$ let $x_n\in F_n$ be a point at which injectivity radius is maximal.  A subsequence of $\{(F_n,x_n)\}$ has a geometric limit $(F,x)$, where $F$ is a non-compact hyperbolic surface with $\mathrm{Area}(F)\leq \mathrm{Area}(F_n)$, hence $\chi(F) \geq 2-2g$, and $x\in F$.   Then $\mathit{injrad}_x F \leq r_{g-1,2}$, by the discussion below Theorem \ref{main app}.  Thus by the key fact the $F_n$ are not all in $\mathfrak{C}_{\geq r,g}$ for any $r > r_{g-1,2}$.  

Thus $\mathfrak{C}_{\geq r,g}$ is closed and bounded in $\mathfrak{M}_g$, hence compact, for $r>r_{g-1,2}$.  Example \ref{closed sequence} describes a sequence in $\mathfrak{C}_{\geq r_{g-1,2},g}$ with minimal injectivity radius approaching $0$, showing that it is not compact.\end{proof}

It is straightforward to extend Corollary \ref{compactness} to moduli spaces of non-compact surfaces, or the bounds of Theorem \ref{main app} to multiple-disk, equal-radius packings on surfaces.  In future work we will apply the centered dual machine to more subtle packing problems on surfaces.  

We now give a brief overview of the paper.  Section \ref{intro Voronoi} recalls basic properties of the Voronoi tessellation of a locally finite subset $\cals$ of $\mathbb{H}^n$ and its geometric dual complex, before pointing out some special features of the two-dimensional setting.  Lemma \ref{vertex polygon} includes the key fact that every geometric dual $2$-cell is \textit{cyclic}: inscribed in a metric circle.  Hence it is determined up to isometry by its collection of side lengths \cite{Schlenker}.

The centered dual complex of $\cals$ is defined in Section \ref{dirichlet dual}.  This runs parallel to Section 3 of \cite{DeB_tessellation}, but the definitions are modified to accommodate non-compact Voronoi edges.  The fact that motivates our definition is that among cyclic polygons in $\mathbb{H}^2$, increasing the length of an edge increases area if and only if that edge is not the longest of a non-centered polygon, see \cite{DeB_cyclic_geom}.  Here is the definition of a ``centered'' polygon.

\begin{definition}\label{first centered}  A polygon $P$ inscribed in a circle $S$ is \textit{centered} if the center of $S$ is in $\mathit{int}\,P$.\end{definition} 

The centered/non-centered dichotomy has been previously considered in the literature, eg.~in \cite{VHGR} (there ``centered'' goes by ``well-centered'').  Centered dual two-cells collect non-centered two-cells of the geometric dual in a natural way.  Two fundamental observations here are Lemma \ref{vertex polygon centered}, relating non-centeredness of geometric dual cells to non-centeredness of Voronoi edges (see Definition \ref{centered edge}), and Lemma \ref{tree components}, describing the structure of the set of these edges.

Centered dual $2$-cells are not determined by their edge lengths, but the set of possible centered dual two-cells with a given combinatorics and edge length collection is parametrized by a compact \textit{admissible space}.  This is defined in Section \ref{admissible for compact}, which parallels Section 5 of \cite{DeB_tessellation}.  The area of centered dual $2$-cells determines a function on the admissible space.  Theorem \ref{main} is proved in Section \ref{compact bounds} by bounding this function below.

Section \ref{moduli for noncompact} has the same structure as \ref{moduli}.  It describes admissible spaces for non-compact centered dual $2$-cells and finishes with a proof of Theorem \ref{main for noncompact}.  We finally consider hyperbolic surfaces in Section \ref{limits}, proving Theorem \ref{main app} there and describing some examples.

\subsection*{Acknowledgements}  This work was prompted by an offhand observation of Marc Culler.  Thanks to Hugo Parlier for pointing out to me the existence of \cite{Bavard}.

\section{The Voronoi tessellation and its geometric dual}\label{intro Voronoi}

In this section we will record some facts about the Voronoi tessellation of a locally finite subset $\cals$ of hyperbolic space and its geometric dual, using \cite{DeB_Delaunay} as a general reference.  We will also establish notation and collect some facts that hold only in the $2$-dimensional setting.  

The \textit{Voronoi tessellation} has $n$-cells in bijection with $\cals$.  The assertions below are from \cite[Lemma 5.2]{DeB_Delaunay}.  For $\bs\in\cals$, the corresponding Voronoi $n$-cell is the convex polyhedron:
$$ V_{\bs} = \{\bx\in\calh^n\,|\, d_H(\bs,\bx)\leq d_H(\bs',\bx)\ \forall\ \bs'\in\cals\} $$
Here $d_H$ is the hyperbolic distance.  The collection of Voronoi $n$-cells is locally finite, and cells of lower dimension are by definition of the form $\bigcap_{i=0}^{n} V_{\bs_i}$ for subsets $\{\bs_0,\hdots,\bs_n\}$ of $\cals$.

The result below, from Corollary 5.5 of \cite{DeB_Delaunay}, identifies the geometric dual to a Voronoi cell.

\begin{proposition}\label{geometric dual}  Let $\cals\subset\mathbb{H}^n$ be locally finite.  For a $k$-cell $V$ of the Voronoi tessellation, if $\cals_0\subset\cals$ is maximal such that $V = \bigcap_{\bs\in\cals_0} V_{\bs}$ then the closed convex hull $C_V$ of $\cals_0$ in $\mathbb{H}^n$ is the \mbox{\rm geometric dual} to $V$, an $(n-k)$-dimensional, compact convex polyhedron in $\mathbb{H}^n$.\end{proposition}

For a locally finite set $\cals$, say the \textit{geometric dual complex} of $\cals$ is the collection of geometric duals to Voronoi cells.  The result below shows it is a \textit{polyhedral complex} in the sense of \cite[Dfn.~2.1.5]{DeLoRS}, and characterizes it by an empty circumspheres condition.

\begin{theorem}[\cite{DeB_Delaunay}, Theorem 5.9]\label{geometric dual char}  Suppose $\cals\subset\mathbb{H}^n$ is locally finite.  For any metric sphere $S$ that intersects $\cals$ and bounds a ball $B$ with $B\cap\cals = S\cap\cals$, the closed convex hull of $S\cap\cals$ in $\mathbb{H}^n$ is a geometric dual cell.  Every geometric dual cell is of this form.  Moreover, if $C$ is the geometric dual to a Voronoi cell then so is every face of $C$, and any geometric dual cell $C'\neq C$ that intersects $C$ does so in a face of each.\end{theorem}

We now specialize to dimension $2$ and make some definitions.

\begin{definition}\label{cyclic poly}  We say a polygon $C\subset \mathbb{H}^2$ is \textit{cyclic} if its vertex set is contained in a metric circle $S$.  The \textit{center} $v\in \mathbb{H}^2$ and \textit{radius} $J>0$ of a cyclic $n$-gon $C$ are respectively the center and radius of its circumcircle $S$ (so $S=\{\bx\,|\,d_H(v,\bx) = J\}$), and $C$ is \textit{centered} if $v\in\mathit{int}\,C$.  The vertex set $\cals_0 = \{\bs_0,\hdots,\bs_{n-1}\}$ of such $C$ is \textit{cyclically ordered} if $\bs_i$ shares an edge with $\bs_{i+1}$ for each $i$ (taking $i+1$ modulo $n$).\end{definition}

\begin{lemma}\label{vertex polygon}  For a vertex $v$ of the Voronoi tessellation of a locally finite set $\cals\subset\mathbb{H}^2$, the geometric dual $C_v$ to $v$ is a cyclic polygon with vertex set $\cals_0\subset\cals$ such that for $\bs\in\cals$, $v\in V_{\bs}$ if and only if $\bs\in\cals_0$.  $C_v$ has center $v$ and radius $J_v \doteq d(v,\bs)$ for any $\bs\in\cals_0$; and:\begin{itemize}
  \item  If $\cals_0 = \{\bs_0,\hdots,\bs_{n-1}\}$ is cyclically ordered then the Voronoi $2$-cell $V_{\bs_i}$ shares an edge $e_i$ with $V_{\bs_{i+1}}$ for each $i$ (taking $i+1$ modulo $n$).
  \item  For each $i\in\{0,\hdots,n-1\}$, the geometric dual $\gamma_i$ to $e_i$ as above joins $\bs_i$ to $\bs_{i+1}$.\end{itemize}
For $v\ne w$, $\mathit{int}\, C_v\cap\mathit{int}\, C_w = \emptyset$, and $C_v$ shares an edge with $C_w$ if and only if $v$ and $w$ are opposite endpoints of a Voronoi edge.\end{lemma}

That the geometric dual to a Voronoi vertex is cyclic follows from Theorem \ref{geometric dual char}.  Proposition \ref{geometric dual} implies $\mathit{int}\,C_v\cap\mathit{int}\,C_w$ for $v\neq w$.  Together with the definitions here, it also implies the fact below, which is useful to record separately:

\begin{fact}\label{vertex radius} Say the \textit{radius} of a Voronoi vertex $v$ is the radius $J_v$ of its geometric dual $C_v$.  For every $\bs\in\cals$, $d_H(v,\bs)\geq J_v$, and equality holds if and only if $\bs$ is a vertex of $C_v$.\end{fact}

In two dimensions the vertex set of any polygon admits a cyclic order.  The remaining assertions of Lemma \ref{vertex polygon} follow from \cite[Lemma 5.8]{DeB_Delaunay}.  The facts below are straightforward:

\begin{facts}  Suppose $\cals\subset\mathbb{H}^2$ is locally finite.\begin{itemize}
\item  Each Voronoi edge is the intersection of exactly two Voronoi $2$-cells $V_{\bs}$ and $V_{\bt}$, for $\bs,\bt\in\cals$, and its geometric dual is the arc $\gamma_{\bs\bt}$ joining $\bs$ to $\bt$.
\item  Each Voronoi vertex $v$ is the intersection of at least three Voronoi $2$-cells.\end{itemize}\end{facts}

The geometric dual complex of a locally finite set $\cals$ is a subcomplex of what we call the \textit{Delaunay tessellation} in \cite{DeB_Delaunay}, whose underlying space contains the convex hull of $\cals$.  In important special cases (eg.~if $\cals$ is finite or lattice-invariant, see respectively Prop.~3.5 or Theorem 6.23 there), the Delaunay tessellation is a locally finite polyhedral complex.  It is important to note that the geometric dual may be a proper subcomplex even in good conditions.  See below, which reproduces Example 5.11 of \cite{DeB_Delaunay}.

\begin{example}\label{no dual}  Figure \ref{three point Vor} illustrates the Voronoi and Delaunay tessellations determined by three points in $\mathbb{H}^2$, using the upper half-plane model.  In each case the Delaunay triangle spanned by $x$, $y$, and $z$ is shaded, with its edges dashed.  The edges of the Voronoi tessellation are in bold.  The Euclidean circumcircle for $\bx$, $\by$, and $\bz$ is also included in each case.

\begin{figure}
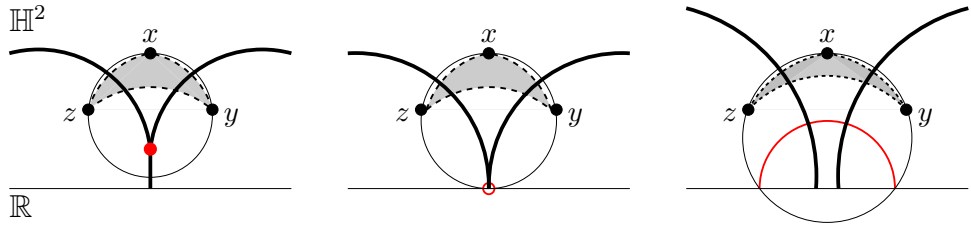
\caption{Delaunay and Voronoi tessellations of three-point sets in $\mathbb{H}^2$.}
\label{three point Vor}
\end{figure}

In the left case the Delaunay tessellation and the geometric dual complex coincide.  In particular, the Delaunay triangle is the geometric dual to the Voronoi vertex: the red dot.  In the middle and on the right, the Voronoi tessellation has no vertex and the Delaunay triangle has no geometric dual; instead, the geometric dual to the Voronoi tessellation has cells $\bx$, $\by$, $\bz$, and the two edges containing $\bx$.

This trichotomy reflects that the Euclidean circumcircle for $\bx$, $\by$, and $\bz$ is a metric hyperbolic circle in the left case, centered at the red dot, and intersects $\mathbb{H}^2$ in a horocycle and geodesic equidistant, respectively, in the middle and right cases.  In particular, the triangle spanned by $\bx$, $\by$, and $\bz$ is cyclic only in the left-hand case.\end{example}

Let us make some precise definitions connected with the upper half-plane model for $\mathbb{H}2$.

\begin{definition}\label{sphere at infinity}  The \textit{upper half-plane model} for $\mathbb{H}^2$ is $\{\bz\in\mathbb{C}\,|\,\Im \bz >0\}$, equipped with the inner product $\langle\bv,\bw\rangle = \frac{\bv\cdot\bw}{\Im z}$ for $\bv,\bw\in T_z\mathbb{H}^2$.  

The \textit{sphere at infinity} of $\mathbb{H}^2$ is $S_{\infty} = \mathbb{R}\cup\{\infty\}$.  For $r\in\mathbb{R}$, a \textit{horocycle} $S$ with \textit{ideal point} $r$ is the non-empty intersection with $\mathbb{H}^2$ of a Euclidean circle in $\mathbb{C}$ tangent to $\mathbb{R}$ at $r$.  The \textit{horoball} $B$ bounded by $S$ is the intersection with $\mathbb{H}^2$ of the Euclidean ball that $S$ bounds.  A horocycle centered at $\infty$ is a horizontal line in $\mathbb{H}^2$, and the horoball that it bounds is the half-plane contained in $\mathbb{H}^2$.\end{definition}

Geodesics of the upper half-plane model are the intersections with $\mathbb{H}^2$ of Euclidean circles and straight lines that meet $\mathbb{R}$ perpendicularly.  Every geodesic ray thus has a well-defined ideal endpoint in $S_{\infty}$ (if it points up in a straight line, its ideal endpoint is $\infty$).

The isometry group of $\mathbb{H}^2$ is $\mathrm{PGL}_2(\mathbb{R})$, acting by M\"obius transformations.  It takes geodesics to geodesics and horocycles to horocycles and extends to a triply transitive action on $S_{\infty}$.

\begin{lemma}\label{exclusive horocycle}  For a locally finite set $\cals\subset\mathbb{H}^2$, if a Voronoi edge $e=V_{\bs}\cap V_{\bt}$, with $\bs,\bt\in\cals$, has an ideal endpoint $v_{\infty}\in S_{\infty}$ then there is a unique horocycle $S$ through $\bs$ and $\bt$ with ideal point $v_{\infty}$, and the horoball $B$ that it bounds satisfies $B\cap\cals= S\cap\cals$.\end{lemma}

\begin{proof}  We work in the upper half-plane model.  After moving $\cals$ by an isometry, $e$ is a subinterval $[iy_0,\infty)$ of $i\mathbb{R}^+$ and $v_{\infty} = \infty$.  Each horocycle with ideal point $\infty$, being a horizontal line, is preserved by reflection $\rho$ through $i\mathbb{R}^+$.  Since $i\mathbb{R}^+$ perpendicularly bisects the geometric dual $\gamma$ to $e$, $\rho$ preserves $\gamma$ and exchanges its endpoints $\bs$ and $\bt$.  They thus lie on the same horocycle through $\infty$.  Moving $\cals$ again, by an isometry preserving $i\mathbb{R}^+$, we may assume this is $S_{\infty} =\mathbb{R}+i$; so $\bs = -x_0+i$, $\bt = x_0+i$,  for some $x_0>0$.

For each $u\geq y_0$, the hyperbolic circle $S_u$ centered at $\bu = iu$ containing $\bx$ and $\by$ has no points of $\cals$ in the interior of the disk that it bounds, since $\bu\in V_{\bs}\cap V_{\bt}$.  Direct computation reveals that this hyperbolic circle is identical to the Euclidean circle of radius $u\sinh r_u$ centered at $(0,u\cosh r_u)$, where $r_u = d(\bu,\bt)$ satisfies $\cosh r_u = (x_0^2+u^2+1)/2u$.    (Recall that circles of the upper half-plane model are Euclidean circles contained in $\mathbb{H}^2$.)  One can use this to show in particular that $r_u > \log u$.

The convex complementary component to $S_{\infty}$ is $\{x+iy\,|\,y>1\}$.  For $z = x+iy$ in this complementary component, we claim there exists $u_1\geq y_0$ such that $S_u$ encloses $z$ for all $u>u_1$.  This is obvious if $|x|\leq x_0$, taking $u_1 = y$, say, so assume that $|x| > x_0$.  For a point $x+iy_u$ on $S_u$, the Euclidean distance formula gives:\begin{align}
  \label{eucdist} u^2\sinh^2 r_u = x^2 + (u\cosh r_u-y_u)^2 \end{align}
Solving for $y_u < u\cosh r_u$ gives
$$ y_u = u\cosh r_u - \sqrt{u^2\sinh^2 r_u - x^2} = \frac{u^2 + x^2}{u\cosh r_u + \sqrt{u^2\sinh^2 r_u - x^2}} \leq 1+\frac{x^2}{u^2} $$
Above we applied the identity $\sqrt{x}-\sqrt{y} = (x-y)/(\sqrt{x}+\sqrt{y})$, then noted that the resulting denominator is at at least $u\cosh r_u \geq u^2$.  We have also assumed without loss of generality that $u$ is large enough that the quantity under the square root is positive.

Such a solution $y_u$ is bounded below by $1$, so it is clear that $y_u\to 1$ as $u\to\infty$.  A simpler argument shows that the solution $y_u >u\cosh r_u$ to (\ref{eucdist}) increases without bound as $u\to\infty$, and the claim follows.  But the claim implies the result since for any $\bu\in e$, no point of $\cals$ has distance less than $d(\bu,\bt)$ from $\bu$.\end{proof}

\section{The centered dual to the Voronoi tessellation}\label{dirichlet dual}

Here we relate the various aspects of ``non-centeredness'', the central notion of the paper.  

\begin{definition}\label{centered edge}  For a locally finite set $\cals\subset\mathbb{H}^2$, we will say an edge $e$ of  the Voronoi tessellation of $\cals$ is \textit{centered} if $e$ intersects its geometric dual edge $\gamma_{\bs\bt}$ at a point in $\mathit{int}\,e$.  If $e$ is not centered, we orient it pointing away from $\gamma_{\bs\bt}$.

We will refer to the one-skeleton of the Voronoi tessellation as the \textit{Voronoi graph}, and to the union of its non-centered edges as the \textit{non-centered Voronoi subgraph}.\end{definition}

We will describe the structure of the non-centered Voronoi subgraph in Section \ref{non-centered structure}.  In Section \ref{centered dual props} we use this structure to organize the ``centered dual decomposition'', in Definition \ref{centered dual}, and prove its basic properties.

\subsection{Non-centeredness in the Voronoi graph}\label{non-centered structure}  We relate centeredness of edges to centeredness of geometric dual cells (in the sense of Definition \ref{cyclic poly}) in Lemma \ref{vertex polygon centered}.  The non-centered Voronoi subgraph has restricted combinatorics: its components are trees, each with a canonical root vertex if finite, see Lemma \ref{tree components}.

\begin{fact} For locally finite $\cals \subset\mathbb{H}^2$ and $\bs\in\cals$, an edge $e$ of the Voronoi $2$-cell $V_{\bs}$ is non-centered with initial vertex $v$  if and only if the angle $\alpha$ at $v$, measured in $V_{\bs}$ between $e$ and the geodesic segment joining $v$ to $\bs$, is at least $\pi/2$.\end{fact}

This is because there is a right triangle with vertices at $\bs$ and $v$ and edges contained in $\gamma_{\bs\bt}$ and $\gamma_{\bs\bt}^{\perp}$, where $\gamma_{\bs\bt}$ is the geometric dual to $e$.  This triangle has angle equal to either $\alpha$ or $\pi-\alpha$ at $v$, depending on the case above; see Figure \ref{non-centered edge}.

\begin{figure}
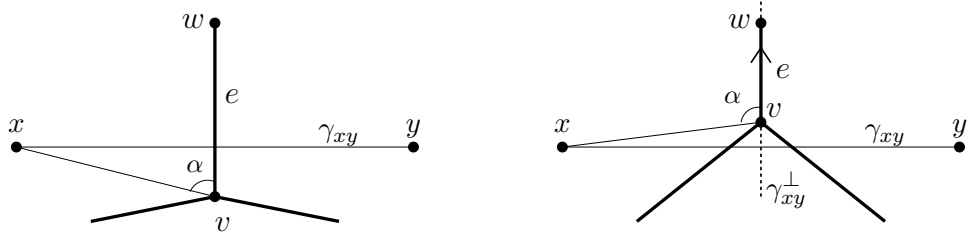
\caption{Centered and non-centered edges.}
\label{non-centered edge}
\end{figure}

If $e$ has another endpoint $w$ then since $\bs\in C_v\cap C_w$, the fact above and the hyperbolic law of cosines imply that the respective radii $J_v$ and $J_w$ of $v$ and $w$ (see Fact \ref{vertex radius}) satisfy:\begin{align}\label{radius relation}  \cosh J_w = \cosh\ell(e)\cosh J_v - \sinh\ell(e)\sinh J_v\cos\alpha \end{align}
Because $\cos\alpha \leq 0$ if $\alpha \geq \pi/2$, we have:

\begin{lemma}\label{increasing radius} Suppose $v$ is the initial and $w$ the terminal vertex of a non-centered edge, oriented as prescribed in Definition \ref{centered edge}, of the Voronoi tessellation of a locally finite set $\cals\subset\mathbb{H}^2$.  Then $J_{v} < J_{w}$.\end{lemma}

\begin{remark}  While every non-centered edge of the Voronoi tessellation has an initial vertex, note that not every such edge has a terminal vertex in $\mathbb{H}^2$, as in the middle case of Figure \ref{three point Vor}.  There all Voronoi edges are non-compact, and $V_y\cap V_z$ is non-centered.\end{remark}

Below we relate centeredness of edges of $V$ to that of geometric dual $2$-cells.

\begin{lemma}\label{vertex polygon centered} Let $v$ be a vertex of the Voronoi tessellation of a locally finite set $\cals\subset\mathbb{H}^2$.  Its geometric dual $C_v$ is non-centered if and only if $v$ is the initial vertex of a non-centered edge $e$ of $V$.  If this is so, the geometric dual $\gamma$ to $e$ is the unique longest edge of $C_v$, and $C_v\cup T(e,v)$ is a convex polygon, where $T(e,v)$ is the triangle determined by $v$ and $\partial\gamma$.\end{lemma}

\begin{proof}  Suppose first that $v$ is the initial vertex of a non-centered edge $e = V_{\bs}\cap V_{\bt}$, and let $\calh'$ be the half-space containing $e$ and bounded by the geodesic containing $\bs$ and $\bt$.  The circle $S$ with radius $J_v$ and center $v$ intersects $\partial\calh'$ in $\{\bs,\bt\}$.  If $\alpha$ is the angle at $v$ between $e$ and the geodesic arc to $\bs$, then by the Fact above, $\alpha \geq \pi/2$.  The hyperbolic law of cosines implies that $z\in S$ is in the interior of $\calh'$ if and only if the angle $\alpha'$ at $v$ between $e$ and the geodesic arc to $z$ is less than $\alpha$.  If $w$ is the other endpoint of $e$ then for such $z$:
$$\cosh d(z,w) = \cosh\ell(e)\cosh J_v - \sinh\ell(e)\sinh J_v\cos\alpha' $$
Since $\alpha' < \alpha$, comparing with (\ref{radius relation}) we find that $d(z,w) < J_w$, so the intersection of $S$ with the interior of $\calh$ is entirely contained in $B_{J_w}(w)$.  Therefore by Fact \ref{vertex radius} it contains no points of $\cals$.  Since all vertices of $C_v$ are on $S$, it follows that $C_v$ is contained in the half-plane $\calh$ opposite $\calh'$, and hence that $v\notin\mathit{int}\,C_v$.  Thus $C_v$ is non-centered (recall Definition \ref{cyclic poly}).

Assume now that $C_v$ is not centered and apply \OneSide.  This produces an edge $\gamma$ of $C_v$ and a half-space $\calh$ containing $C_v$ and bounded by the geodesic containing $\gamma$, such that $v$ is in the half-space $\calh'$ opposite $\calh$.  \OneSide\ further asserts that $P\cup T(e,v)$ is a convex polygon; also, $\gamma$ is the unique longest edge of $C_v$, by \LongestSide.  We claim that the other endpoint $w$ of the geometric dual $e$ to $\gamma$ is further from $\calh$ than $v$, and hence that $e$ is non-centered with initial vertex $v$.

If $w$ is closer to $\calh$ than $v$ (this includes the possibility $w\in\calh$), then $e$ intersects the geodesic joining $v$ to $x$ in an angle of $\alpha\leq\pi/2$.  $J_w$ again satisfies (\ref{radius relation}), and if $S$ is the circle of radius $J_v$ centered at $v$, the hyperbolic law of cosines again implies that $z\in S$ is in $\mathit{int}\,\calh$ if and only if the angle at $v$ between $e$ and the geodesic joining $v$ to $z$ is less than $\alpha$.  As in the previous case, this implies that the distance from the other vertices of $P_v$ to $w$ is less than $J_w$, contradicting Fact \ref{vertex radius}.  Therefore $w$ is further from $\calh$ than $v$.\end{proof}

If $v$ is the initial vertex of a non-centered Voronoi edge $e$, the fact that the geometric dual to $e$ is the \textit{unique} longest edge of $C_v$ immediately implies the following.

\begin{corollary}\label{one direction}  For a locally finite set $\cals\subset\mathbb{H}^2$, no vertex of the Voronoi tessellation of $\cals$ is the initial vertex of more than one non-centered edge.  \end{corollary}


Below, given a graph $G$ we will say that $\gamma = e_0\cup e_1\cup \hdots\cup e_{n-1}$ is an \textit{edge path} if $e_i$ is an edge of $G$ for each $i$ and $e_i\cap e_{i-1}\neq \emptyset$ for $i>0$.  An edge path $\gamma$ as above is \textit{reduced} if $e_i\neq e_{i-1}$ for each $i>0$, and $\gamma$ is \textit{closed} if $e_0\cap e_{n-1} \neq\emptyset$.

\begin{lemma}\label{tree components}  Each component $T$ of the non-centered Voronoi subgraph determined by locally finite $\cals\subset\mathbb{H}^2$ is a tree.  Each compact reduced edge path $\gamma$ of $T$ has a unique vertex $v_{\gamma}$ such that $J_{v_{\gamma}} > J_v$ for all vertices $v\neq v_{\gamma}$ of $\gamma$, and every edge of $\gamma$ points toward $v_{\gamma}$.\end{lemma}

\begin{proof}  Suppose that such a component $T$ admits closed, reduced edge paths, and let $\gamma = e_0\cup e_1\cup \hdots \cup e_{n-1}$ be shortest among them.  Orienting the $e_i$ as in Definition \ref{centered edge}, we may assume (after re-numbering if necessary) that $e_0$ points toward $e_0 \cap e_{n-1}$.  We claim that then $e_i$ points to $e_i\cap e_{i-1}$ for each $i>0$ as well.
Otherwise, for the minimal $i>0$ such that $e_i$ points toward $e_{i+1}$ it would follow that the vertex $e_i \cap e_{i-1}$ was the initial vertex of both $e_i$ and $e_{i-1}$, contradicting Corollary \ref{one direction}.

Let $v_0 = e_0\cap e_{n-1}\in V^{(0)}$, and for $i>1$ take $v_i = e_i\cap e_{i-1}$.  Applying Lemma \ref{increasing radius} to $e_i$ for each $i$, we find that $J_{v_i}> J_{v_{i+1}}$.  By induction this gives $J_{v_0} > J_{v_{n-1}}$; but since $e_{n-1}$ points to $v_{n-1}$ Lemma \ref{increasing radius} implies that $J_{v_{n-1}}$ must exceed $J_{v_0}$, a contradiction.  Thus $T$ contains no closed, reduced edge paths, so it is a tree.

Let $\gamma = e_0\cup \hdots \cup e_{n-1}$ be a reduced edge path, and let $v_{\gamma}$ be a vertex with $J_{v_{\gamma}}$ maximal.  Assume for now that $v_{\gamma}$ is on the boundary of $\gamma$, say the endpoint of $e_0$ not in $e_1$.  Lemma \ref{increasing radius} implies that $e_0$ points toward $v_{\gamma}$; thus if $i>0$ were minimal such that $e_i$ did not point toward $v_{\gamma}$ then $v_i = e_i\cap e_{i-1}$ would the initial endpoint of $e_i$ and $e_{i-1}$, contradicting Corollary \ref{one direction}.  It follows that each edge of $\gamma$ points toward $v_{\gamma}$, and by repeated application of Lemma \ref{increasing radius}, that $J_{v_T}>J_v$ for all vertices $v\neq v_{\gamma}$.  The case that $v_{\gamma}$ is in the interior of $\gamma$ follows by applying the argument above to the compact subpaths obtained by splitting $\gamma$ along $v_{\gamma}$.\end{proof}

\begin{definition}\label{root vertex}If a component $T$ of the non-centered Voronoi subgraph determined by locally finite $\cals\subset\mathbb{H}^2$ has a vertex $v_T$ with maximal radius, we call it the \textit{root vertex} of $T$.\end{definition}

If $v_T$ is a root vertex of $T$, Lemma \ref{tree components} immediately implies that $J_{v_T}> J_v$ for all $v\in T^{(0)}-\{v_T\}$.  In particular, $v_T$ is unique.

\begin{proposition}\label{to the root}  A component $T$ of the non-centered Voronoi subgraph determined by locally finite $\cals\subset\mathbb{H}^2$ has at most one non-compact edge.\begin{enumerate}
\item\label{root not centered}If one exists then its initial vertex is the root vertex $v_T$ of $T$, and $C_{v_T}$ is non-centered.
\item\label{root centered}If all edges are compact and there is a root vertex $v_T$, then $C_{v_T}$ is centered.
\end{enumerate}
For every non-root vertex $v$ of $T$, the geometric dual $C_v$ is non-centered.\end{proposition}

\begin{remark}The middle case of Figure \ref{three point Vor} is an example of the phenomenon (\ref{root not centered}) above.\end{remark}

\begin{proof}  A vertex $v$ of $T$ is contained in at least one non-centered Voronoi edge.  If it is the initial point of a non-centered edge $e$, then by Lemma \ref{vertex polygon centered}, that $C_v$ is non-centered. 

If $v$ is the initial vertex of a non-compact edge $e$ of $T$ then by Corollary \ref{one direction}, $v$ is the terminal vertex of every other edge of $T$ that contains it.  In particular, for any $w\in T^{(0)}-\{v\}$, each edge of the unique reduced edge path $\gamma$ in $T$ joining $v$ to $w$ is compact, so the edge of $\gamma$ that contains $v$ points towards it.  By Lemma \ref{tree components} every other edge of $\gamma$ points toward $v$ as well, and $J_{v}>J_w$.  Since $w$ was arbitrary, it follows that $v = v_T$ is the root vertex of $T$.  The uniqueness of the root vertex now implies that $e$ is the unique non-compact edge of $T$.

If every edge of $T$ is compact and $v_T$ is a root vertex, then by Lemma \ref{increasing radius} $v_T$ is the terminal point of every edge of $T$ that contains it.  Hence Lemma \ref{vertex polygon centered} implies that $C_{v_T}$ is centered.\end{proof}

\subsection{Introducing the centered dual}\label{centered dual props}  The basic idea is to think of the centered dual as dual to a coarsening of the Voronoi tessellation that has a ``large vertex'' for each component of the non-centered Voronoi subgraph.  In particular:

\begin{definition}\label{tree cells}  For a component $T$ of the non-centered Voronoi subgraph of a  locally finite set $\cals\subset\mathbb{H}^2$, we define the \textit{centered dual $2$-cell $C_T$ dual to $T$} as follows:\begin{enumerate}
\item  If $T$ has a non-compact edge $e_0$ with ideal endpoint $v_{\infty}$ (recall Definition \ref{sphere at infinity}), take:
$$  C_T = \Delta(e_0,v_{\infty}) \cup \left(\bigcup_{v\in T^{(0)}} C_v\right),  $$
where $\Delta(e_0,v_{\infty})$ is the convex hull in $\mathbb{H}^2$ of $v_{\infty}$ and the geometric dual to $e_0$.
\item  Otherwise, let $C_T = \bigcup_{v\in T^{(0)}} C_v$.\end{enumerate}

Define the \textit{boundary} $\partial C_T$ of $C_T$ as the union of geometric duals $\gamma$ to Voronoi edges $e$ that are not in $T$ but have an endpoint there, in case (2) above; or in case (1) the union of such $\gamma$ with the infinite edges of $\Delta(e_0,v_{\infty})$.  Let the \textit{interior} $\mathit{int}\,C_T$ of $C_T$ be $C_T-\partial C_T$.\end{definition}

We note that the two cases of the definition match those of Proposition \ref{to the root}.  

\begin{lemma}\label{no infinite triangle overlap} Suppose an edge $e_0$ of the Voronoi tessellation of a locally finite set $\cals\subset\mathbb{H}^2$ has an ideal endpoint $v_{\infty}$.  If $e_0$ has an endpoint $v_0\in\mathbb{H}^2$ then $C_{v_0}\cap \Delta(e_0,v_{\infty})=\gamma$, where $\gamma$ is the geometric dual to $e_0$ and $\Delta(e_0,v_{\infty})$ is as in Definition \ref{tree cells}.  For any other Voronoi vertex $v$, $C_v\cap \Delta(e_0,v_{\infty})\subset\partial \gamma$.\end{lemma}

\begin{proof}  Working in the upper half-plane model and moving $\cals$ by an isometry, we will take $e_0 = V_{\bs_0}\cap V_{\bt_0}$, where $\bs_0 = -x_0+i$ and $\bt_0 = x_0+i$ for some $x_0>0$, and $v_{\infty}=\infty$.  The horocycle through $\bs_0$ and $\bt_0$ with ideal point $v_{\infty}$ is $S_{\infty} = \mathbb{R}+i$, so by Lemma \ref{exclusive horocycle} every $z=x+iy\in\cals$ has $y\leq 1$.

For each Voronoi vertex $v$, the above implies that all vertices of $C_v$ are of the form $x+iy$ for $x\geq -x_0$, or all have $x\leq x_0$.  This is because for $x_1+iy_1$ with $x_1<-x_0$ and $x_2+iy_2$ with $x_2>x_0$ (and $y_i\leq 1$ for each $i$), every circle containing $x_1+iy_1$ and $x_2+iy_2$ encloses one of $\bs_0$ or $\bt_0$.  (Since metric circles of the upper half-plane model are Euclidean circles that lie in $\mathbb{H}^2$, though with different hyperbolic centers and radii, this is an entirely Euclidean calculation.)

If the vertices of $C_v$ are all on one side of the geodesic containing the geometric dual $\gamma$ to $C_v$, then so is their convex hull $C_v$ and one easily finds in this case that the lemma holds. This includes the case $v = v_0$, since then $\gamma$ is an edge of $C_v$.  We thus assume that $C_v$ has vertices on either side of this geodesic and moreover, applying the paragraph above, that every vertex $x+iy$ has $x\geq -x_0$.  We claim in this case that $C_v\cap\Delta(e_0,v_{\infty})$ is either $\bt_0 = x_0+i$ or empty.

Since $C_v$ is convex and closed it intersects the geodesic $\lambda$ containing $\gamma$ in a closed interval, with nonempty interior by the assumption above.  By our positioning of $\bs_0$ and $\bt_0$, $\lambda$ is the intersection with $\mathbb{H}^2$ of the Euclidean circle of radius $\sqrt{x_0^2+1}$ centered at the origin.  Because $\gamma$ is a geometric dual edge it does not cross an edge of $C_v$ nor contain any points of $\mathit{int}\,C_v$; thus by our assumption above, $\lambda\cap C_v$ is contained in the subinterval joining $x_0+i$ to $\sqrt{x_0^2+1}$.  The point of $\lambda\cap C_v$ closest to $x_0+i$ is contained in an edge $\delta$ of $C_v$, and the geodesic containing $\delta$ determines a half-plane that contains $C_v$ and separates it from $\Delta(e_0,v_{\infty})$.  This proves the claim and hence the lemma.\end{proof}

\begin{lemma}\label{vertex set}  For every $\bs\in\cals$ and each component $I$ of the intersection between $V_{\bs}$ and the non-centered Voronoi subgraph there is a boundary vertex $v$ of $I$ such that all edges of $I$ point away from $v$.\end{lemma}

\begin{proof}  Note that $I$ is a one-manifold.  If it is compact then its minimal-radius vertex $v$ has this property (by Lemma \ref{increasing radius} and Corollary \ref{one direction}, $v$ is in only one edge of $I$), so let us assume otherwise.  Let $e$ be an edge of $I$ and $v_0$ its initial vertex.  It follows from Lemma \ref{increasing radius} that the initial vertex $w$ of an edge pointing toward $v$ in $I$ has $J_w< J_{v_0}$, so $w$ is contained in the ball about $\bs$ of radius $J_{v_0}$ since $d(\bs,w) = J_w$.  By local finiteness of the Voronoi tessellation there are only finitely many such vertices, so again there is one with minimal radius.\end{proof}

\begin{lemma}\label{out on the co-nah}  For a component $T$ of the non-centered Voronoi subgraph of a locally finite set $\cals\subset\mathbb{H}^2$, $\partial C_T$ contains each $\bs\in\cals\cap C_T$, and every geometric dual edge $\gamma$ whose dual Voronoi edge is centered.\end{lemma}

\begin{proof}  For $\bs\in\cals\cap C_T$ the intersection $V_{\bs}\cap T$ is non-empty.  Let $I$ be a component, and let $v$ be the vertex supplied by Lemma \ref{vertex set}.  The geometric dual to the centered edge of $V_{\bs}$ containing $v$ lies in $\partial C_T$ and contains $\bs$.

Any geometric dual edge $\gamma$ contained in $C_T$ is by definition an edge of $C_v$ for some $v\in T^{(0)}$, so the geometric dual $e$ to $\gamma$ has $v$ as a vertex.  If $e$ is centered then it does not lie in $T$, so $\gamma\subset\partial C_T$ by definition.\end{proof}

\begin{lemma}\label{frontier vs boundary}  For a component $T$ of the non-centered Voronoi subgraph of a locally finite set $\cals\subset\mathbb{H}^2$, the interior of its geometric dual $C_T$ is connected, open in $\mathbb{H}^2$ and dense in $C_T$.  If $T^{(0)}$ is finite then $C_T$ is closed, and its topological frontier is contained in $\partial C_T$.\end{lemma}

\begin{remark}  In fact, the proof below will reveal that an edge $\gamma$ of $\partial C_T$ is entirely contained in the topological frontier of $C_T$ unless its geometric dual has both endpoints in $T$.\end{remark}

\begin{proof}  For any vertex $v$ of $T$ the geometric dual $C_v$ is a convex polyhedron and therefore closed in $\mathbb{H}^2$, with dense interior that is the complement of the union of its edges.  This also holds for $\Delta(e_0,v_{\infty})$, if applicable.  Since $\partial C_T$ is defined in \ref{tree cells} as a union of edges, the interior $C_T-\partial C_T$ of $C_T$ is therefore dense in $C_T$.

It is also connected: for points $\bx$ and $\by$ in the interior of $C_T$ there is a path $\rho$ in $T$ joining $v$ and $w$, where $\bx\in C_v$ and $\by\in C_w$ respectively.  For any edge $e$ of $\rho$, the geometric duals to the endpoints of $e$ intersect in the geometric dual $\gamma$ to $e$ by Lemma \ref{vertex polygon}.  Each point in the interior of $\gamma$ is in the interior of $C_T$, so one easily produces a path from $\bx$ to $\by$ in the interior of $C_T$ that is contained in the union of geometric duals to vertices of $\rho$.

For any $\bx\in\mathit{int}\,C_v\subset C_T$, $\bx$ is in the interior of $C_T$ and has an open neighborhood in $\mathbb{H}^2$ with this property.  If $\bx$ is in the interior of the geometric dual to an edge $e$ of $T$ then $\bx$ has an open neighborhood in $\mathbb{H}^2$ that is contained in $\mathit{int}\,C_v\cup\mathit{int}\,C_w$ and hence the interior of $C_T$, where $v$ and $w$ are the endpoints of $e$.  By Lemma \ref{out on the co-nah}, no point of $\cals$ is in the interior of $C_T$.  Therefore $C_T$ is the union of points already described, hence open in $\mathbb{H}^2$.

If $T^{(0)}$ is finite then $C_T$ is closed in $\mathbb{H}^2$, being a finite union of polygons.  Any convergent sequence in $C_T$ has an infinite subsequence in $C_v$ for some fixed $v\in T^{(0)}$, so if it converges outside the interior of $C_T$ the accumulation point lies in an edge of $C_v\cap\partial C_T$.\end{proof}

To establish finer properties of $C_T$ we will re-decompose it in a couple of different ways.  We first use the collection of triangles defined below, of which $\Delta(e_0,v_{\infty})$ from Definition \ref{tree cells} is one example.

\begin{definition}\label{edge-vertex tri dfn}  For an edge $e$ of the Voronoi tessellation of a locally finite set $\cals\subset\mathbb{H}^2$, and a vertex $v$ of $e$, let $\Delta(e,v)$ be the triangle in $\mathbb{H}^2$ with a vertex at $v$, and the geometric dual $\gamma$ to $e$ as an edge; i.e. $\Delta(e,v)$ is the convex hull in $\mathbb{H}^2$ of $v$ and $\gamma$.

If $e$ is non-compact and $v_{\infty}\in S_{\infty}$ is an ideal endpoint, again let $\Delta(e,v_{\infty})$ be the convex hull in $\mathbb{H}^2$ of $v_{\infty}$ and $\gamma$.\end{definition}

The endpoints of the geometric dual to $e$ are points $\bs,\bt\in\cals$ such that $e = V_{\bs}\cap V_{\bt}$.  Thus $\Delta(e,v)$ is isosceles: its edges joining $v$ to $\bs$ and $\bt$ each have length $J_v$.  If $v$ and $w$ are opposite endpoints of $e$, then $\Delta(e,v)$ and $\Delta(e,w)$ share the edge $\gamma$.  Whether their intersection is larger than this depends on whether $e$ is centered --- see Figure \ref{isosceles dichotomy}.  In particular:

\begin{figure}
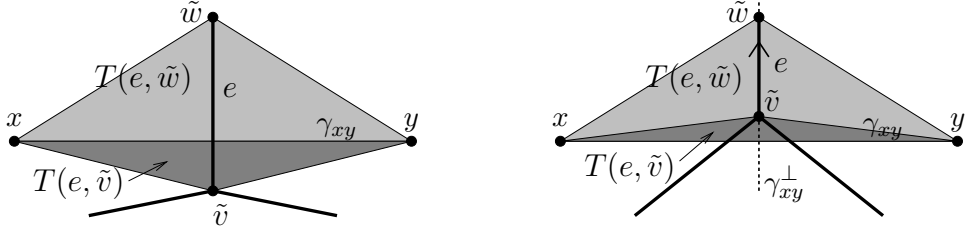
\caption{Triangles $\Delta(e,v)$ and $\Delta(e,w)$ when $e$ is centered (on the left) and not centered.}
\label{isosceles dichotomy}
\end{figure}

\begin{lemma}\label{edge-vertex tri}  If $e$ is a non-centered edge of the Voronoi tessellation of a locally finite set $\cals \subset\mathbb{H}^2$, with initial vertex $v$ and terminal vertex $w$, then $\Delta(e,v) \subset \Delta(e,w)$, and $\Delta(e,v)\cap \partial \Delta(e,w)$ is the geometric dual $\gamma$ to $e$.  The same holds if $e$ is non-compact and $w=v_{\infty}$ is its ideal endpoint.\end{lemma}

\begin{proof}  Since $e$ is non-centered it is contained on one side of the geodesic in $\mathbb{H}^2$ containing its geometric dual $\gamma$.  Since $v$ is the nearest point on $e$ to $\gamma$, for any point $w$ of $e-\{v\}$ the triangle $T_w$ determined by $w$ and the geometric dual $\gamma$ to $e$ has $v$ in its interior.  Hence by convexity $\Delta(e,v)\subset T_w$, and $\Delta(e,v)\cap\partial T_w$ is their common edge $\gamma$.

If $w$ above is the other endpoint of $e$ then $T_w = \Delta(e,w)$ and the conclusion of the lemma holds.  If $e$ is non-compact with ideal endpoint $w_{\infty}$ then $\Delta(e,w_{\infty}) \supset \bigcup_{w\in e} T_w$, and the conclusion again holds.\end{proof}

\begin{lemma}\label{vertex polygon decomp}  For a vertex $v$ of the Voronoi tessellation of a locally finite set $\cals\subset\mathbb{H}^2$:\begin{enumerate} 
\item  If $C_v$ is centered then $C_{v} = \bigcup\, \{ \Delta(e,v)\,|\,v\in e\}$.
\item  Otherwise, $C_v\cap \Delta(e_v,v) = \gamma_v$ and $C_{v}\cup \Delta(e_v,v) = \bigcup\,\{ \Delta(e',v)\,|\, e'\ne e_v, v\in e'\}$, where $e_v$ is the non-centered edge of $V$ with initial vertex $v$ and $\gamma_v$ is its geometric dual.\end{enumerate}
Let $\gamma$ and $\gamma'$ be the respective geometric duals to $e$ and $e'\ne e$ containing $v$.  In case (1), $\Delta(e,v)\cap \Delta(e',v) = \{v\}$ if $\gamma\cap\gamma'=\emptyset$, and otherwise is an edge joining $v$ to $\gamma\cap\gamma'$.  This holds in case (2) for $e,e'\neq e_v$.\end{lemma}

\begin{proof}  If the edges of $C_v$ are enumerated $\gamma_0,\hdots,\gamma_i$, for each $i$ the triangle $T_i$ described in the hypothesis of \IsoscelesDecomp\ is identical to $\Delta(e_i,v)$, where $e_i$ is the geometric dual to $\gamma_i$.  If $C_v$ is non-centered then Lemma \ref{vertex polygon centered} implies that $\gamma_v$ as defined above is its unique longest edge, so \LongestSide\ and \OneSide\ imply that $\Delta(e_v,v)\cap C_v = \gamma_v$.  The decompositions of $C_v$ and $C_v\cup \Delta(e,v)$ described above follow directly from \IsoscelesDecomp.\end{proof}

\begin{lemma}\label{re-decompose}  Let $C_T$ be a centered dual $2$-cell, dual to a component $T$ of the non-centered Voronoi subgraph determined by locally finite $\cals\subset\mathbb{H}^2$.  Then:\begin{enumerate}
\item  If $T$ has a noncompact edge $e_0$ with ideal endpoint $v_{\infty}$, then:
$$C_T = \Delta(e_0,v_{\infty})\cup\left(\bigcup_{v\in T^{(0)},e\ni v} \Delta(e,v)\right)$$
\item  Otherwise, $C_T = \bigcup_{v\in T^{(0)}, e\ni v} \Delta(e,v)$.\end{enumerate}\end{lemma}

\begin{proof}  Lemma \ref{vertex polygon decomp} and the definition of $C_T$ together directly imply that $C_T$ is contained in the union above.  For the other inclusion we first consider the case that $T$ has a root vertex $v_T$ and no non-compact edge.  In this case we claim, for any $v\in T^{(0)}$ and edge $e$ containing $v$, that $\Delta(e,v)\subset \bigcup_{w\in\gamma^{(0)}} C_w$, where $\gamma$ is the unique reduced edge path joining $v$ to $v_T$.  

The proof is by induction on the number of edges in $\gamma$.  The base case $v = v_T$ follows directly from Lemma \ref{vertex polygon decomp} (since $C_{v_T}$ is centered; see Proposition \ref{to the root}), so we assume $\gamma$ has $n\geq 1$ edges.  By Lemma \ref{vertex polygon decomp}, $\Delta(e,v)\subset C_v\cup \Delta(e_v,v)$, where $e_v$ is the non-centered edge with initial vertex $v$.  Lemma \ref{tree components} implies that the edge of $\gamma$ containing $v$ points toward $v_T$, so $v$ is its initial vertex; hence by Corollary \ref{one direction} this edge is $e_v$.  The claim follows upon applying the inductive hypothesis to the terminal vertex $w$ of $e_v$, which is connected to $v_T$ by the unique reduced edge path consisting of the edges of $\gamma$ other than $e_v$.

In the case that $T$ has a non-compact edge $e_0$, we change the claim to assert that $\Delta(e,v)\subset \Delta(e_0,v_{\infty})\cup \left(\bigcup_{w\in\gamma^{(0)}} C_w\right)$.  The proof is unchanged, except that in the base case Lemma \ref{vertex polygon decomp} gives $\Delta(e,v_T)\subset C_{v_T}\cup \Delta(e_0,v_T)$, and we appeal to Lemma \ref{edge-vertex tri} to show that this is contained in $C_{v_T}\cup \Delta(e_0,v_{\infty})$.

Finally suppose that $T$ has compact edges but no root vertex.  The following inductive procedure constructs a reduced edge ray $\gamma$ that begins at $v$ and has coherently oriented edges: $\gamma = \bigcup_{n=1}^{\infty} \gamma_n$, where $\gamma_1 = e_v$ and for $n>1$, $\gamma_n = \gamma_{n-1}\cup e_{v_{n-1}}$, where $v_{n-1}$ is the terminal vertex of $\gamma_{n-1}$.  Applying Lemmas \ref{edge-vertex tri} and \ref{vertex polygon decomp} in sequence gives 
$$\Delta(e_v,v)\subset \Delta(e_v,v_1) \subset C_{v_1}\cup \Delta(e_{v_1},v_1),$$
so by an inductive argument we have $\Delta(e_v,v)\subset \Delta(e_{v_n},v_n) \bigcup_{i=1}^n C_{v_i}$ for any $n$.

Since $\bigcup_{i=1}^{\infty} C_{v_i}$ is closed, it intersects $\Delta(e_v,v)$ in a closed subset.  This is also open, since for any $n$ the frontier of $\left(\bigcup_{i=1}^{\infty} C_{v_i}\right)\cap \Delta(e_v,v)$ in $\Delta(e_v,v)$ is contained in $C_{v_n}\cap \Delta(e_{v_n},v_n) = C_{v_n}\cap C_{v_{n+1}}$.  Hence $\Delta(e_v,v)\subset \bigcup_{n=1}^{\infty} C_{v_n}$, and we have proved the result.\end{proof}

That the $\Delta(e,v)$ overlap is a problem that we deal with by re-decomposing again.

\begin{definition}\label{obtuse quad}  For an edge $e$, with (possibly infinite) endpoints $v$ and $w$, of the Voronoi tessellation of a locally finite set $\cals\subset\mathbb{H}^2$, define $Q(e)$ by:\begin{itemize}
 \item $Q(e) = \Delta(e,v)\cup \Delta(e,w)$, if $e$ is centered; or
 \item  $Q(e) = \overline{\Delta(e,w)-\Delta(e,v)}$ if $e$ is non-centered and $v$ is its initial vertex.\end{itemize}
Here $\Delta(e,v)$ and $\Delta(e,w)$ are as in Definition \ref{edge-vertex tri}.\end{definition}

\begin{lemma}\label{no quad overlap}  For distinct edges $e$ and $f$ of the Voronoi tessellation of a locally finite set $\cals\subset\mathbb{H}^2$, with geometric duals $\gamma_e$ and $\gamma_f$, 
$$Q(e)\cap Q(f) = \left\{\begin{array}{ll}
  [\bx,v] & \mbox{if}\ e\cap f = v\ \mbox{and}\ \gamma_e\cap\gamma_f = \bx;\ \mbox{or}\\
  v = e\cap f & \mbox{if}\ e\cap f = v\ \mbox{and}\ \gamma_e\cap\gamma_f = \emptyset;\ \mbox{or}\\
  \bx & \mbox{if}\ e\cap f = \emptyset\ \mbox{and}\ \gamma_e\cap\gamma_f = \bx;\ \mbox{or}\\
  \emptyset & \mbox{otherwise}\end{array}\right.$$
Above, $[\bx,v]$ is the geodesic arc joining $\bx$ to $v$.  In particular, $Q(e)$ does not overlap $Q(f)$.\end{lemma}

\begin{proof}  For $\bx$ and $\by\in\cals$ such that $e = V_{\bx}\cap V_{\by}$, inspection of Figure \ref{isosceles dichotomy} reveals that $Q(e)$ is the union of the arcs joining $\bx$ to points of $e$, together with those joining $\by$ to points of $e$.  In particular, $Q(e)\subset V_{\bx}\cup V_{\by}$, and it intersects the boundary of this union only at the endpoints of $e$.  For $\bx'$ and $\by'$ such that $f = V_{\bx}\cap V_{\by}$, it is clear that $\{\bx,\by\}\ne\{\bx',\by'\}$, and if these sets are disjoint then $Q(e)$ can intersect $Q(f)$ only at a shared endpoint of $e$ and $f$.  Therefore suppose $\bx' = \bx$ (and hence $\by'\ne\by$).  It is now easy to see from the description above that $Q(e)$ intersects $Q(f)$ only at $\bx$, if $e$ and $f$ are not adjacent edges of $V_{\bx}$, or along the arc joining $\bx$ to $v$ if $e\cap f$ is a vertex $v$.\end{proof}

\begin{definition}\label{look out below}  For distinct vertices $v$ and $w$ of a component $T$ of the non-centered Voronoi subgraph determined by locally finite $\cals\subset\mathbb{H}^2$, say $w < v$ if $J_v$ is maximal among radii of vertices of the unique edge arc of $T$ joining $v$ to $w$.  (Recall Lemma \ref{tree components}.)\end{definition}

\begin{lemma}\label{looking down}  Let $T$ be a component of the non-centered Voronoi subgraph determined by locally finite $\cals\subset\mathbb{H}^2$, with edge set $\cale$.  For $v\in T^{(0)}$, if $C_v$ is non-centered then\begin{align}\label{like a boss}
 C_v\cup\Delta(e_v,v) \subset \bigcup_{w<v} \left(Q(e_w) \cup \bigcup \{\Delta(e,w)\,| w\in e, e\notin\cale\}\right)
 \end{align}
Here $e_v$ is the edge of $T$ with initial vertex $v$.  The analog holds if $C_v$ is centered, replacing $C_v\cup\Delta(e,v)$ with $C_v$ on the left side above.\end{lemma}

\begin{proof}  Below, let $v-1$ refer to the set of $w\in T^{(0)}$ such that $v$ is the terminal vertex of the edge $e_w$ of $T$ with initial vertex $w$.  Note that if $w\in v-1$ then $w<v$ (as in Definition \ref{look out below}).  By Lemma \ref{vertex polygon decomp}, $C_v\cup\Delta(e_v,v) = \bigcup \{\Delta(e,v)\,|\,v\in e, e\neq e_v\}$.  Each $e\in \cale-\{e_v\}$ with $v\in e$ is of the form $e_w$ for some $w\in v-1$.  We refine the union above using this fact, yielding:\begin{align}\label{like a bizoss}
  C_v\cup\Delta(e_v,v) & = \left(\bigcup_{w\in v-1} \Delta(e_w,v)\right)\cup \bigcup\ \{\Delta(e,v)\,|\,v\in e,e\notin\cale\} \notag\\
  & = \left(\bigcup_{w\in v-1} Q(e_w)\cup\Delta(e_w,w)\right)\cup \bigcup\ \{\Delta(e,v)\,|\,v\in e,e\notin\cale\}\end{align}
On the second line applies the definition of $Q(e)$.  It is clear from the second line that if (\ref{like a boss}) holds for all $w\in v-1$ then it holds for $v$ as well.

If $T$ is, say, finite, this is enough to prove the result by induction: let the \textit{height} of a vertex $v$ be maximal length of a chain $\{v_0,v_1,\hdots,v_n\}$ such that $v_i\in v_{i-1}-1$ for each $i$.  If $v$ has height zero then the union in parentheses above is empty and (\ref{like a boss}) holds for $v$.  For $v$ of height $n>0$, (\ref{like a boss}) then follows immediately from (\ref{like a bizoss}) and the inductive hypothesis.

For general $T$ we use (\ref{like a bizoss}) to show that the right-hand quantity of (\ref{like a boss}), which we call $U_v$, intersects $C_v\cup\Delta(e_v,v))$ in a set that is both closed and open there.   It is not hard to see that $U_v$ is closed in $\mathbb{H}^2$.  Lemma \ref{vertex polygon decomp} implies that the frontier of its intersection with $C_v\cup\Delta(e_v,v)$ is contained in the union of $Q(e_w)\cap\Delta(e_w,w)$, $w\in v-1$.  For such $w$, we claim that this union intersects the frontier of $U_v$ only in vertices of the geometric dual $\gamma_w$ to $e_w$, if at all.

Using (\ref{like a bizoss}), one can show that each point in the interior of an edge of $\Delta(e_v,v)$ other than $\gamma_v$ has a neighborhood in $\bigcup \{\Delta(e,v)\,|\,v\in e, e\neq e_v\}$ that is disjoint from $\bigcup \{\Delta(e_w,w)\,|\,w\in v-1\}$, and that the same holds true for $v$.  The analog of (\ref{like a bizoss}), applied to each $w\in v-1$, proves the claim.

Since the vertices of $\gamma_w$, $w\in v-1$, are vertices of $C_v$, it follows that $U_v\cap C_v$ is open and closed in $C_v$ and hence all of $C_v$.  The lemma follows for non-centered $v$.  The analogous assertion for centered $C_v$ follows in the same way, upon noting that by Lemma \ref{vertex polygon decomp}, the analog of (\ref{like a bizoss}) differs only by replacing $C_v\cup\Delta(e_v,v)$ by $C_v$ on the left-hand side.\end{proof}

\begin{proposition}\label{re-re-decompose}  Let $C_T$ be a centered dual $2$-cell, dual to a component $T$ of the non-centered Voronoi subgraph determined by locally finite $\cals\subset\mathbb{H}^2$.  Then:
$$ C_T = \left(\bigcup_{e\in\cale} Q(e)\right)\cup \bigcup_{v\in T^{(0)}} \{\Delta(e,v)\,|\,v\in e, e\notin\cale\}, $$
where $\cale$ is the edge set of $T$.  This union is non-overlapping.\end{proposition}

\begin{proof}  Lemma \ref{looking down} implies that the right-hand side contains the left (compare with Definition \ref{tree cells}).  For an edge $e$ of $T$ with terminal vertex $v$, $Q(e)\subset \Delta(e,v)$, so Proposition \ref{re-decompose} implies the other containment.  That the union is non-overlapping follows directly from Lemma \ref{no quad overlap}.\end{proof}

\begin{corollary}\label{cell a cell}  Let $C_T$ be a centered dual $2$-cell, dual to a component $T$ of the non-centered Voronoi subgraph determined by locally finite $\cals\subset\mathbb{H}^2$.  Then $T\subset \mathit{int}\, C_T$, and there is a homeomorphism $U\to \mathit{int}\,C_T$, where $U$ is the interior of the unit disk $\mathbb{D}^2$.  If $T^{(0)}$ is finite this extends to an onto map $\mathbb{D}^2\to C_T$; or $\mathbb{D}^2\to C_T\cup\{v_{\infty}\}$ if $T$ has a non-compact edge with ideal vertex $v_{\infty}$.\end{corollary}

\begin{proof}  For a compact Voronoi edge $e = V_{\bx}\cap V_{\by}$, $\bx,\by\in\cals$, $Q(e) - \{\bx,\by\}$ deformation retracts to $e$ along geodesic arcs from $\bx$ and $\by$.  If $e$ has an endpoint $v_{\infty}$ on $S_{\infty}$ then we must also exclude the arcs joining $\bx$ and $\by$ to $v_{\infty}$.  If the restriction to each arc is parametrized by arc length then the resulting homotopies are well-defined on overlaps $Q(e)\cap Q(f)$ for $e\neq f$ (recall Lemma \ref{no quad overlap}).  

For a component $T$ of the non-centered Voronoi subgraph with edge set $\cale$, $T\subset \bigcup_{e\in\cale} Q(e)\subset C_T$.  The deformation retracts described above determine a well-defined homotopy on a set that includes the interior of $C_T$, since all points of $\cals$ are in $\partial C_T$ by Lemma \ref{out on the co-nah}.  

Note also that if $e$ is centered then its geometric dual $\gamma$ is a flow line of the deformation retract $Q(e)-\{\bx,\by\}\to e$.  This thus restricts to a deformation retract $Q(e)-\gamma\to e - (e\cap\gamma)$.  Since a geometric dual edge $\gamma$ in $C_T$ is dual to a centered Voronoi edge if and only if $\gamma\subset\partial C_T$, Proposition \ref{re-re-decompose} and these observations  define a deformation retract from the interior of $C_T$ to 
$$T\cup\left(\bigcup \{[v,e\cap\gamma)\,|\, v\in e\cap T^{(0)}, e\notin\cale, \gamma\ \mbox{geometrically dual to}\ e\}\right)$$
Here $[v,e\cap\gamma)$ refers to the sub arc of $e$ joining $v$ to $e\cap\gamma$, but not including the latter point.  Each such arc deformation retracts to $v$, determining a further homotopy from the set above to $T$.

By the above, the interior of $C_T$ deformation-retracts to $T$ and hence is simply-connected.  The Riemann mapping theorem thus asserts the existence of a homeomorphism $f\co\mathit{int}\,\mathbb{D}^2\to\mathit{int}\,C_T$.  (Recall from Lemma \ref{frontier vs boundary} that $\mathit{int}\,C_T$ is a connected, open subset of $\mathbb{H}^2$, which we may take in $\mathbb{C}$ using the Poincar\'e disk model.)  

If $T^{(0)}$ is finite then by Lemma \ref{frontier vs boundary} $C_T$ is closed in $\mathbb{H}^2$ and the closure of its interior.  Either $C_T$ is compact, therefore also closed in $\mathbb{C}$, or it is compactified by the addition of the ideal point $v_{\infty}$ of the non-compact edge of $T$ (recall Proposition \ref{to the root}).  It is not hard to show that $\mathit{int}\,C_T$ is ``finitely connected along its boundary'' in the sense of \cite[\S IX.4.4]{Palka}, so the results there on conformal mapping imply that $f$ extends to a map from $\mathbb{D}^2$ to $C_T$ or $C_T\cup\{v_{\infty}\}$.\end{proof}

\begin{corollary}  Let $C_T$ be a centered dual $2$-cell, dual to a component $T$ of the non-centered Voronoi subgraph determined by locally finite $\cals\subset\mathbb{H}^2$.  Then:\begin{itemize}
  \item  If $T'\neq T$ is a component of the non-centered Voronoi subgraph then $C_T\cap C_{T'}\subset \partial C_T$.
  \item  For a Voronoi vertex $v$ outside the non-centered Voronoi subgraph, $C_T\cap C_v = \partial C_T\cap\partial C_v$.\end{itemize}
\end{corollary}

\begin{proof}  Let $T$ be a component of the non-centered Voronoi subgraph and $v$ a Voronoi vertex outside $T$.  Then for any $w\in T^{(0)}$, $C_v\cap C_w$ is either a vertex of each or the geometric dual to an edge $e$ with endpoints $v$ and $w$.  In the former case $C_v\cap C_w\subset\partial C_T$ by Lemma \ref{out on the co-nah}.  In the latter $e$ is centered since $v\notin T$, so again $C_v\cap C_w\subset\partial C_T$ (recall Definition \ref{tree cells}).  

The paragraph above implies the Lemma's second assertion if $T$ has no non-compact edges.  If there is a non-compact edge $e_0$ we appeal to Lemma \ref{no infinite triangle overlap}.  The first assertion follows as well, upon noting that for distinct non-compact edges $e_0$ and $f_0$ with respective ideal endpoints $v_{\infty}$ and $w_{\infty}$, $\Delta(e_0,v_{\infty})$ and $\Delta(f_0,w_{\infty})$ intersect along their boundaries.\end{proof}

\begin{definition}\label{centered dual}  The \textit{centered dual complex} of a locally finite set $\cals\subset\mathbb{H}^2$ has vertex set:
$$\cals\cup \{v_{\infty}\,|\,v_{\infty}\ \mbox{is the ideal endpoint of a non-centered Voronoi edge}\}$$ 
one-skeleton consisting of all:\begin{itemize}
  \item geometric dual edges whose dual Voronoi edges are centered; and
  \item rays $[\bx,v_{\infty}]\doteq [\bx,v_{\infty})\cup\{v_{\infty}\}$, where $v_{\infty}$ is the ideal endpoint of a non-centered Voronoi edge, $\bx$ is an endpoint of its geometric dual, and $[\bx,v_{\infty})$ is the geodesic ray from $\bx$ with ideal endpoint $v_{\infty}$;\end{itemize}
and two-skeleton consisting of all:\begin{itemize}
  \item  geometric dual two-cells $C_v$ where $v$ is a Voronoi vertex outside the non-centered Voronoi subgraph (in particular such a $C_v$ is centered, see Proposition \ref{vertex polygon centered}); and
  \item  cells $C_T$ or $C_T\cup \{v_{\infty}\}$ of Definition \ref{tree cells}, where $T$ is a component of the non-centered Voronoi subgraph and $v_{\infty}$ is an ideal endpoint of its noncompact edge (if applicable).\end{itemize}\end{definition}

The results of this section imply that if each component $T$ of the non-centered Voronoi subgraph has finite vertex set then the centered dual is indeed a cell decomposition of a subspace of $\mathbb{H}^2\cup S_{\infty}$, in the sense that each cell above is the image of a disk by a map that restricts on the interior to a homeomorphism.  By construction, its underlying topological space contains every geometric dual two-cell.

\section{Admissible spaces and area bounds: the compact case}\label{moduli}

The ultimate goal of this section is to prove Theorem \ref{main}, which bounds the area of a compact centered dual $2$-cell below given a uniform lower bound on its edge lengths.  There is no corresponding result for cyclic polygons (at least no good one) because of a non-monotonicity property of the area of those that are non-centered.  See Section \ref{cyclic geom} below, where we will collect useful results from \cite{DeB_cyclic_geom} on cyclic polygons.

The price we pay for passing from the geometric dual to the centered dual complex is that a two-cell is no longer determined by its collection of boundary edge lengths.  In Section \ref{admissible for compact} we will define an ``admissible space'' that parametrizes all possibilities for a centered dual two-cell with a given combinatorics and edge length collection, and prove some of its basic properties.  Finally in Section \ref{compact bounds} we will prove the theorem, by bounding values of the area functional on admissible spaces.

\subsection{The geometry of cyclic polygons}\label{cyclic geom}  J.M.~Schlenker originally proved existence and uniqueness of cyclic polygons with prescribed edge lengths satisfying a certain inequality \cite{Schlenker}.  Given this it is natural to parametrize the set of cyclic $n$-gons by $n$-tuples of real numbers that satisfy this inequality.  Here we will record some differential formulas that we proved in \cite{DeB_cyclic_geom}, treating the area and radius of cyclic polygons as functions on this space.

\begin{definition}\label{parameter spaces}For $d>0$ and $J\geq d/2$, let $A_d(J)\in(0,\pi]$ satisfy $\cos A_d(J) = 1-\frac{2\sinh^2(d/2)}{\sinh^2 J}$. Let each set below inherit the subspace topology from $\mathbb{R}^n$:\begin{align*}
  \mathcal{C}_n & = \left\{(d_0,\hdots,d_{n-1}) \in (\mathbb{R}^+)^n\,|\,\sum_{i=0}^{n-1}A_{d_i}(D/2) > 2\pi,\ \mbox{where}\ D=\max\{d_i\}_{i=0}^{n-1} \right\} \\
  \mathcal{BC}_n & = \left\{(d_0,\hdots,d_{n-1})\in (\mathbb{R}^+)^n\,|\,\sum_{i=0}^{n-1}A_{d_i}(D/2) = 2\pi,\mbox{where}\ D=\max\{d_i\}_{i=0}^{n-1} \right\} \\
   \mathcal{AC}_n & = \left\{(d_0,\hdots,d_{n-1}) \in (\mathbb{R}^+)^n\,|\, \sinh (d_i/2) < \sum_{j\neq i} \sinh(d_j/2) \ \mbox{for each}\ i \in \{0,\hdots,n-1\}\right\} \\
   \mathcal{HC}_n & = \left\{(d_0,\hdots,d_{n-1}) \in (\mathbb{R}^+)^n\,|\, \sinh (d_i/2) = \sum_{j\neq i} \sinh(d_j/2) \ \mbox{for some}\ i \in \{0,\hdots,n-1\}\right\}\end{align*}\end{definition}

$\calAC_n$ and $\calc_n\subset\calAC_n$ parametrize the sets of marked, oriented cyclic and centered $n$-gons, respectively.  $\calBC_n$ bounds $\calc_n$ in $\calAC_n$, and parametrizes a transitional class of cyclic $n$-gons: those with the center of the circumcircle in a side.  See Section 3 of \cite{DeB_cyclic_geom}.  $\calHC_n$ bounds $\calAC_n$ in $\mathbb{R}^n$ and parameterizes the \textit{horocyclic} $n$-gons, those with vertices on a horocycle; see \cite[\S 6]{DeB_cyclic_geom}.

The result below, on radii of cyclic $n$-gons (cf.~Definition \ref{cyclic poly}), combines parts of Lemmas 3.6, 4.5, and 4.7 of \cite{DeB_cyclic_geom}, and Proposition 4.1 there.

\begin{proposition}\label{omniradius}  For $n\geq 3$, $J\co\calAC_n\to\mathbb{R}^+$ is symmetric and $C^1$, where $J(d_0,\hdots,d_{n-1})$ is the radius of the cyclic $n$-gon determined by $(d_0,\hdots,d_{n-1})$.  Its partial derivatives satisfy:\begin{itemize}
  \item  For $(d_0,\hdots,d_{n-1})\in\calc_n$, $0<\frac{\partial J}{\partial d_i} (d_0,\hdots,d_{n-1}) <\frac{1}{2}$ for each $i$.
  \item  For $(d_0,\hdots,d_{n-1})\in\calBC_n$,\begin{align*}
  \frac{\partial J}{\partial d_i} (d_0,\hdots,d_{n-1}) & = \left\{\begin{array}{ll} 
    \frac{1}{2} &\mbox{ if $d_i$ is maximal} \\ 
    0 &\mbox{ otherwise} \end{array}\right.\end{align*}
  \item  For $(d_0,\hdots,d_{n-1})\in\calAC_n-(\calc_n\cup\calBC_n)$,
  \begin{align*}  \frac{\partial J}{\partial d_i} (d_0,\hdots,d_{n-1}) & \left\{\begin{array}{ll} 
    > \frac{1}{2} &\mbox{ if $d_i$ is maximal} \\ 
    < 0 &\mbox{ otherwise} \end{array}\right.\end{align*}\end{itemize}
 Finally, $J\to\infty$ as $(d_0,\hdots,d_{n-1})\in\calAC_n$ approaches $\calHC_n$.\end{proposition}

The area of cyclic $n$-gons (a.k.a the ``radius-$0$ defect''; see \cite[Definition 5.1]{DeB_cyclic_geom}) satisfies a nice differential formula.

\begin{proposition}[\cite{DeB_cyclic_geom}, Proposition 5.5]\label{omniarea}  For $n\geq 3$, $D_0\co\calAC_n\to\mathbb{R}^+$ is symmetric and $C^1$, where $D_0(d_0\hdots,d_{n-1})$ is the area of the cyclic $n$-gon determined by $(d_0,\hdots,d_{n-1})$.  At $(d_0,\hdots,d_{n-1})\in\calc_n\cup\calBC_n$, for each $i\in\{0,\hdots,n-1\}$ we have:
$$ \frac{\partial}{\partial d_i} D_0(d_0,\hdots,d_{n-1}) = \sqrt{\frac{1}{\cosh^2(d_i/2)}-\frac{1}{\cosh^2 J(d_0,\hdots,d_{n-1})}}  $$
This also holds for $(d_0,\hdots,d_{n-1})\in\calAC_n - \calc_n$ such that $d_i$ is not maximal.  Otherwise:
$$ \frac{\partial}{\partial d_i} D_0(d_0,\hdots,d_{n-1}) = -\sqrt{\frac{1}{\cosh^2(d_i/2)}-\frac{1}{\cosh^2 J(d_0,\hdots,d_{n-1})}} $$\end{proposition}

\begin{corollary}[\cite{DeB_cyclic_geom}, Corollary 5.8]\label{monotonicity}  For $n\geq 3$ and $(d_0,\hdots,d_{n-1}), (d_0',\hdots,d_{n-1}')\in\calc_n\cup\calBC_n$, if $d_i < d'_{\sigma(i)}$ for each $i$ and some fixed permutation $\sigma$ then $D_0(d_0,\hdots,d_{n-1})< D_0(d_0',\hdots,d_{n-1}')$.\end{corollary}

\begin{remark}  Since the radius and defect functions are symmetric, we will not worry much in practice about the particular cyclic order on edge or vertex sets of geometric dual polygons.\end{remark}

$\calBC_n$ and $\calHC_n$ are smoothly parametrized, disjoint, codimension-one submanifolds of $(\mathbb{R}^+)^n$.  The result below combines \BCn, \HCn, and \VerticalRay.

\begin{proposition}\label{smooth BC}  For each $n\geq 3$, there are smooth, positive-valued functions $b_0$ and $h_0$ on $(\mathbb{R}^+)^{n-1}$ such that $(b_0(\bd),\bd)\in\calBC_n$ and $(h_0(\bd),\bd)\in\calHC_n$ for any $\bd = (d_1,\hdots,d_{n-1})\in(\mathbb{R}^+)^{n-1}$.  Moreover, under the $\mathbb{Z}_n$-action on $(\mathbb{R}^+)^n$ by cyclic permutation of entries we have:\begin{align*}
  & \calBC_n = \bigsqcup_{\gamma\in\mathbb{Z}_n} \gamma\cdot\left\{(b_0(\bd),\bd)\,|\,\bd\in(\mathbb{R}^+)^{n-1}\right\} &
  & \calHC_n = \bigsqcup_{\gamma\in\mathbb{Z}_n} \gamma\cdot\left\{(h_0(\bd),\bd)\,|\,\bd\in(\mathbb{R}^+)^{n-1}\right\}\end{align*}
The functions $b_0$ and $h_0$ have the following additional properties:\begin{enumerate}
\item\label{the proper order}  For any $(d_1,\hdots,d_{n-1})\in(\mathbb{R}^+)^{n-1}$, $\max\{d_i\}_{i=1}^{n-1} < b_0(d_1,\hdots,d_{n-1}) < h_0(d_1,\hdots,d_{n-1})$.
\item\label{vertical ray} If $\bd = (d_0,\hdots,d_{n-1})\in(\mathbb{R}^+)^n$ has largest entry $d_0$ then
$$ \bd\in\calc_n \Leftrightarrow d_0 < b_0(d_1,\hdots,d_{n-1})\quad\mbox{and}\quad\bd\in\calAC_n \Leftrightarrow d_0 < h_0(d_1,\hdots,d_{n-1}) $$
\item\label{parameter monotonicity}  If $0< d_i \leq d_i'$ for each $i\in\{1,\hdots,n-1\}$ then $b_0(d_1,\hdots,d_{n-1}) \leq b_0(d_1',\hdots,d_{n-1}')$, and the same holds for $h_0$.
\end{enumerate}\end{proposition}
  
It is not convenient to attempt explicit formulas for the functions $D_0$ and $J$, but it is useful to know explicit values in a few cases.

\begin{lemma}\label{symmetric cyclic}  For $n\geq 3$ and $d>0$, $\bd\doteq (d,\hdots,d)\in(\mathbb{R}^+)^n$ is in $\calc_n$, and:\begin{align*}
 & D_0(\bd) = (n-2)\pi - 2n\sin^{-1}\left(\frac{\cos(\pi/n)}{\cosh(d/2)}\right) &
 & \sinh J(\bd) = \frac{\sinh(d/2)}{\sin (\pi/n)} \end{align*}
For $n\geq 3$ and $(d,\hdots,d)\in(\mathbb{R}^+)^{n-1}$, let $B_0= b_0(d,\hdots,d)$, where $b_0\co(\mathbb{R}^+)^{n-1}\to\mathbb{R}^+$.  Then:
$$ D_0(B_0,d,\hdots,d) = (n-2)\pi - (2n-2)\sin^{-1}\left(\frac{\cos(\pi/(2n-2))}{\cosh(d/2)}\right)$$
\end{lemma}

\begin{proof}  It follows directly from the definitions that $\bd\in\calc_n$.  A cyclic $n$-gon with all sides of length $d$ is divided into $n$ isometric isosceles triangles by arcs joining its vertices to its center $v$.  Each of the resulting triangles thus has angle $2\pi/n$ at $v$, with each edge containing $v$ of length $J(\bd)$, and opposite edge of length $d$.  Applying the hyperbolic law of cosines and rearranging yields:
$$ \sinh J = \sqrt{\frac{\cosh d-1}{1-\sin(2\pi/n)}} = \frac{\sinh(d/2)}{\sin(\pi/n)} $$
The latter equation above follows from half-angle formulas.  Applying the hyperbolic law of sines to find the angle $\alpha$ between the sides of length $J(\bd)$ and $d$ yields:
$$ \sin \alpha = \frac{\sinh d}{\sinh(d/2)}\cdot\frac{\sin(2\pi/n)}{\sin(\pi/n)} = \frac{\cos(\pi/n)}{\cosh(d/2)} $$
The latter equation again follows from half-angle formulas.  $D_0(\bd)$ is $n$ times the area of one of these triangles, the angle defect $\pi - 2\pi/n - 2\alpha$.  This gives the first formula above.

We finally note that a cyclic $n$-gon $P$ with side length collection $(B_0,d,\hdots,d)$ has its center at the midpoint of its longest side.  Therefore the union of $P$ with its reflection $\overline{P}$ across the longest side is a cyclic $(2n-2)$-gon with the same circumcircle, all sides of length $d$, and 
$$\mathrm{Area}(P\cup\overline{P}) = \mathrm{Area}(P)+\mathrm{Area}(\overline{P}) = 2\mathrm{Area}(P)$$
The second area formula therefore follows from the first.\end{proof}

\subsection{Admissible spaces}\label{admissible for compact}  By the results of Section \ref{cyclic geom}, a centered dual $2$-cell $C_T$ is determined by the edge lengths of its constituent geometric dual polygons, together with their combinatorial arrangement.  The latter data are captured by the corresponding component $T$ of the non-centered Voronoi subgraph.  Recall from Definition \ref{tree cells} that the boundary of $C_T$ is the union of geometric duals to edges in the frontier of $T$.  Our goal here is to understand the geometry of $C_T$ using only its combinatorial structure and edge length data.

It is not hard to see that this is insufficient to determine $C_T$, but in this section we will describe properties of an \textit{admissible space} that, given this data collection, parametrizes all possibilities for such a cell.  We focus on the case that $C_T$ is compact, so $T$ is as well; in particular, all its edges are compact and $T^{(0)}$ is finite.

\theoremstyle{definition}
\newtheorem*{hypothesis}{Blanket hypothesis}  
\begin{hypothesis}  In this subsection we take $V$ to be a graph, perhaps with some non-compact edges, such that each vertex $v$ has valence $n_v$ satisfying $3\leq n_v < \infty$.  $T\subset V$ is a compact, rooted subtree with root vertex $v_T$, edge set $\cale$, and frontier $\calf$ in $V$.  The sole exception to this rule is Lemma \ref{admissible Voronoi}, where explicit hypotheses are given.\end{hypothesis}

Here the \textit{frontier} of $T$ in $V$ is the set of pairs $(e,v)$ such that $e$ is an edge of $V$ but not of $T$, and $v$ is a vertex in $e \cap T$.  We may refer to ``an edge'' of the frontier of $T$, without reference to its vertices, but note that such an edge may contribute up to two elements to $\calf$.

\begin{definition}\label{partial order}  Partially order $T^{(0)}$ by setting $v < v_T$ for each $v\in T^{(0)} - \{v_T\}$, and $w < v$ if the edge arc in $T$ joining $w \in T^{(0)}- \{v_T,v\}$ to $v_T$ runs through $v$.  Let $v-1$ be the set of $w<v$ joined to it by an edge, and say $v$ is \textit{minimal} if $v-1=\emptyset$.  For $v\in T^{(0)}-\{v_T\}$, let $e_v$ be the initial edge of the arc in $T$ joining $v$ to $v_T$, and say ``$e\to v$'' for each edge $e\neq e_v$ of $V$ containing $v$.\end{definition}

\begin{definition}\label{admissible criteria}  Let $(\mathbb{R}^+)^{\calf}$ be the set of tuples of positive real numbers indexed by the elements of $\calf$, and define $(\mathbb{R}^+)^{\cale}$ analogously.  For any elements $\bd_{\cale}=(d_e\,|\,e\in\cale)\in (\mathbb{R}^+)^{\mathcal{E}}$ and $\bd_{\calf}\in(\mathbb{R}^+)^{\calf}$, let $\bd = (\bd_{\cale},\bd_{\calf})$ and $P_{v}(\bd) = (d_{e_0},\hdots,d_{e_{n-1}})$ for $v\in T^{(0)}$, where the edges of $V$ containing $v$ are cyclically ordered as $e_0,\hdots,e_{n-1}$.  We say the \textit{admissible set} $\mathit{Ad}(\bd_{\calf})$ determined by $\bd_{\calf}$ is the collection of $\bd \in(\mathbb{R}^+)^{\cale}\times\{\bd_{\calf}\}$ such that: \begin{enumerate}
\item\label{not centered} For $v\in T^{(0)}-\{v_T\}$, $P_v(\bd) \in \mathcal{AC}_{n_v} - \calc_{n_v}$ has largest entry $d_{e_v}$. 
\item\label{centered} $P_{v_T}(\bd) \in \calc_{n_T}$, where we refer by $n_T$ to the valence $n_{v_T}$ of $v_T$ in $V$.
\item\label{radius order} $J(P_{v}(\bd)) > J(P_w(\bd))$ for each $w\in v- 1$, where $J(P_v(\bd))$ and $J(P_w(\bd))$ are the respective radii of $P_v(\bd)$ and  $P_w(\bd)$.  \end{enumerate}  
\end{definition}

\begin{remark}\label{empty admissible}  $\mathit{Ad}(\bd_{\calf})$ is empty for certain $\bd_{\calf}\in(\mathbb{R}^+)^{\calf}$.  For instance if $T$ has one edge and vertices of valence $3$ in $V$ then for any $d>0$ and $\bd_{\calf} = (d,d,d,d)$, $\mathit{Ad}(\bd_{\calf}) = \emptyset$.\end{remark}

\begin{remark}  If $T = \{v_T\}$ then $\mathit{Ad}(\bd_{\calf})$ is either empty or $\{\bd_{\calf}\}$ for any $\bd_{\calf}\in(\mathbb{R}^+)^{\calf}$; the latter if and only if $P_{v_T}(\bd_{\calf})\in\calc_{n_T}$.  (Note that the valence $n_T$ of $v_T$ in $V$ is $|\calf|$.)\end{remark}

\begin{definition}\label{tree defect}  Fix $\bd_{\calf} = (d_e\,|\,e\in\calf)\in (\mathbb{R}^+)^{\calf}$ such that $\mathit{Ad}(\bd_{\calf})\neq \emptyset$.  For each $\bd\in\mathit{Ad}(\bd_{\calf})$ and $R\geq 0$, 
define:
$$ D_T(\bd) = \sum_{v\in T^{(0)}} D_0(P_v(\bd)), $$
where $P_v(\bd)$ is as in Definition \ref{admissible criteria} and $D_0(P)$ is as in Proposition \ref{omniarea}.  \end{definition}

\begin{lemma}\label{admissible Voronoi}  Let $C_T$ be a compact centered dual two-cell, dual to a component $T$ of the non-centered Voronoi subgraph determined by locally finite $\cals\subset\mathbb{H}^2$.  Let $\mathcal{E}$ be the edge set of $T$ and $\mathcal{F}$ its frontier in the Voronoi graph $V$, and for each edge $e$ of $V$ that intersects $T$ let $d_e$ be the length of the geometric dual to $e$.  Then $\bd = (d_e\,|\,e\in\cale) \in \mathit{Ad}(\bd_{\calf})$, where $\bd_{\calf}=(d_e\,|\, (e,v)\in\calf\mbox{ for some }v\in T^{(0)})$, and $C_T$ has area $D_T(\bd)$.  \end{lemma}

\begin{proof}  Since $C_T$ is compact so is $T$; in particular, $T^{(0)}$ is finite.  It follows that $T$  has a root vertex $v_T$ (recall Definition \ref{root vertex}).  By Proposition \ref{to the root}(\ref{root centered}), the geometric dual $C_{v_T}$ to $v_T$ is centered, and $C_v$ is non-centered for each $v\in T^{(0)}-\{v_T\}$.  It further follows from Lemma \ref{tree components} that for each $v\in T^{(0)}-\{v_T\}$, $e_v$ as defined in \ref{partial order} is the edge of $T$ with initial vertex $v$.

If $e_0,\hdots,e_{n-1}$ is the cyclically ordered collection of edges of $V$ containing $v\in T^{(0)}$, then appealing to Definition \ref{parameter spaces} we find that $C_v$ is represented by $(d_{e_0},\hdots,d_{e_{n-1}})\in\calAC_n$. Criterion (\ref{centered}) from Definition \ref{admissible criteria} follows, as does (\ref{not centered}) upon observing that for each $v\in T^{(0)}-\{v_T\}$, $C_v$ has longest side length $d_{e_v}$ by Lemma \ref{vertex polygon centered}.  

For $v\in T^{(0)}$ and $w \in v -1$, since $w$ is the initial vertex of $e_w$ and $v$ is its terminal vertex Lemma \ref{increasing radius} yields $J_v > J_w$.   Definition \ref{admissible criteria}(\ref{radius order}) follows, upon noting that $J_v = J(P_v)$ and $J_w=J(P_w)$, where the left-hand quantities are described in Lemma \ref{vertex polygon} and the others in Proposition \ref{omniradius}.  

That $C_T$ has area $D_T(\bd)$ is a direct consequence of Definitions \ref{tree cells} and \ref{tree defect}, since the union $C_T = \bigcup_{v\in T^{(0)}} C_v$ is non-overlapping and $D_0(P_v(\bd))$ is the area of $C_v$ for each $v\in T^{(0)}$.  \end{proof}

It is not hard to see that $\mathit{Ad}(\bd_{\calf})$ is generally not closed in $(\mathbb{R}^+)^{\cale}\times\{\bd_{\calf}\}$.  We will find it convenient to enlarge it slightly, since our main goal here is to compute minima of $D_T$.

\begin{definition}\label{bigger admissible}  For $\bd_{\calf} = (d_e\,|\,e\in\calf)\in(\mathbb{R}^+)^{\calf}$ let $\overline{\mathit{Ad}}(\bd_{\calf})$ consist of those $\bd = (\bd_{\cale},\bd_{\calf})$, for $\bd_{\cale}\in (\mathbb{R}^+)^{\cale}$, such that:\begin{enumerate}
\item\label{closure not centered} For $v\in T^{(0)}-\{v_T\}$, $P_v(\bd) \in\mathcal{AC}_{n_v} - \calc_{n_v}$ has largest entry $d_{e_v}$. 
\item\label{closure centered} $P_{v_T}(\bd) \in \calc_{n_T}\cup \calBC_{n_T}$, where we refer by $n_T$ to the valence $n_{v_T}$ of $v_T$ in $V$.
\item\label{closure radius order} $J(P_{v}(\bd)) \geq J(P_w(\bd))$ for each $w\in v- 1$, where $J(P_v(\bd))$ and $J(P_w(\bd))$ are the respective radii of $P_v(\bd)$ and $P_w(\bd)$.\end{enumerate}
\end{definition}

It is immediate from its definition that $\overline{\mathit{Ad}}(\bd_{\calf})$ contains $\mathit{Ad}(\bd_{\calf})$.  We will show in Lemma \ref{admissible closure} that it is compact and in particular closed, so it contains the closure of $\mathit{Ad}(\bd_{\calf})$.  However:

\begin{remark}  If $T$ has one edge and vertices of valence $3$ in $V$ then for any $d>0$ and $\bd_{\calf} = (d,d,d,d)$, $\overline{\mathit{Ad}}(\bd_{\calf}) = \{(B,\bd_{\calf})\}$ where $B = b_0(d,d)$.\end{remark}

With Remark \ref{empty admissible} this shows that the inclusion $\overline{\mathit{Ad}(\bd_{\calf})}\subset\overline{\mathit{Ad}}(\bd_{\calf})$ is proper in some cases.  

\begin{remark}\label{one vertex}  If $T = \{v_T\}$ then $\overline{\mathit{Ad}}(\bd_{\calf})$ is either empty or $\{\bd_{\calf}\}$ for any $\bd_{\calf}\in(\mathbb{R}^+)^{\calf}$; the latter if and only if $P_{v_T}(\bd_{\calf})\in \calc_{n_T}\cup\calBC_{n_T}$.  (Here $n_T = |\calf|$ is the valence of $v_T$ in $V$.)\end{remark}

\begin{remark}\label{length increasing}  Definition \ref{bigger admissible}(\ref{closure not centered}) implies that for any $v\in T^{(0)}-\{v_T\}$, $d_{e_v} > d_e$ for each $e\to v$ (cf.~Proposition \ref{smooth BC}(\ref{vertical ray})).  It follows that $d_{e_v} > d_e$ for each $e\to w$ such that $w < v$.  In particular, for some fixed $d>0$ if $d_e \geq d$ for all $e\in\calf$ then $d_e> d$ for all $e\in\cale$.\end{remark}

The Lemma below expands on Remark \ref{length increasing}.

\begin{lemma}\label{precompact}  Collections $\{b_e\co (\mathbb{R}^+)^{\calf}\to\mathbb{R}^+\}_{e\in\cale}$ and $\{h_e\co (\mathbb{R}^+)^{\calf}\to\mathbb{R}^+\}_{e\in\cale}$ are determined by the following properties: for $\bd_{\calf}\in(\mathbb{R}^+)^{\calf}$ and $\bd_{\cale}\in(\mathbb{R}^+)^{\cale}$, with $\bd=(\bd_{\cale},\bd_{\calf})$,\begin{itemize}
  \item  $P_v(\bd)$ is in $\calBC_{n_v}$, with largest entry $d_{e_v}$, for all $v\in T^{(0)}-\{v_T\}$ if and only if $d_e = b_e(\bd_{\calf})$ for each $e\in\cale$.
  \item  $P_v(\bd)$ is in $\calHC_{n_v}$, with largest entry $d_{e_v}$, for all $v\in T^{(0)}-\{v_T\}$ if and only if $d_e = h_e(\bd_{\calf})$ for each $e\in\cale$.\end{itemize}
For $e\in\cale$, the functions $b_e$ and $h_e$ have the following properties:\begin{enumerate}
\item\label{b increasing in T}  For $\bd_{\calf}$ and $v\in T^{(0)}-\{v_T\}$, $b_{e_v}(\bd_{\calf}) > \max \{b_e(\bd_{\calf})\,|\, e\to v\in\cale\}\cup\{d_e\,|\,e\to v\in\calf\}$.
\item\label{admissible bounded}  If $\bd=(\bd_{\cale},\bd_{\calf}) \in \overline{\mathit{Ad}}(\bd_{\calf})$ then for each $e\in\mathcal{E}$, $b_e(\bd_{\calf}) \leq d_e < h_e(\bd_{\calf})$.
\item\label{b increasing}  If $d_e' \geq d_e$ for each $e\in\calf$ then $b_e(\bd_{\calf}')\geq b_e(\bd_{\calf})$ for each $e\in\cale$, where $\bd_{\calf}' = (d_e')_{e\in\calf}$.\end{enumerate}\end{lemma}

\begin{proof}  We construct by induction, the key point being that for $v\in T^{(0)}-\{v_T\}$, $b_{e_v}(\bd_{\calf})$ is determined by $\bd_{\calf}$ and $\{b_{e_w}(\bd_{\calf})\,|\,w<v\}$, and similarly for $h_{e_v}(\bd_{\calf})$.  Fix $\bd_{\calf}\in(\mathbb{R}^+)^{\calf}$.

Suppose first that $v\in T^{(0)}$ is minimal, so each $e\to v$ is in $\mathcal{F}$.  Cyclically enumerate the edges of $V$ containing $v$ as $e_0,\hdots,e_{n-1}$ so that $e_0 = e_v$, and for each $i >0$ let $d_i = d_{e_i}$.  Let $b_{e_v}(\bd_{\calf})= b_0(d_1,\hdots,d_{n-1})$ and $h_{e_v}(\bd_{\calf}) = h_0(d_1,\hdots,d_{n-1})$, for $b_0$ and $h_0$ taking $(\mathbb{R}^+)^{n-1}$ to $\mathbb{R}^+$ as in Proposition \ref{smooth BC}.  That result implies that $b_{e_v}(\bd_{\calf})$ is the unique real number with the property that $(b_{e_v}(\bd_{\calf}),d_1,\hdots,d_{n-1})$ is in $\calBC_n$ and has its largest entry first; and it implies the analog for $h_{e_v}(\bd_{\calf})$ and $\calHC_n$. 

Note also that if $\bd\in \overline{\mathit{Ad}}(\bd_{\calf})$ then Definition \ref{bigger admissible} (\ref{closure not centered}) implies that $P_v(\bd) = (d_{e_v},d_1,\hdots,d_{n-1})\in \calAC_n - \calc_n$ has largest entry $d_{e_v}$, so $b_{e_v}(\bd_{\calf})\leq d_{e_v} < h_{e_v}(\bd_{\calf})$ by Proposition \ref{smooth BC}(\ref{vertical ray}).  This implies property (\ref{admissible bounded}) above for $b_{e_v}$.  Property (\ref{b increasing in T}) and property (\ref{b increasing}) above also follow from Proposition \ref{smooth BC}, respectively using assertions (\ref{the proper order}) and (\ref{parameter monotonicity}) there.

Now fix $v\in T^{(0)}-\{v_T\}$ non-minimal, and assume that $b_{e_w}(\bd_{\calf})$ and $h_{e_w}(\bd_{\calf})$ are defined for each $w<v$, uniquely such that for $\bd_{\calf}\in(\mathbb{R}^+)^{\calf}$ and $\bd_{\cale}\in(\mathbb{R}^+)^{\cale}$, with $\bd=(\bd_{\cale},\bd_{\calf})$:\begin{itemize}
\item $P_w(\bd)\in\calBC_{n_w}$ with largest entry $d_{e_w}$, for all $w<v$ if and only if $d_{e_w} = b_{e_w}(\bd_{\calf})$ for each $w < v$.
\item $P_w(\bd)\in\calHC_{n_w}$, with largest entry $d_{e_w}$, for all $w<v$ if and only if $d_{e_w} = h_{e_w}(\bd_{\calf})$ for each $w < v$.
\item Property (\ref{admissible bounded}) holds for each $b_{e_w}$ and $h_{e_w}$, and (\ref{b increasing in T}) and (\ref{b increasing}) hold for each $b_{e_w}$, $w<v$.\end{itemize}
Cyclically order the edges containing $v$ as $e_0,\hdots,e_{n-1}$ so that $e_0 = e_v$, and for $i>0$ take:\begin{align*}
  & d_i = \left\{\begin{array}{ll} d_{e_i} & e_i\in\calf \\ b_{e_i}(\bd_{\calf}) & e_i\in\cale \end{array}\right. &
  & d_i' = \left\{\begin{array}{ll} d_{e_i} & e_i\in\calf \\ h_{e_i}(\bd_{\calf}) & e_i\in\cale\end{array}\right.\end{align*}
Proposition \ref{smooth BC} again implies that $b_{e_v}(\bd_{\calf})\doteq b_0(d_1,\hdots,d_{n-1})$ is unique among $b>\max\{d_i\}$ such that $(b_{e_v}(\bd_{\calf}),d_1,\hdots,d_{n-1})\in\calBC_n$; and $h_{e_v}(\bd_{\calf})\doteq h_0(d_1',\hdots,d_{n-1}')$ is unique among $h> \max\{d_i'\}$ such that $(h_{e_v}(\bd_{\calf}),d_1',\hdots,d_{n-1}')\in\calHC_n$.

Now let $\bd\in\overline{\mathit{Ad}}(\bd_{\calf})$.   Since property (\ref{admissible bounded}) holds by hypothesis for each $e_i\in\cale$, $d_{e_i}\geq d_i$ for such $i$ (and otherwise $d_{e_i}=d_i$ by construction).  Thus Proposition \ref{smooth BC}(\ref{parameter monotonicity}) implies that $b_{e_v}(\bd_{\calf})\leq b_0(d_{e_1},\hdots,d_{e_{n-1}})$, and Definition \ref{bigger admissible}(\ref{not centered}) and Proposition \ref{smooth BC}(\ref{vertical ray}) imply that $b_0(d_{e_1},\hdots,d_{e_{n-1}}) < d_{e_v}$.  Analogously, $h_{e_v}(\bd_{\calf}) > h_0(d_{e_1},\hdots,d_{e_{n-1}}) > d_{e_v}$.  To summarize:
$$ b_{e_v}(\bd_{\calf}) \leq b_0(d_{e_1},\hdots,d_{e_{n-1}}) \leq d_{e_v} < h_0(d_{e_1},\hdots,d_{e_{n-1}}) < h_{e_v}(\bd_{\calf})$$
This proves property (\ref{admissible bounded}) for $e_v$.  Property (\ref{b increasing in T}) and (\ref{b increasing}) again follow from the corresponding assertions of Proposition \ref{smooth BC}, along with the inductive hypothesis.

The lemma now follows by induction.  (Recall in particular that there is a unique $e_v$ for each $v\in T^{(0)}-\{v_T\}$, and that $\cale$ is the set of all such $e_v$.)\end{proof}

\begin{remark}\label{outside-in}  For any given tree $T$ with frontier $\calf$, the proof of Lemma \ref{precompact} is easily adapted (using formulas from \cite{DeB_cyclic_geom}) to produce a recursive algorithm that takes $\bd_{\calf}\in(\mathbb{R}^+)^{\calf}$ and computes the values $b_e(\bd_{\calf})$ or $h_e(\bd_{\calf})$ from the ``outside in.''\end{remark}

\begin{lemma}\label{admissible closure}  For any $\bd_{\calf}\in(\mathbb{R}^+)^{\calf}$, $\overline{\mathit{Ad}}(\bd_{\calf})$ is compact.\end{lemma}

\begin{proof}  This is vacuous if $\overline{\mathit{Ad}}(\bd_{\calf})$ is empty, so fix $\bd_{\calf}$ such that $\overline{\mathit{Ad}}(\bd_{\calf})\neq \emptyset$.  It is enough to show that $\overline{\mathit{Ad}}(\bd_{\calf})$ is closed in $\mathbb{R}^{\cale}\times\{\bd_{\calf}\}$, since Lemma \ref{precompact}(\ref{admissible bounded}) implies it is bounded.  Note also that Lemma \ref{precompact}(\ref{b increasing in T}) implies for fixed $\bd_{\calf}\in(\mathbb{R}^+)^{\calf}$ that if $d = \min\{d_e\,|\,e\in\calf\}$ then $\overline{\mathit{Ad}}(\bd_{\calf})\subset [d,\infty)^{\cale}\times\{\bd_{\calf}\}$.

It is clear from the definition of $P_v(\bd)\in\calAC_{n_v}$ that it varies continuously with $\bd$ (to this point, recall that $\calAC_n$ takes the subspace topology from $\mathbb{R}^n$; see \cite{DeB_cyclic_geom} above Lemma 3.2).  Since $\calc_{n_T}\cup\calBC_{n_T}$ is closed in $(\mathbb{R}^+)^{n_T}$ (see Lemma 3.3 of \cite{DeB_cyclic_geom}), and no sequence in $\overline{\mathit{Ad}}(\bd_{\calf})$ approaches the frontier of $(\mathbb{R}^+)^{\cale}\times\{\bd_{\calf}\}$ (see above), condition (\ref{closure centered}) is preserved under any limit of points in $\overline{\mathit{Ad}}(\bd_{\calf})$.  By Proposition \ref{omniradius}, $J(P_v(\bd))$ varies continuously with $\bd$ on $\overline{\mathit{Ad}}(\bd_{\calf})$ for each $v\in T^{(0)}$, so (\ref{closure radius order}) is also preserved by such a limit.

Since $\calAC_n$ is open in $(\mathbb{R}^+)^n$ it is \textit{a priori} possible that (\ref{closure not centered}) is not preserved; ie, for some sequence $\{\bd_{i}\}\subset\overline{\mathit{Ad}}(\bd_{\calf})$ limiting to $\bd\in (\mathbb{R}^+)^{\cale}\times\{\bd_{\calf}\}$ that there exists $v\in T^{(0)}-\{v_T\}$ such that $P_v(\bd) \in\calHC_{n_v}$, where $v$ has valence $n_v$ in $V$.  For such $\{\bd_{i}\}\to\bd$, let $v$ be a closest vertex to $v_T$ such that $P_v(\bd)\in\calHC_{n_v}$.  In particular $P_w(\bd) \in \mathcal{AC}_{n_w}$ for the endpoint $w$ of $e_v$ (note that $P_{v_T}(\bd)\in\calAC_{n_T}$ by preservation of (\ref{closure centered})).  Proposition \ref{omniradius} implies on the one hand that $J(P_w(\bd_i)) \to J(P_w(\bd))$, since $P_w(\bd_i)\to P_w(\bd)$, and on the other that $J(P_v(\bd_i)) \to \infty$, since $P_v(\bd_i)\to P_v(\bd)\in\calHC_n$.  But then for some $\bd_i$ the inequality of Definition \ref{bigger admissible}(\ref{radius order}) fails, a contradiction.  Therefore (\ref{closure not centered}) is preserved under taking limits, and $\overline{\mathit{Ad}}(\bd_{\calf})$ is closed.  \end{proof}

\begin{lemma}\label{continuous defect}  Fix $\bd_{\calf} = (d_e\,|\,e\in\calf)\in (\mathbb{R}^+)^{\calf}$ such that $\overline{\mathit{Ad}}(\bd_{\calf})\neq \emptyset$.  Then $D_T(\bd)$ is continuous on $\overline{\mathit{Ad}}(\bd_{\calf})$ and attains a minimum there.  \end{lemma}

\begin{proof}  Since $P\mapsto D_0(P)$ is continuous on $\mathcal{AC}_n$ (by Proposition \ref{omniarea}), and  $P_v(\bd) \in \mathcal{AC}_n$ for each $\bd \in \overline{\mathit{Ad}}(\bd_{\calf})$, $D_T(\bd)$ is continuous on $\overline{\mathit{Ad}}(\bd_{\calf})$.  Since this is compact by Lemma \ref{admissible closure}, $D_T(\bd)$ attains a minimum on it.\end{proof}

Finally, we observe that $D_T$ attains a minimum only at one of a short list of special locations.

\begin{proposition}\label{minimum at}  For $\bd_{\calf}\in(\mathbb{R}^+)^{\calf}$ with $\overline{\mathit{Ad}}(\bd_{\calf})\neq\emptyset$, at a minimum point $\bd = (\bd_{\cale},\bd_{\calf})$ for $D_T(\bd)$ one of the following holds:\begin{enumerate}
\item\label{really not centered} $P_v(\bd)\in\calBC_{n_v}$ for each $v\in T^{(0)}-\{v_T\}$, where $v$ has valence $n_v$ in $V$;
\item\label{really centered} $P_{v_T}(\bd)\in\calBC_{n_T}$, where $v_T$ has valence $n_T$ in $V$; or
\item\label{strict radius order} $J(P_v(\bd)) = J(P_w(\bd))$ for some $v\in T^{(0)}$ and $w\in v-1$.
\end{enumerate}\end{proposition}

\begin{proof}  Suppose that none of the criteria above hold at $\bd$, and fix $v\in T^{(0)}-\{v_T\}$ such that $P_v(\bd)\notin\calBC_{n_v}$, where $v$ has valence $n_v$ in $V$.  We will show that for the edge $e_v$ of $T$ with initial point $v$, reducing $d_{e_v}$ while keeping the remaining entries of $\bd$ constant produces new points of $\overline{\mathit{Ad}}(\bd_{\calf})$ at which $D_T$ takes smaller values.

We first observe that $D_T(\bd)$ is reduced by reducing $d_{e_v}$.  Changing only the length of $e_v$ affects only $P_v(\bd)$ and $P_{v'}(\bd)$, where $v'$ is its terminal vertex.  $P_v(\bd)\in\calAC_{n_v}-\calc_{n_v}$ has largest side length $d_{e_v}$, but either $P_{v'}(\bd)$ has  $d_{e_{v'}}$ as its largest or, if $v'=v_T$, $P_{v'}(\bd) \in\calc_{n_T}\cup\calBC_{n_T}$.  Thus Proposition \ref{omniarea} implies that $\frac{\partial}{\partial d_{e_v}}D_T=\frac{\partial}{\partial d_{e_v}}\left[D_0(P_{v'}(\bd))+D_0(P_v(\bd))\right]$ is:\begin{align}\label{defect deriv diff}
   \sqrt{\frac{1}{\cosh^2(d_{e_v}/2)} - \frac{1}{\cosh^2 J(P_{v'}(\bd))}} - \sqrt{\frac{1}{\cosh^2(d_{e_v}/2)} - \frac{1}{\cosh^2 J(P_v(\bd))}} \end{align}
Since condition (\ref{strict radius order}) above does not hold by hypothesis, but condition (\ref{closure radius order}) of Definition \ref{bigger admissible} does, $J(P_{v'}(\bd)) > J(P_v(\bd))$.  Therefore the quantity above is positive.  Since this is also $\frac{\partial}{\partial d_{e_v}} D_T(\bd)$, reducing $d_{e_v}$ reduces the value of $D_T$ near $\bd$.

Our hypothesis and Definition \ref{bigger admissible}(\ref{closure not centered}) imply that $P_v(\bd)$ is in the open subset $\calAC_{n_v} - (\calc_{n_v}\cup\calBC_{n_v})$ of $\mathbb{R}^{n_v}$.  Thus small deformations of $d_{e_v}$ keep it there.  It is possible that $v' = v_T$; if so then because (\ref{really centered}) above does not hold but the corresponding criterion from Definition \ref{bigger admissible} does, $P_{v'}(\bd)$ is in the open set $\calc_n$.  It follows again in this case that small deformations of $d_{e_v}$ keep it here.

If $v'\neq v_T$ then it is possible that $P_{v'}(\bd)\in\calBC_{n'}$, where $v'$ has valence $n'$ in $V$.  However in this case, direct appeal to the definition of $\calc_{n'}$ in \CenteredSpace\ shows that reducing $d_{e_v}$ keeps $P_{v'}$ in $\calAC_{n'} -\calc_{n'}$.  (Recall in particular that $d_{e_v}$ is not the largest side length of $P_{v'}(\bd)$ by Definition \ref{bigger admissible}(\ref{closure not centered}), and $A_d(D/2)$ increases with $d$ for any fixed $D\geq d$.)  

Criterion (\ref{closure radius order}) from Definition \ref{bigger admissible} holds for any small deformation of $\bd$.  This is because $J(P_v(\bd)) > J(P_w(\bd))$ for all $v\in T^{(0)}$ and $w\in v-1$, as we pointed out above, and $J(P_v(\bd))$ varies continuously with $\bd$.  Thus by Definition \ref{bigger admissible}, any small deformation of $\bd$ that reduces $d_{e_v}$ and leaves every other entry constant lies in $\overline{\mathit{Ad}}(\bd_{\calf})$.
\end{proof}

\subsection{A lower bound on area}\label{compact bounds}  Here we will prove Theorem \ref{main}, by induction on the number of vertices of the component $T$ of the non-centered Voronoi subgraph contained in a centered dual $2$-cell $C_T$.  For the purposes of this argument we will give each Voronoi vertex $v$ that is not contained in the non-centered Voronoi subgraph honorary status as a component of it.  Thus $T=\{v\}$ is a tree with no edges, and the case $C_T = C_v$ is the base case of the induction.  Note that $C_v$ is centered for such $v$, by Lemma \ref{vertex polygon centered}.

\begin{proposition}\label{centered main}  For $d > 0$ and $(d,\hdots,d)\in(\mathbb{R}^+)^n$, where $n\geq 4$, $D_0(d,\hdots,d) \geq (n-2) D_0(B_0,d,d)$, where $B_0 = b_0(d,d)$ for $b_0\co\mathbb{R}^2\to\mathbb{R}$ as in Proposition \ref{smooth BC}.\end{proposition}

\begin{proof} For $d^n \doteq (d,\hdots,d)\in(\mathbb{R}^+)^n$, Lemma \ref{symmetric cyclic} implies that
$$ D_0(d,\hdots,d) = (n-2)\pi - 2n\sin^{-1}\left(\frac{\cos(\pi/n)}{\cosh(d/2)}\right); $$
and also that
$$ (n-2) D_0(B_0,d,d) = (n-2)\left[\pi - 4\sin^{-1}\left(\frac{1/\sqrt{2}}{\cosh(d/2)}\right)\right] $$
Fixing $d>0$, for $n\geq 4$ we define $f_d(n) = D_0(d^n) - (n-2)D_0(B_0,d,d)$:
$$ f_d(n) = 2\left[2(n-2)\sin^{-1}\left(\frac{1/\sqrt{2}}{\cosh(d/2)}\right) - n\sin^{-1}\left(\frac{\cos(\pi/n)}{\cosh(d/2)}\right)\right] $$
Note that $f_d(4) = 0$ for each $d$.  This reflects the fact that a cyclic quadrilateral with all sides of length $d$ is the union of two triangles in $\calBC_3$, each with two sides of length $d$.  Now allowing $n$ to take arbitrary values in $[4,\infty)$, we record the first and second derivatives of $f_d$:\begin{align*}
 f_d'(n) & = 2\left[ 2\sin^{-1}\left(\frac{1/\sqrt{2}}{\cosh(d/2)}\right) - \sin^{-1}\left(\frac{\cos(\pi/n)}{\cosh(d/2)}\right) - \frac{\pi}{n}\frac{\sin(\pi/n)}{\sqrt{\cosh^2(d/2)-\cos^2(\pi/n)}} \right] \\
 f_d''(n) & = 2\ \frac{\pi^2}{n^3}\frac{\cos(\pi/n)\sinh^2(d/2)}{(\cosh^2(d/2) - \cos^2(\pi/n))^{3/2}} \end{align*}
From this we find in particular that for any fixed $d$, $f$ is concave up.  For fixed $d$ we have
$$ f_d'(4) = 2\left[\sin^{-1}\left(\frac{1/\sqrt{2}}{\cosh(d/2)}\right) - \frac{\pi/4}{\sqrt{2\cosh^2(d/2)-1}}\right] $$
We claim that the quantity above is positive for each $d>0$.  To this end, we compute:
$$ \frac{\partial}{\partial d} \left(f_d'(4)\right) = \frac{\sinh(d/2)}{\sqrt{2\cosh^2(d/2)-1}}\left[\frac{(\pi/2)\cosh(d/2)}{2\cosh^2(d/2)-1} - \frac{1}{\cosh(d/2)} \right] $$
The quantity in brackets above is positive at $d=0$, and one easily finds the unique $d_0 > 0$ at which it vanishes.  Thus $\frac{\partial}{\partial d} f_d'(4)$ is positive on $(0,d_0)$ and negative on $(d_0,\infty)$.  It is not hard to see that $f_0'(4) = 0=\lim_{d\to\infty} f_d'(4)$, so $f_d'(4)$ is positive on $(0,\infty)$.

For any fixed $d>0$, we showed above that $f_d'(4) > 0$ and that $f_d'(n)$ increases in $n$ on $(4,\infty)$, so in particular $f_d'(n) > 0$ for all $n$.  Therefore $f_d(n) > 0$ for every $n\in(4,\infty)$.  The result follows.\end{proof}

We will address the case that $T$ has more than one vertex using Proposition \ref{minimum at}.  Of the three conditions there, (\ref{really not centered}) and (\ref{really centered}) may each be addressed directly in different ways.  We will use the lemma below to reduce complexity in case (\ref{strict radius order}) and thereby apply induction.

\begin{lemma}\label{two polys}  Suppose $P_0 = (c_0,\hdots,c_m)\in\calAC_m - \calc_m$ and $P_1 = (d_0,\hdots,d_n)\in\calAC_n$ satisfy:\begin{itemize}
  \item  $J(P_0) = J(P_1)$;
  \item  $c_0 = d_0$ is the longest side length of $P_0$; and
  \item  either $P_1\in\calc_n\cup\calBC_n$ or $d_0$ is not the longest side length of $P_1\in\calAC_n - \calc_n$.\end{itemize}
Then $P \doteq (c_1,\hdots,c_m,d_1,\hdots,d_n)$ is in $\calAC_{m+n-2}$, and in $\calc_{m+n-2}\cup\calBC_{m+n-2}$ if and only if $P_1\in\calc_n\cup\calBC_n$.  Also, $D_0(P_0) + D_0(P_1) = D_0(P)$.\end{lemma}

\begin{proof}  Let $C$ be a circle in $\mathbb{H}^2$ containing a fixed representative of $P_0$ (with radius $J(P_0)$ in particular), and let $v$ be the center of $C$.  By \LongestSide\ and \OneSide, $v$ is in the opposite half-space determined by the side $\gamma_0$ with length $d_0$ from the one containing $\calp_0$.

Let a representative of $P_1$ be arranged in $\mathbb{H}^2$ so that its intersection with $P_0$ is $\gamma_0$ (corresponding in $P_1$ to the side length $d_0$).  The circle $C'$ in $\mathbb{H}^2$ containing the vertices of $P_1$ has the same radius as $C$ and intersects $C$ in at least the endpoints of $\gamma_0$.  Furthermore, \OneSide\ implies that the center $v'$ of $C'$ is in the half-space determined by $\gamma_0$ that contains $P_1$, whence $v' = v$ and $C' = C$.

It follows from the above that $P = P_0\cup P_1$ is itself a cyclic polygon inscribed in $C$ (cf.~\PtsToPoly), and from the definition of the defect that $D_0(P) = D_0(P_0)\cup D_0(P_1)$.  Furthermore, by construction $v\in P$ if and only if $v\in P_1$, so the remainder of the lemma follows from the final claim of \InCenteredClosure.\end{proof}

\begin{hypothesis}  Until the proof of Theorem \ref{main}, each definition and result below uses the following hypothesis: $V$ is a finite graph with vertices of valence at least $3$, $T$ is a rooted subtree of $V$ with root vertex $v_T$, $\cale$ is the edge set of $T$, and $\calf$ is its frontier in $V$.\end{hypothesis}

\begin{definition}\label{smushanedge}  For an edge $e$ of $T$, let $p_e\co V\to V_e$ be the quotient map that identifies $e$ to a point, and let $T_e = p_e(T)$.
\end{definition}

\begin{remark}  It is easy to see that $T_e$ is a tree, and that $p_e$ maps $\cale - \{e\}$ and $\calf$ bijectively to the edge set $\cale_e$ and frontier $\calf_e$ of $T_e$, respectively.  In particular, if the endpoints $v$ and $w$ have valences $n_v$ and $n_w$ in $V$, respectively, then $p_e(v) = p_e(w)$ has valence $n_v+n_w - 2$.\end{remark}

\begin{lemma}\label{reduce edges} For $\bd_{\calf}\in(\mathbb{R}^+)^{\calf}$, suppose that some $\bd = (\bd_{\cale},\bd_{\calf})\in\overline{\mathit{Ad}}(\bd_{\calf})$ satisfies condition (\ref{strict radius order}) of Proposition \ref{minimum at}.  Then $D_T(\bd) = D_{T_{f}}(\bd_{f})$ for $T_f$ as in Definition \ref{smushanedge}, where:\begin{itemize} 
  \item  $f\in\cale$ has initial vertex $v$ and terminal vertex $w$ such that $J(P_v(\bd)) = J(P_w(\bd))$.  
  \item  $\bd_{f} = (\bd_{\cale_{f}},\bd_{\calf_{f}}) \in\overline{\mathit{Ad}}(\bd_{\calf_{f}})$ for $\bd_{\cale_{f}} = (d_{p_{f}(e)}\,|\,e\in\mathcal{E}-\{f\})$ and $\bd_{\calf_{f}} = (d_{p_{f}(e)}\,|\,e\in\mathcal{F})$, where $d_{p_{f}(e)} = d_e$ for each $e$ in $\mathcal{E}-\{f\}$ or occuring in $\mathcal{F}$.
\end{itemize}\end{lemma}

This follows directly from Lemma \ref{two polys}.  The result below will allow us to address condition (\ref{really centered}) of Proposition \ref{minimum at}, by varying $\bd_{\calf}$ and tracking the changes in $\overline{\mathit{Ad}}(\bd_{\calf})$.

\begin{lemma}\label{sad space}  The set $\mathit{SAd}_T$, consisting of $\bd_{\calf}\in(\mathbb{R}^+)^{\calf}$ such that $\overline{\mathit{Ad}}(\bd_{\calf})\neq\emptyset$, is closed in $(\mathbb{R}^+)^{\calf}$, and the function
$$ \bd_{\calf} \mapsto \min\{ D_T(\bd) \,|\,\bd \in\overline{\mathit{Ad}}(\bd_{\calf}) \}$$
is lower-semicontinuous on $\mathit{SAd}_T$.\end{lemma}

\begin{proof}  Suppose $\bd_{\calf}^{(i)}$ is a sequence in $\mathit{SAd}_T$ converging to $\bd_{\calf}\in(\mathbb{R}^+)^{\calf}$, and for each $i$ let $\bd^{(i)}\in\overline{\mathit{Ad}}(\bd_{\calf}^{(i)})$ be a point at which $D_T(\bd)$ attains its minimum over all $\bd\in\overline{\mathit{Ad}}(\bd_{\calf}^{(i)})$.  We claim that there exist $0 < d < D$ such that $\bd^{(i)}\subset[d,D]^{\cale}\times[d,D]^{\calf}$ for each $i$.

Since $\{\bd_{\calf}^{(i)}\}$ converges, $D_0=\sup \{d_e^{(i)}\,|\,e\in\calf, i\in\mathbb{N}\}$ is finite, and since no side of a polygon has length greater than the sum of the lengths of the other sides one finds easily that $d_e^{(i)} \leq D\doteq q^pD_0$ for each $e\in\cale$ and $i\in\mathbb{N}$, where $p = |\cale|$ and $q+1$ is the maximal valence in $V$ of a vertex of $T$.  Furthermore, $d = \inf\{d_e^{(i)}\,|\,e\in\calf, i\in\mathbb{N}\} > 0$, and Lemma \ref{precompact}(\ref{b increasing in T}) implies that $d_e^{(i)} \geq d$ for each $e\in\cale$ and $i\in\mathbb{N}$.  The claim follows.

The claim implies that a subsequence of $\{\bd^{(i)}\}$ converges to $\bd=(\bd_{\cale},\bd_{\calf})$ for some $\bd_{\cale}\in(\mathbb{R}^+)^{\cale}$ and $\bd_{\calf}$ as fixed at the beginning.  We next claim that $\bd\in\overline{\mathit{Ad}}(\bd_{\calf})$.  The proof of this claim is essentially identical to the proof of Lemma \ref{admissible closure}, replacing $\bd_i$ there with $\bd^{(i)}$.  

The second claim immediately implies that $\mathit{SAd}_T$ is closed.  Furthermore, for $\bd=(\bd_{\cale},\bd_{\calf})$, $D_T(\bd)$ is an upper bound for the value at $\bd_{\calf}$ of the minimum function in question here.  \DefectDerivative\ implies $D_T(\bd) = \lim_{i\to\infty} D_T(\bd_i)$, and lower semicontinuity follows.\end{proof}

\begin{proposition}\label{general lower}  For fixed $d>0$ and $\bd_{\calf}\in(\mathbb{R}^+)^{\calf}$ with all entries at least $d$,
$$ \min\{D_T(\bd)\,|\,\bd\in\overline{\mathit{Ad}}(\bd_{\calf})\} \geq \left(|\calf|-2\right)D_0(B_0,d,d),$$
where $B_0 = b_0(d,d)$ for $b_0$ as described in Proposition \ref{smooth BC}.\end{proposition}

\begin{proof}  We prove this by induction on the number of vertices of $T$.  The base case $T=\{v_T\}$ follows directly from Proposition \ref{centered main} and Corollary \ref{monotonicity} (cf.~Remark \ref{one vertex}), so let us suppose that $T$ has $n > 1$ vertices and the result holds for all trees with fewer than $n$ vertices.  Since $T$ is a tree, $|\cale| = n-1$ (by Euler characteristic), so in particular $T$ has at least one edge.

Fix an arbitrary $D > d$ and consider the intersection of $\mathit{SAd}_T$ (as in Lemma \ref{sad space}) with $[d,D]^{\calf}$.  Lemma \ref{sad space} implies that this set is compact, so the function
$$ \bd_{\calf} \mapsto \min\{ D_T(\bd) \,|\,\bd \in\overline{\mathit{Ad}}(\bd_{\calf}) \}$$
attains a minimum on it by lower-semicontinuity.  Fix $\bd_{\calf}$ at which the minimum occurs, and let $\bd$ be a minimum point for $D_T$ on $\overline{\mathit{Ad}}(\bd_{\calf})$.  This satisfies at least one of the three conditions described in Proposition \ref{minimum at}.  We claim that because of our choice of $\bd_{\calf}$, $\bd$ in fact satisfies at least one of conditions (\ref{really not centered}) or (\ref{strict radius order}).

Assume by way of contradiction that $\bd$ satisfies only (\ref{really centered}), and cyclically enumerate the edges of $V$ containing $v_T$ as $e_0,\hdots,e_{n_T-1}$ so that $d_{e_0}$ is maximal.  By the hypothesis and Remark \ref{length increasing}, $d_{e_i} \geq d$ for each $i>0$.  Since $P_{v_T}(\bd)\in\calBC_{n_T}$ we have $d_{e_0} = b_0(d_{e_1},\hdots,d_{e_{n_T-1}})$ for $b_0$ as in Proposition \ref{smooth BC}.  That result implies in particular that $d_{e_0} > \max\{d_{e_i}\}_{i=1}^{n_T} \geq d$, so $d_{e_0}$ can be reduced slightly while preserving the inequality $d_{e_0} > d$.

We note that it is not the case that $e_0\in\cale$: since $P_{v_T}(\bd)\in\calBC_{n_T}$ by (\ref{really centered}), \InCenteredClosure\ implies that $J(P_{v_T}(\bd)) = d_{e_0}/2$.  But for any $v\in T^{(0)}$, $J(P_v(\bd)) \geq \max\{d_e/2\,|\, e\ni v\}$ by \RadiusFunction; in particular, if $e_0$ were in $\cale$ it would follow that $J(P_{v_0}(\bd))\geq d_{e_0}/2$ for the other endpoint $v_0\in T^{(0)}$ of $e_0$.  But Definition \ref{bigger admissible}(\ref{radius order}) implies that $J(P_{v_0}(\bd))\leq J(P_{v_T}(\bd))$, so in fact equality would hold, violating our assumption that Proposition \ref{minimum at}(\ref{strict radius order}) does not.

Thus $e_0$ is in $\calf$; or, more precisely, $(e_0,v_T)\in\calf$.  Moreover, the other endpoint $v_0$ of $e_0$ is not in $T$: if it were, Remark \ref{length increasing} would imply that $d_{e_{v_0}} > d_{e_0}$, and applying Definition \ref{bigger admissible}(\ref{closure not centered}) inductively along the path joining $v_0$ to $v_T$ would yield $i > 0$ such that $d_{e_i} > d_{e_0}$, a contradiction.  Therefore changing $d_{e_0}$ while fixing the other entries of $\bd$ changes only $P_{v_T}(\bd)$.

Reducing $d_{e_0}$ while fixing all other entries of $\bd$ takes $P_{v_T}(\bd)$ into $\calc_{n_T}$, by Proposition \ref{smooth BC}, while reducing $D_0(P_{v_T}(\bd))$, by Proposition \ref{omniarea}.  (Note that a such a deformation $\bd(t)$ would have $\frac{d}{dt} D_0(P_{v_T}(\bd(t)) = 0$ at $t=0$ but negative thereafter.)  Since all other $P_v(\bd)$ are unaffected by such a deformation, and since $J(P_{v_T}(\bd))$ varies continuously with $\bd$, by Definition \ref{bigger admissible} such a family $\bd(t)$ would produce new $\bd_{\calf}(t)\in\mathit{SAd}_T\cap [d,D]^{\calf}$ with $\bd(t)\in\overline{\mathit{Ad}}(\bd_{\calf}(t))$ and $D_T(\bd(t)) < D_T(\bd)$.  This contradicts our minimality hypothesis, and it follows that one of (\ref{really not centered}) or (\ref{strict radius order}) must hold.

Suppose first that $\bd$ satisfies (\ref{really not centered}), so $P_v(\bd)\in\calBC_{n_v}$ for each $v\in T^{(0)}-\{v_T\}$, where $n_v$ is the valence of $v$ in $T$.  Fix such $v$ and cyclically enumerate the edges containing $v$ as $e_0,\hdots,e_{n_v-1}$ so that $e_0 = e_v$.  Remark \ref{length increasing} implies that $d_{e_i} \geq d$ for all $i>0$, so $d_{e_0} = b_0(d_{e_1},\hdots,d_{e_{n_v-1}})$ is at least $b_0(d,\hdots,d)$ (recall Proposition \ref{smooth BC}(\ref{parameter monotonicity})).  If $n_v = 3$ then we conclude from Corollary \ref{monotonicity} that $D_0(P_v(\bd)) > D_0(B_0,d,d)$.  If $n_v \geq 4$ then we only need the bound $d_{e_v} > d$ (and Corollary \ref{monotonicity}) to conclude that $D_0(P_v(\bd)) \geq (n_v-2) D_0(B_0,d,d)$ using Proposition \ref{centered main}.  

For $v = v_T$ we argue as above to show that $D_0(P_{v_T}(\bd))\geq (n_T-2)D_0(B_0,d,d)$, where $n_T$ is the valence of $v_T$ in $V$.  If $n_T \geq 4$ then this follows from Proposition \ref{centered main}.  If $n_T = 3$ one observes that since $T$ has at least one edge (by hypothesis), at least one $e\to v$ is of the form $e_v$ for some $v\in T^{(0)}$ so $d_{e_v}\geq b_0(d,\hdots,d)\geq b_0(d,d,0,\hdots,0) = B_0$ (the latter inequality follows from \ShrinkingParameters).  Thus in this case Corollary \ref{monotonicity} implies that $D_0(P_{v_T}(\bd))\geq D_0(B_0,d,d)$.

For $\bd$ satisfying (\ref{really not centered}) the above implies that
$$ D_T(\bd) \geq \sum_{v\in T^{(0)}} (n_v -2)D_0(B_0,d,d) = \left[\left(\sum_{v\in T^{(0)}} n_v\right) - 2n\right]D_0(B_0,d,d), $$
since $T$ has $n$ vertices. We recall that since $T$ is a tree, its Euler characteristic is one so $|\cale|=n-1$.  We also have $\sum_{v\in T^{(0)}} n_v = 2|\cale|+|\calf| = |\calf|-2$.  (Here it is important to recall that each edge $e$ of $V$ that is not in $\cale$ but has both endpoints in $T$ contributes two distinct elements to $\calf$, see above Definition \ref{partial order}.)  This establishes case (\ref{really not centered}).

It remains only to consider the case that $\bd$ satisfies condition (\ref{strict radius order}); ie, that $J(P_v(\bd)) = J(P_w(\bd))$ for some $v\in T^{(0)}$ and $w\in v -1$.  This case follows directly from Lemma \ref{reduce edges} and the induction hypothesis.  The conclusion thus holds for each  $\bd_{\calf}\in\mathit{SAd}_T\cap [d,D]^{\calf}$ and hence, since $D>d$ is arbitrary, for each $\bd_{\calf}\in\mathit{SAd}_T$ with all entries at least $d$.  Since the conclusion is vacuous for $\bd_{\calf}\notin\mathit{SAd}_T$, the result follows.\end{proof}

\begin{theorem}\label{main}\maintheorem\end{theorem}

\begin{proof}  If $C$ is a triangle then it is centered (recall Definition \ref{centered dual}), so the result follows directly from Corollary \ref{monotonicity}.  Therefore assume below that $\partial C$ has $k >3$ edges.

Proposition \ref{omniarea} implies that $A_m(d)$, as defined above, equals $D_0(b_0(d,d),d,d)$ for $b_0$ as defined in Proposition \ref{smooth BC}.  If $C$ is a centered geometric dual cell, the conclusion thus follows by combining Proposition \ref{centered main} with Corollary \ref{monotonicity}.  We may therefore assume that $C = C_T$ is dual to a component $T$ of the non-centered Voronoi subgraph (recall Definition \ref{tree cells}).  In this case $C_T$ has area $D_T(\bd)$ by Lemma \ref{admissible Voronoi}, where the entries of $\bd$ are lengths of geometric duals to edges of $T$ or its frontier in the Voronoi graph.  Since each such edge has length at least $d$ by hypothesis, the result follows directly from Proposition \ref{general lower}.\end{proof}

\section{Admissible spaces and area bounds with mild noncompactness}\label{moduli for noncompact}

The goal of this section is to produce and prove a result analogous to Theorem \ref{main} for centered dual $2$-cells that are not compact, but for which the associated Voronoi subtree still has finite vertex set.  The development follows a parallel track: we introduce an admissible space parametrizing all possible cells with a given edge length collection, in Section \ref{admissible for noncompact}, and minimize the area functional on it in Section \ref{noncompact bounds}.

Section \ref{horocyclic geom} collects some useful results on horocyclic and ``horocyclic ideal'' polygons.

\subsection{Horocyclic ideal polygons}\label{horocyclic geom}  Recall that \textit{horocycles} of $\mathbb{H}^2$ are defined in \ref{sphere at infinity}; in particular, a horocycle has a single \textit{ideal point} on the sphere at infinity $S_{\infty}$ of $\mathbb{H}^2$.

\begin{definition}\label{horocyclic ideal poly}  A \textit{horocyclic polygon} is the convex hull in $\mathbb{H}^2$ of a locally finite subset of a horocycle.  An \textit{infinite} horocyclic polygon $C$ is the convex hull of an infinite, locally finite subset of a horocycle.  A \textit{horocyclic ideal polygon} is the convex hull of the union of a finite subset of a horocycle with its ideal point.\end{definition}

In constructing the centered dual decomposition we have already encountered a horocyclic ideal triangle. The result below follows from Lemma \ref{exclusive horocycle}.

\begin{lemma}  If $e_0$ is a non-compact edge, with ideal vertex $v_{\infty}$, of the Voronoi tessellation of a locally finite set $\cals\subset\mathbb{H}^2$, then $\Delta(e_0,v_{\infty})$ from Definition \ref{tree cells} is a horocyclic ideal triangle.\end{lemma}

The set of marked, oriented horocyclic $n$-gons is parametrized up to orientation-preserving isometry by a subset $\calHC_n$ of $(\mathbb{R}^+)^n$; see Definition \ref{parameter spaces}.  The set of horocyclic ideal $n$-gons is parametrized up to OP isometry by $(\mathbb{R}^+)^{n-2}$.  Such a polygon has two sides of infinite length sharing an ideal vertex, which determine a marking and account for the $n-2$ above.  See \cite[Proposition 6.15]{DeB_cyclic_geom}.  The area of horocyclic and horocyclic ideal polygons has a nice explicit expression.  See Lemma 6.4 and Proposition 6.17 of \cite{DeB_cyclic_geom}.

\begin{proposition}\label{horocyclic defects}  A horocyclic $n$-gon represented by $(d_0,\hdots,d_{n-1})\in\calHC_n$ with maximal entry $d_0$ has area:
$$ D_0(d_0,\hdots,d_{n-1}) = (n-2)\pi + 2\sin^{-1}\left(\frac{1}{\cosh(d_0/2)}\right) - 2\sum_{i=1}^{n-1} \sin^{-1}\left(\frac{1}{\cosh(d_i/2)}\right) $$
An ideal horocyclic $n$-gon represented by $(d_1,\hdots,d_{n-2})\in\calHC_{n+1}^{\infty}$ has area:
$$ D_0(d_1,\hdots,d_{n-2}) = (n-2)\pi - 2\sum_{i=1}^{n-2} \sin^{-1}\left(\frac{1}{\cosh(d_i/2)}\right) $$
The formula above determines a lower-semicontinuous extension of $D_0$ to $\calAC_n\cup\calHC_n\cup\calHC_n^{\infty}$ with the property that $D_0(d_1,\hdots,d_n) \geq D_0(d_1',\hdots,d_n')$ if $d_i \geq d_i'$ for all $n$.\end{proposition}

Recall that for $d>0$, the maximal-area triangle with two sides of length $d$ has a third with length $b_0(d,d)$, where $b_0$ is as defined in Proposition \ref{smooth BC}.  This is still less than the area of a horocyclic ideal triangle with finite side length $d$.

\begin{corollary}\label{big ideals}  For any $d>0$, $D_0(\infty,d,\infty) > D_0(b_0(d,d),d,d)$, for $b_0$ as in Proposition \ref{smooth BC}.\end{corollary}

\begin{proof}  For any $x > d$, Corollary \ref{monotonicity} implies that $D_0(b_0(d,x),d,x) > D_0(b_0(d,d),d,d)$.  Note also that $b_0(d,x) > x$.  Because the extension to $\calHC_3$ described in Proposition \ref{horocyclic defects} is lower-semicontinuous, taking a liminf as $x\to\infty$ gives the result.\end{proof}

The lemma below is the analog, in the context of horocyclic polygons, to Lemma \ref{two polys}.

\begin{lemma}\label{two horocyclic polys}  Suppose $(c_0,\hdots,c_{m-1})\in\calHC_m$ and $(d_0,\hdots,d_{n-1})\in\calHC_n$ have largest entries $c_0$ and $d_0$, respectively, such that $c_0 = d_i$ for some $i>0$.  Then 
$$\bd = (d_0,d_1,\hdots,d_{i-1},c_1,\hdots,c_{m-1},d_{i+1},\hdots,d_{n-1})\in\calHC_{m+n-2},$$ 
and $D_0(c_0,\hdots,c_{m-1})+D_0(d_0,\hdots,d_{n-1}) = D_0(\bd)$.  Analogously, if $(c_0,\hdots,c_m)\in\calHC_m$ has $c_0$ maximal, and $(d_1,\hdots,d_{n-1})\in(\mathbb{R}^+)^{n-1}$ has $d_i = c_0$ for some $i$, then:\begin{align*}
  D_0(c_0,\hdots,c_{m-1}) + D_0(\infty,d_1,\hdots,d_{n-1},\infty) = 
  \qquad\qquad\qquad\qquad\quad\ \  \\
  D_0(\infty,d_1,\hdots,d_{i-1},c_1,\hdots,c_{m-1},d_{i+1},\hdots,d_{n-1},\infty)\end{align*}\end{lemma}

Lemma \ref{two horocyclic polys} follows directly from the formulas of Proposition \ref{horocyclic defects}.  It reflects the following geometric picture: if the largest side length of a horocyclic polygon $P$ happens to equal a side length of a (say) horocyclic ideal polygon $Q$, then upon moving $P$ by an isometry so that $P\cap Q$ is the longest side of $P$, $P\cup Q$ is itself a horocyclic ideal polygon with area the sum $D_0(P)+D_0(Q)$.

\subsection{Admissible spaces: the case of a non-compact edge}\label{admissible for noncompact}  This is the analog of Section \ref{admissible for compact} for centered dual two-cells $C_T$ such that the dual tree $T$ has a non-compact edge $e_0$ but $T^{(0)}$ finite.  We recall Proposition \ref{to the root}, which motivates this section's blanket hypothesis.

\begin{hypothesis}  Except where explicitly noted, here $V$ is a finite graph with vertices of valence at least $3$ and (possibly) some non-compact edges, and $T\subset V$ is a rooted subtree with a single non-compact edge $e_0$ and root vertex $v_T\in e_0$.  Let $\cale$ be the edge set of $T$ and $\calf$ its frontier in $F$.  We let $n_v$ denote the valence in $V$ of a vertex $v$ of $T$.\end{hypothesis}

The major definitions and results of this section closely parallel those of Section \ref{admissible for compact}, though almost all will require some revision.  We will compare and contrast as appropriate.  Below is the analog of Definition \ref{partial order}, differing from the original only in that we define $e_v$ for $v = v_T$.

\begin{definition}\label{partial order for noncompact}  Partially order $T^{(0)}$ by setting $v < v_T$ for each $v\in T^{(0)} - \{v_T\}$, and $w < v$ if the edge arc in $T$ joining $w \in T^{(0)}- \{v_T,v\}$ to $v_T$ runs through $v$.  Let $v-1$ be the set of $w<v$ joined to it by an edge, and say $v$ is \textit{minimal} if $v-1=\emptyset$.  Let $e_{v_T} = e_0$, and for $v\in T^{(0)}-\{v_T\}$, let $e_v$ be the initial edge of the arc in $T$ joining $v$ to $v_T$.   For each $v\in T^{(0)}$, say ``$e\to v$'' for each edge $e\neq e_v$ of $V$ containing $v$.\end{definition}

Again the definition below differs from its predecessor Definition \ref{admissible criteria} only in treating $v_T$ like other vertices of $T$.

\begin{definition}\label{admissible criteria for noncompact}  Let $(\mathbb{R}^+)^{\calf}$ be the set of tuples of positive real numbers indexed by the elements of $\calf$, and define $(\mathbb{R}^+)^{\cale}$ analogously.  For any elements $\bd_{\cale}=(d_e\,|\,e\in\cale)\in (\mathbb{R}^+)^{\mathcal{E}}$ and $\bd_{\calf}\in(\mathbb{R}^+)^{\calf}$, let $\bd = (\bd_{\cale},\bd_{\calf})$ and $P_{v}(\bd) = (d_{e_0},\hdots,d_{e_{n-1}})$ for $v\in T^{(0)}$, where the edges of $V$ containing $v$ are cyclically ordered as $e_0,\hdots,e_{n-1}$.  We say the \textit{admissible set} $\mathit{Ad}(\bd_{\calf})$ determined by $\bd_{\calf}$ is the collection of $\bd \in(\mathbb{R}^+)^{\cale}\times\{\bd_{\calf}\}$ such that: \begin{enumerate}
\item\label{not centered for noncompact} For each $v\in T^{(0)}$, $P_v(\bd) \in \mathcal{AC}_{n_v} -\calc_{n_v}$ has largest entry $d_{e_v}$. 
\item\label{radius order for noncompact} $J(P_{v}(\bd)) > J(P_w(\bd))$ for each $w\in v- 1$, where $J(P_v(\bd))$ and $J(P_w(\bd))$ are the respective radii of $P_v(\bd)$ and  $P_w(\bd)$.  \end{enumerate}  
\end{definition}

\begin{definition}\label{tree defect for noncompact}  Fix $\bd_{\calf} = (d_e\,|\,e\in\calf)\in (\mathbb{R}^+)^{\calf}$ such that $\mathit{Ad}(\bd_{\calf})\neq \emptyset$.  For each $\bd\in\mathit{Ad}(\bd_{\calf})$ and $R\geq 0$, 
define:
$$ D_T(\bd) = \pi-2\sin^{-1}\left(\frac{1}{\cosh (d_{e_0}/2)}\right) + \sum_{v\in T^{(0)}} D_0(P_v(\bd)), $$
where $P_v(\bd)$ is as in Definition \ref{admissible criteria for noncompact} and $D_0(P)$ is as in Proposition \ref{omniarea}.  \end{definition}

\begin{lemma}\label{admissible Voronoi for noncompact}  Let $C_T$ be a centered dual two-cell, dual to a component $T$ of the non-centered Voronoi subgraph determined by locally finite $\cals\subset\mathbb{H}^2$ with a non-compact edge $e_0$ and $T^{(0)}$ finite.  Let $\mathcal{E}$ be the edge set of $T$ and $\mathcal{F}$ its frontier in the Voronoi graph $V$, and for each edge $e$ of $V$ that intersects $T$ let $d_e$ be the length of the geometric dual to $e$.  Then $\bd = (d_e\,|\,e\in\cale) \in \mathit{Ad}(\bd_{\calf})$, where $\bd_{\calf}=(d_e\,|\, (e,v)\in\calf\mbox{ for some }v\in T^{(0)})$, and $C_T$ has area $D_T(\bd)$.  \end{lemma}

\begin{proof}  The proof is analogous to that of Lemma \ref{admissible Voronoi}, with a couple of differences.  Again the main point is that for each $v\in T^{(0)}$, $C_v$ is represented in $\calAC_{n_v}$ by $P_v(\bd)$.  In contrast with that case, here $C_{v_T}$ is non-centered, by Proposition \ref{to the root}, and its longest side is the geometric dual $\gamma_0$ to $e_0$, by Lemma \ref{vertex polygon decomp}.

By Definition \ref{tree cells}, $C_T = \Delta(e_0,v_{\infty})\cup \bigcup_{v\in T^{(0)}} C_v$ in this case, where $v_{\infty}$ is the ideal endpoint of $e_0$.  Lemma \ref{no infinite triangle overlap} implies that the union above is non-overlapping, so the area of $C_T$ is the sum of the areas of the $C_v$ with that of $\Delta(e_0,v_{\infty})$.  But $\Delta(e_0,v_{\infty})$ is a horocyclic ideal triangle with vertices on the unique horocycle through the endpoints of $\gamma_0$ with ideal point $v_{\infty}$ (see Lemma \ref{exclusive horocycle}).  Thus its area is $\pi-2\sin^{-1}(1/\cosh (d_{e_0}/2))$; see eg.~\cite[Proposition 6.17]{DeB_cyclic_geom}.  The lemma follows.\end{proof}

As with Definition \ref{bigger admissible}, we will compactify $\mathit{Ad}(\bd_{\calf})$ here by expanding it somewhat.

\begin{definition}\label{bigger admissible for noncompact}  For $\bd_{\calf} = (d_e\,|\,e\in\calf)\in(\mathbb{R}^+)^{\calf}$ let $\overline{\mathit{Ad}}(\bd_{\calf})$ consist of those $\bd = (\bd_{\cale},\bd_{\calf})$, for $\bd_{\cale}\in (\mathbb{R}^+)^{\cale}$, such that:\begin{enumerate}
\item\label{closure not centered for noncompact} For each $v\in T^{(0)}$, with valence $n_v$ in $V$, $P_v(\bd) \in(\mathcal{AC}_{n_v}\cup\calHC_{n_v}) - \calc_{n_v}$ has largest entry $d_{e_v}$. 
\item\label{closure radius order for noncompact} $J(P_{v}(\bd)) \geq J(P_w(\bd))$ for each $w\in v- 1$, where $J(P_v(\bd))$ and $J(P_w(\bd))$ are the respective radii of $P_v(\bd)$ and $P_w(\bd)$, and $J(P)\doteq\infty$ if $P\in\calHC_n$; see \cite[Lemma 6.6]{DeB_cyclic_geom}.\end{enumerate}
\end{definition}

Note that $\overline{\mathit{Ad}}(\calf)$ above is somewhat ``larger'' than its analog from Definition \ref{bigger admissible}, since it includes points of $\calHC_n$.  We nonetheless require only subtle changes to the analog of Lemma \ref{precompact}. In particular, the properties below apply to all vertices of $T$, including $v_T$, and in (\ref{admissible bounded for noncompact}) below an inequality ceases to be strict.

\begin{lemma}\label{precompact for noncompact}  Collections $\{b_e\co(\mathbb{R}^+)^{\calf}\to\mathbb{R}^+\}_{e\in\cale}$ and $\{h_e\co(\mathbb{R}^+)^{\calf}\to\mathbb{R}^+\}_{e\in\cale}$ are determined by the following properties: for $\bd_{\calf}\in(\mathbb{R}^+)^{\calf}$ and $\bd_{\cale}\in(\mathbb{R}^+)^{\cale}$, with $\bd=(\bd_{\cale},\bd_{\calf})$:\begin{itemize}
\item  $P_v(\bd)\in\calBC_{n_v}$, with largest entry $d_{e_v}$, for all $v\in T^{(0)}$ if and only if $d_e=b_e(\bd_{\calf})$ for each $e\in\cale$.
\item  $P_v(\bd)\in\calHC_{n_v}$, with largest entry $d_{e_v}$, for all $v\in T^{(0)}$ if and only if $d_e=h_e(\bd_{\calf})$ for each $e\in\cale$.\end{itemize}
The collections $\{b_e\}$ and $\{h_e\}$ have the following properties:\begin{enumerate}
\item\label{admissible bounded for noncompact}  If $\bd=(\bd_{\cale},\bd_{\calf}) \in \overline{\mathit{Ad}}(\bd_{\calf})$ then for each $e\in\mathcal{E}$, $b_e(\bd_{\calf}) \leq d_e\leq h_e(\bd_{\calf})$.
\item\label{b increasing in T for noncompact}  For $\bd_{\calf}$ and $v\in T^{(0)}$, $b_{e_v}(\bd_{\calf}) > \max \{b_e(\bd_{\calf})\,|\, e\to v\in\cale\}\cup\{d_e\,|\,e\to v\in\calf\}$.
\item\label{b increasing for noncompact}  If $d_e' \geq d_e$ for each $e\in\calf$ then $b_e(\bd_{\calf}')\geq b_e(\bd_{\calf})$ for each $e\in\cale$, where $\bd_{\calf}' = (d_e')_{e\in\calf}$.\end{enumerate}\end{lemma}

The proof directly follows that of Lemma \ref{precompact}, and there is no need to rehash it.  We merely point out that the reason $d_e$ as at most (instead of less than) $h_e(\bd_{\calf})$ in (\ref{admissible bounded for noncompact}) here is that now $P_v(\bd)$ may be in $\calHC_{n_v}$ for $\bd\in\overline{\mathit{Ad}}(\bd_{\calf})$.

\begin{lemma}\label{admissible closure for noncompact}  For any $\bd_{\calf}\in(\mathbb{R}^+)^{\calf}$, $\overline{\mathit{Ad}}(\bd_{\calf})$ is compact.\end{lemma}

The proof of this result is easier than the proof of its antecedent Lemma \ref{admissible closure} because $(\calAC_n\cup\calHC_n)-\calc_n$ is closed in $(\mathbb{R}^+)^n$, unlike $\calAC_n-\calc_n$.  See Lemmas 6.4 and 6.6 of \cite{DeB_cyclic_geom}.  For this reason the criterion (\ref{closure not centered for noncompact}) of Definition \ref{bigger admissible for noncompact} is preserved in limits of points in $\overline{\mathit{Ad}}(\bd_{\calf})$.  That criterion (\ref{radius order for noncompact}) is preserved in limits, and that $\overline{\mathit{Ad}}(\bd_{\calf})$ is bounded in $(\mathbb{R}^+)^{\calf}$ away from its frontier in $\mathbb{R}^{\calf}$, follow as before.

We also observe the analog of Lemma \ref{continuous defect}.

\begin{lemma}\label{continuous defect for noncompact}  For $\bd_{\calf} = (d_e\,|\,e\in\calf)\subset (\mathbb{R}^+)^{\calf}$ such that $\overline{\mathit{Ad}}(\bd_{\calf})\neq\emptyset$, $D_T(\bd)$ is continuous on $\overline{\mathit{Ad}}(\bd_{\calf})$ and attains a minimum there.\end{lemma}

The only thing worth adding to the proof here is that $\sin^{-1}(1/\cosh(d_{e_0}/2))$ clearly varies continuously with $\bd$ (compare Definitions \ref{tree defect} and \ref{tree defect for noncompact}).  Finally, a version of Proposition \ref{minimum at} for the current context:

\begin{proposition}\label{minimum at for noncompact}  For $\bd_{\calf}\in(\mathbb{R}^+)^{\calf}$ with $\overline{\mathit{Ad}}(\bd_{\calf})\neq\emptyset$, at a minimum point $\bd = (\bd_{\cale},\bd_{\calf})$ for $D_T(\bd)$ on of the following holds:\begin{enumerate}
\item\label{really not centered for noncompact}  $P_v(\bd)\in\calBC_{n_v}$ for each $v\in T^{(0)}$;
\item\label{really horocyclic}  $P_{v_T}(\bd)\in\calHC_{n_{v_T}}$; or
\item\label{strict radius order for noncompact}  $J(P_v(\bd)) = J(P_w(\bd))$ for some $v\in T^{(0)}$ and $w\in v-1$.
\end{enumerate}\end{proposition}

\begin{proof}  The proof follows the strategy of Proposition \ref{minimum at}: we suppose none of the above criteria holds at $\bd\in\overline{\mathit{Ad}}(\bd_{\calf})$, fix $v\in T^{(0)}$ such that $P_v(\bd)\notin\calBC_{n_v}$, and show that reducing $d_{e_v}$ while keeping all other entries of $\bd$ constant produces a deformation through $\overline{\mathit{Ad}}(\bd_{\calf})$ that lowers the value of $D_T$.

That (\ref{really horocyclic}) does not hold implies for each $w\in T^{(0)}$ that $P_w(\bd)\notin\calHC_{n_w}$.  This is because if $P_w(\bd)\in\calHC_{n_w}$ then $J(P_w(\bd))=\infty$, so criterion (\ref{radius order for noncompact}) of Definition \ref{bigger admissible for noncompact} implies that  $J(P_{v'}(\bd))=\infty$, and hence $P_{v'}(\bd)\in\calHC_{n_{v'}}$, for all $v'\in T^{(0)}$ with $w<v'$; in particular for $v'=v_T$.

For $v < v_T$ the argument of Proposition \ref{minimum at} thus shows that the deformation described in the first paragraph acts as claimed there.  If $v = v_T$ then $e_v = e_0$ is the non-compact edge of $T$, so $\partial D_T/\partial d_{e_v}$ is not quite as described in (\ref{defect deriv diff}).  Instead we have:\begin{align}\label{defect deriv diff for noncompact}
  \frac{\partial D_T}{\partial d_{e_0}} = \frac{1}{\cosh (d_{e_0}/2)} - \sqrt{\frac{1}{\cosh^2(d_{e_0}/2)} - \frac{1}{\cosh^2 J(P_{v_T}(\bd))}}  \end{align}
The right-hand quantity above is $\partial D_0(P_{v_T}(\bd))/\partial d_{e_0}$ (by Proposition \ref{omniarea}); on the left is $\left[\pi - 2\sin^{-1}(1/\cosh(d_{e_0}/2))\right]'$ (by direct computation).
\end{proof}

\subsection{Another area bound}\label{noncompact bounds}  Here we will prove an analog of Theorem \ref{main} for centered dual $2$-cells $C_T$ that are dual to components $T$ of the non-centered Voronoi subgraph with a non-compact edge but finite vertex set (recall Definition \ref{tree cells}).

\begin{theorem}\label{main for noncompact}  Let $C_T$ be a centered dual $2$-cell, dual to a component $T$ of the non-centered Voronoi subgraph determined by locally finite $\cals\subset\mathbb{H}^2$ with finite vertex set but a noncompact edge.  For $d>0$, if $C_T$ has $k$ edges and each has length at least $d$ then
$$ \mathrm{Area}(C_T) \geq D_0(\infty,b_0(d,d),\infty) + (k-3)D_0(b_0(d,d),d,d) $$
for $b_0$ as in Proposition \ref{smooth BC} and $D_0$ as in Proposition \ref{horocyclic defects}.\end{theorem}

The proof strategy is similar to that of Theorem \ref{main}.  In particular, we again induct on the number of vertices.  Here however, in the one-vertex case $T$ has a single non-compact edge.  We address this directly below.

\begin{lemma}\label{base case for noncompact}  Let $T =\{e_0\}$ be a non-compact edge of a finite graph $V$, with vertex $v$ of valence $n\geq 3$ in $V$.  Cyclically enumerate the edges containing $v$ as $e_0,e_1,\hdots,e_{n-1}$.  Then for any $\bd_{\calf} = (d_{e_1},\hdots,d_{e_{n-1}})\in(\mathbb{R}^+)^{n-1}$, $\overline{\mathit{Ad}}(\bd_{\calf}) = [b_{e_0}(\bd_{\calf}),h_{e_0}(\bd_{\calf})]\times\{\bd_{\calf}\}$ (where $b_{e_0}$ and $h_{e_0}$ are as in Lemma \ref{precompact for noncompact}).  $D_T$ takes its minimum and maximum values at the left and right endpoints of this interval, respectively.  These are given by:\begin{align*}
  D_T(b_{e_0}(\bd_{\calf}),\bd_{\calf}) & = D_0(\infty,b_{e_0}(\bd_{\calf}),\infty) + D_0(b_{e_0}(\bd_{\calf}),\bd_{\calf}) \\
  D_T(h_{e_0}(\bd_{\calf}),\bd_{\calf}) & = D_0(\infty,\bd_{\calf},\infty)\end{align*}\end{lemma}

Before proving the lemma we record a useful corollary pertaining to horocyclic ideal polygons.

\begin{corollary}\label{horocyclic ideal bound}  For $n\geq 2$ and $\bd = (d_1,\hdots,d_n)\in(\mathbb{R}^+)^{n}$, we have:\begin{align*}
  D_0(\infty,\bd,\infty) & > D_0(\infty,b_0(\bd),\infty) + D_0(b_0(\bd),\bd) \end{align*}
Here $b_0$ is as in Proposition \ref{smooth BC} and $D_0$ is from Proposition \ref{horocyclic defects}.  For $d>0$, if $d_i\geq d$ for each $i$ then:
$$  D_0(\infty,\bd,\infty) > D_0(\infty,b_0(d,d),\infty) + (n-1)D_0(b_0(d,d),d,d) $$\end{corollary}

The first assertion follows directly from Lemma \ref{base case for noncompact}; the only thing to note is that by its definition in Lemma \ref{precompact}, in this case $b_{e_0} = b_0\co(\mathbb{R}^+)^{n-2}\to\mathbb{R}^+$.  The second follows from the first, applying monotonicity of $b_0$ (see Proposition \ref{smooth BC}(\ref{parameter monotonicity})) and $D_0$ (by Corollary \ref{monotonicity} and Proposition \ref{horocyclic defects}), and Proposition \ref{centered main}.

\begin{proof}[Proof of Lemma \ref{base case for noncompact}]  Lemma \ref{precompact for noncompact}(\ref{admissible bounded for noncompact}) asserts that $\overline{\mathit{Ad}}(\bd_{\calf})$ is contained in the set above, and by the definitions of $b_{e_0}$ and $h_{e_0}$, $P_v(\bd)\in(\calAC_n-\calc_n)\cup\calHC_n$ for any $\bd = (d,\bd_{\calf})$ where $b_{e_0}(\bd_{\calf})\leq d\leq h_{e_0}(\bd_{\calf})$ (cf.~Proposition \ref{smooth BC}).  Thus Definition \ref{bigger admissible for noncompact}(\ref{closure not centered for noncompact}) holds for such $\bd$, and since (\ref{closure radius order for noncompact}) holds vacuously in this case, $\overline{\mathit{Ad}}(\bd_{\calf}) = [b_{e_0}(\bd_{\calf}),h_{e_0}(\bd_{\calf})]\times\{\bd_{\calf}\}$.

We appeal to Definition \ref{tree defect for noncompact} and Proposition \ref{omniarea} to compute the derivative of $D_T(\bd)$ at $\bd = (d,\bd_{\calf})$ in the open interval $(b_{e_0}(\bd_{\calf}),h_{e_0}(\bd_{\calf}))\times\{\bd_{\calf}\}$:
$$ \frac{\partial}{\partial d} D_T(d,\bd_{\calf}) = \frac{1}{\cosh(d/2)} - \sqrt{\frac{1}{\cosh^2(d/2)} -\frac{1}{\cosh^2 J(d,\bd_{\calf})}} $$
This is clearly positive for all such $d$, so $D_T(d,\bd_{\calf})$ attains its minimum on $\overline{\mathit{Ad}}(\bd_{\calf})$ at $(b_{e_0}(\bd_{\calf}),\bd_{\calf})$ and its maximum at $(h_{e_0}(\bd_{\calf}),\bd_{\calf})$.

That $D_T(b_{e_0}(\bd_{\calf}),\bd_{\calf})$ is as described above is a direct application of Definition \ref{tree defect for noncompact}.  That $D_T(h_{e_0}(\bd_{\calf},\bd_{\calf})) = D_0(\infty,\bd_{\calf},\infty)$ follows from the definition and the second assertion of Lemma \ref{two horocyclic polys}, since  $P_{v_T}(h_{e_0}(\bd_{\calf}),\bd_{\calf})\in\calHC_n$ by its definition in Lemma \ref{precompact for noncompact}.\end{proof}

We are now in position to prove Theorem \ref{main for noncompact}.

\begin{proof}[Proof of Theorem \ref{main for noncompact}]  Recall (from Definition \ref{tree cells}) that $\partial C_T$ is the union of geometric duals to edges in the frontier $\calf$ of the tree $T$ dual to $C_T$, together with the two infinite edges of $\Delta(e_0,v_{\infty})$.  Here $e_0$ is the noncompact edge of $T$ and $v_{\infty}$ is its ideal endpoint.  In particular, $\partial C_T$ has $k = |\calf|+2$ edges.

Let $\bd_{\calf}$ collect the lengths of the geometric duals to edges of $\calf$.  By hypothesis, $d_e \geq d > 0$ for each $e\in\calf$.  By Lemma \ref{admissible Voronoi for noncompact}, $C_T$ has area equal to $D_T(\bd)$ for some $\bd\in\mathit{Ad}(\bd_{\calf})\subset \overline{\mathit{Ad}}(\bd_{\calf})$.  We will thus prove the result by showing that for every tree $T$ with one non-compact edge:
$$D_T(\bd)\geq  D_0(\infty,b_0(d,d),\infty) + (|\calf|-1)D_0(b_0(d,d),d,d) $$
for every $\bd_{\calf} = (d_e\,|\,(e,v)\in\calf\ \mbox{for some}\ v\in T^{(0)})$ with all $d_e\geq d$ and $\bd\in\overline{\mathit{Ad}}(\bd_{\calf})$.

If $T$ has one vertex the result follows directly from Lemma \ref{base case for noncompact}.  We will thus assume that $T$ has $n>1$ vertices, and that for all trees with fewer than $n$ vertices, $\min \{D_T(\bd),\bd\in\overline{\mathit{Ad}}(\bd_{\calf})\}$ satisfies the conclusion if $d_e \geq d$ for all $e\in\calf$.

A minimum point $\bd$ for $D_T$ on $\overline{\mathit{Ad}}(\bd_{\calf})$ satisfies one of the cases described in Proposition \ref{minimum at for noncompact}.  Cases (\ref{really not centered for noncompact}) and (\ref{strict radius order for noncompact}) follow from lines of argument analogous to those for Theorem \ref{main}.  In Case (\ref{really not centered for noncompact}), a direct computation yields the conclusion here as well.  In Case (\ref{strict radius order for noncompact}) we collapse the edge $f$ shared by $v$ and $w$ and appeal to induction.  A new possibility here is that $J(P_v(\bd)) = J(P_w(\bd)) = \infty$; i.e., $P_v(\bd)\in\calHC_{n_v}$ and $P_w(\bd)\in\calHC_{n_w}$.  Lemma \ref{smushanedge} still holds in this case, though, upon replacing the appeal to Lemma \ref{two polys} with one to Lemma \ref{two horocyclic polys}.

Our treatment of Case (\ref{really horocyclic}) from the conclusion of Proposition \ref{minimum at for noncompact} departs from the analogous case in the proof of Theorem \ref{main}.  We suppose henceforth that $\bd\in\overline{\mathit{Ad}}(\bd_{\calf})$ satisfies this case; i.e. that $P_{v_T}(\bd)\in\calHC_{n_{v_T}}$.

Let $\cale$ be the edge set of $T$ and $e_0$ its non-compact edge. Removing $e_0$ from $T$ yields a collection $T_1,\hdots, T_l$ of subtrees, one for each edge of $T$ that contains $v_T$.  For each $i$, $T_i$ has a single non-compact edge $e_i$ whose closure contains $v_T$.  Let $\cale_i$ be the edge set of $T_i$ and $\calf_i$ its frontier in $V$, and define $\bd_{\calf_i} = (d_e\,|\,(e,v)\in\calf_i\ \mbox{for some}\ v\in T_i^{(0)})$, $\bd_{\cale_i} = (d_e\,|\,e\in\cale_i)$ and $\bd_i = (\bd_{\cale_i},\bd_{\calf_i})$.  It is evident that the property $\bd_i\in\overline{\mathit{Ad}}(\bd_{\calf_i})$ is inherited from the corresponding property of $\bd$.

Suppose $v_T$ has valence $n_0$ in $V$, and let $e_{l+1},\hdots,e_{n_0-1}$ be the collection of edges in $\calf$ that contain $v_T$.  Then $\cale = \{e_0\}\cup\bigcup_{i=1}^l \cale_i$, and $\calf = \{e_{l+1},\hdots,e_{n_0-1}\}\cup\bigcup_{i=1}^l\calf_i$.  We claim that:
$$  D_T(\bd) = \sum_{i=1}^l D_{T_i}(\bd_i) + \sum_{i=l+1}^{n_0-1} D_0(\infty,d_{e_i},\infty)  $$
The main point here is simply that because $P_{v_T}(\bd)\in \calHC_{n_{v_T}}$, applying Lemma \ref{two horocyclic polys} gives\begin{align*}
  D_0(\infty,d_{e_0},\infty)+D_0(P_{v_T}(\bd)) 
    & = D_0(\infty,d_{e_1},\hdots,d_{e_{n_0-1}},\infty) \\
    & = \sum_{i=1}^{n_0-1} D_0(\infty,d_{e_i},\infty)\end{align*}
That the second and third quantities above are equal is evident on its face from the latter formula of Proposition \ref{horocyclic defects}.  Since each vertex of $T$ other than $v_T$ is in exactly one $T_i$, the claim follows.

The inductive hypothesis applies to $T_i$ for each $i$, so using the claim we find:\begin{align*}
  D_T(\bd) & \geq \sum_{i=1}^l \left[ D_0(\infty,b_0(d,d),\infty) + (|\calf_i|-1)D_0(b_0(d,d),d,d)\right] +\sum_{i=l+1}^{n_0-1} D_0(\infty,d,\infty) \\
  & \geq D_0(\infty,b_0(d,d),\infty) + \left(\sum_{i=1}^l |\calf_i| - 1\right)D_0(b_0(d,d),d,d) + (n_0-1-l)D_0(\infty,d,\infty) \end{align*}
The latter inequality above follows from Lemma \ref{big ideals} and Corollary \ref{monotonicity}.  Applying Lemma \ref{big ideals} again, and the fact that $|\calf| = \sum_{i=1}^l |\calf_i| + n_0-1-l$, gives the result.
\end{proof}

\section{On hyperbolic surfaces}\label{limits}

The main goal of this section is to prove Theorem \ref{main for noncompact}.  Then in Section \ref{examples} we will describe families of hyperbolic surfaces with maximal injectivity radius approaching its upper bound, showing this bound is sharp.  First, in Section \ref{surface dual} below we recall some facts from \cite{DeB_Delaunay} on the Delaunay tessellation and geometric dual complex of a finite subset of a hyperbolic surface, and use these to produce a description of the centered dual.

\subsection{When covering a surface}\label{surface dual}  Below we interpret Theorem 6.23 of \cite{DeB_Delaunay} for surfaces.

\begin{theorem}\label{pseudo EP} For a complete, oriented, finite-area hyperbolic surface $F$ with locally isometric universal cover $\pi\co\mathbb{H}^2\to F$, and a finite set $\cals\subset F$, the Delaunay tessellation of $\widetilde{\cals} = \pi^{-1}(\cals)$ is a locally finite, $\pi_1 F$-invariant decomposition of $\mathbb{H}^2$ into convex polygons (the \mbox{\rm cells}) such that each edge of each cell is a cell, and distinct cells that intersect do so in an edge of each.  For each circle or horocycle of $\mathbb{H}^2$ that intersects $\cals$ and bounds a disk or horoball $B$ with $B\cap\cals = S\cap\cals$, the closed convex hull of $S\cap\cals$ in $\mathbb{H}^2$ is a Delaunay cell.  Each Delaunay cell has this form.

For each parabolic fixed point $u\in S_{\infty}$ there is a unique $\Gamma_u$-invariant $2$-cell $C_U$, where $\Gamma_u$ is the stabilizer of $u$ in $\pi_1 F$, whose unique circumcircle (in the sense above) is a horocycle with ideal point $u$.  Each other cell is compact and has a metric circumcircle.\end{theorem}

Fixing a locally isometric universal cover $\pi\co\mathbb{H}^2\to F$ determines an isomorphic embedding of $\pi_1 F$ to a lattice in $\mathrm{PSL}_2(\mathbb{R})$, so that $\pi$ factors through an isometry $\mathbb{H}^2/\pi_1 F\to F$.  An element of $\pi_1 F$ is \textit{parabolic} if it fixes a unique $u\in S_{\infty}$ (recall Definition \ref{sphere at infinity}); such a point $u$ is a \textit{parabolic fixed point}.

Corollary 6.26 of \cite{DeB_Delaunay} describes the image of the Delaunay tessellation in $F$ itself.

\begin{corollary}\label{down below}   For a complete, oriented hyperbolic surface $F$ of finite area with locally isometric universal cover $\pi\co\mathbb{H}^2\to F$, and a finite set $\cals\subset F$, there are finitely many $\pi_1 F$-orbits of Delaunay cells of $\widetilde{\cals} = \pi^{-1}(\cals)$.  The interior of each compact Delaunay cell embeds in $F$ under $\pi$.  For a cell $C_u$ with parabolic stabilizer $\Gamma_u$, $\pi|_{\mathit{int}\, C_u}$ factors through an embedding of $\mathit{int}\,C_u/\Gamma_u$ to a set containing a cusp of $F$.\end{corollary}

A \textit{cusp} of $F$ is a non-compact component of the \textit{$\epsilon$-thin part} of $F$, $F_{(0,\epsilon]} = \{x\in F\,|\, \mathit{injrad}_x F \leq \epsilon\}$, for some $\epsilon >0$ that is less than the two-dimensional Margulis constant.  See eg.~\cite[Ch.~D]{BenPet}.  Each cusp is of the form $B/\Gamma_u$, for a horoball $B$ whose ideal point is a parabolic fixed point $u$ with cyclic stabilizer $\Gamma_u$ in $\pi_1 F$.

The geometric dual complex is identified in Remark 6.24 of \cite{DeB_Delaunay}.

\begin{remark}\label{lattice geometric dual}  For a complete, oriented hyperbolic surface $F$ of finite area with locally isometric universal cover $\pi\co\mathbb{H}^2\to F$, and a finite set $\cals\subset F$, the geometric dual complex of $\widetilde{\cals} = \pi^{-1}(\cals)$ includes all Delaunay cells but the parabolic-invariant ones, and $\pi$ embeds the interior of each geometric dual cell in $F$.\end{remark}

We now build on these results to describe the centered dual complex.  The first observations below use the notions of ``centeredness'' from Proposition \ref{vertex polygon} and Definition \ref{centered edge}.

\begin{lemma}\label{equivalence}  For a complete, oriented hyperbolic surface $F$ of finite area with locally isometric universal cover $\pi\co\mathbb{H}^2\to F$, and a finite set $\cals\subset F$ with $\widetilde{\cals} = \pi^{-1}(\cals)$:\begin{itemize}
  \item if Voronoi vertices $v$ and $w$ of $\widetilde{\cals}$ satisfy $v = g.w$ for some $g\in\pi_1 F$ then $J_v = J_w$ (recall Fact \ref{vertex radius}), and the geometric dual $C_v= g.C_w$ is centered if and only if $C_w$ is;
  \item  and if Voronoi edges $e$ and $f$ determined by $\widetilde{\cals}$ satisfy $\pi(e)=\pi(f)$, then $e$ is centered if and only if $f$ is centered.\end{itemize}\end{lemma}

This follows from the fact that $\pi_1 F$-acts \textit{isometrically} by covering transformations.  If Voronoi edges $v$ and $w$ project to the same point of $F$, then the covering transformation taking $v$ to $w$ takes the sphere of radius $J_v$ centered at $v$ to the sphere of radius $J_v$ centered at $w$, and $C_v$ to $C_w$.  And if a centered Voronoi edge $e$ has the same projection as $f$, then the covering transformation taking $e$ to $f$ takes the intersection of $e$ with its geometric dual to the intersection of $f$ with its geometric dual.

\begin{lemma}\label{surface tree components}  For a complete, oriented hyperbolic surface $F$ of finite area with locally isometric universal cover $\pi\co\mathbb{H}^2\to F$, and a finite set $\cals\subset F$, any component $T$ of the non-centered Voronoi subgraph of $\widetilde{\cals} = \pi^{-1}(\cals)$ is a tree, with $T^{(0)}$ finite, that embeds in $F$ under $\pi$.\end{lemma}

\begin{proof}  Corollary \ref{down below} implies the set of Voronoi vertices determined by $\widetilde{\cals}$ has finitely many $\pi_1 F$-orbits, since Voronoi vertices are geometric duals to compact $2$-cells of the Delaunay tessellation (Remark \ref{lattice geometric dual}).  Therefore the set $\{J_v\,|\,v\in T^{(0)}\}$ has only finitely many distinct elements (recall Lemma \ref{equivalence}).  Let $v_T$ satisfy $J_{v_T}\geq J_v$ for all $v\in T^{(0)}$.  By Lemma \ref{tree components}, $T$ is a tree and $v_T$ satisfies $J_{v_T} > J_v$ for all $v\in T^{(0)}-\{v_T\}$ (cf.~Definition \ref{root vertex} and below).

Since covering transformations exchange components of the union of non-centered edges, if $\gamma.T\cap T \neq \emptyset$ for some $\gamma\in\pi_1 F - \{1\}$ then $\gamma.T = T$.  Since $J_{\gamma.v} = J_v$ for each $v\in T^{(0)}$, the claim above would imply that $\gamma.v_T = v_T$ for such $\gamma$, contradicting freeness of the $\pi_1 F$-action.  Therefore $T$ does not intersect its $\pi_1 F$-translates and thus projects homeomorphically to $F$.  Moreover, each $\pi_1 F$-orbit of Voronoi vertices contains at most one point of $T^{(0)}$, so $T^{(0)}$ is finite.\end{proof}

\begin{corollary}\label{it's a decomp!}  For a complete, oriented hyperbolic surface $F$ of finite area with locally isometric universal cover $\pi\co\mathbb{H}^2\to F$, and a finite set $\cals\subset F$, the centered dual complex of $\widetilde{\cals} = \pi^{-1}(\cals)$ is $\pi_1F$-invariant, and $\pi$ embeds the interior of each centered dual cell in $F$.\end{corollary}

\begin{proof}  The invariance of the centered dual follows directly from Lemma \ref{equivalence} (recalling Definition \ref{centered dual}).  Since $g\in\pi_1 F$ has no fixed points as a covering transformation of $\mathbb{H}^2$, it is not hard to see that $g$ does not preserve any vertex or edge, or any geometric dual two-cell (each of which is compact, recall Lemma \ref{vertex polygon}).  For a cell of the form $C_T$, where $T$ is a component of the non-centered Voronoi subgraph, applying Definition \ref{tree cells} then Lemma \ref{surface tree components} shows that $g(C_T) = C_{g(T)}\neq C_T$.\end{proof}

\subsection{The centered dual plus}\label{plus!}  For a non-compact hyperbolic surface $F$ and a finite subset $\cals$ as in the previous section, the Delaunay tessellation and centered dual complex of the preimage of $\cals$ in $\mathbb{H}^2$ have some related problems.  Parabolic-invariant Delaunay cells are stabilized by nontrivial subgroups of $\pi_1 F$, so their interiors do not embed in $F$ and have non-trivial topology.  On the other hand the centered dual complex may not cover $\mathbb{H}^2$, and its non-compact $2$-cells overlap parabolic-invariant Delaunay cells without containing them.

In this brief section we will try to clarify the situation, ultimately introducing a complex that we call the ``centered dual plus'' in which parabolic-invariant Delaunay cells have been decomposed into unions of horocyclic ideal triangles.

\begin{lemma}\label{inclusive horocycle} For a horocycle $S$ of $\mathbb{H}^2$, a locally finite set $\cals_0\subset S$ that is invariant under a parabolic isometry $g$ fixing the ideal point $v$ of $S$ can be enumerated $\{s_i\,|\,i\in\mathbb{Z}\}$ so that for each $i$, the compact interval of $S$ bounded by $s_i$ and $s_{i+1}$ contains no other points of $\cals_0$, and $g(s_i) = s_{i+k}$ for each $i$ and some fixed $k\in\mathbb{Z}$.

For such an enumeration, the closed convex hull of $\cals_0$ in $\mathbb{H}^2$ is the non-overlapping union $\bigcup_i T_i$, where $T_i$ is horocyclic ideal triangle with vertices at $s_i$, $s_{i+1}$ and $v$ for each $i$.\end{lemma}

\begin{proof} Applying an isometry of $\mathbb{H}^2$, one can arrange that $S = \mathbb{R}+i$, so $v = \infty$ and for some fixed $r\in\mathbb{R}$, $g(z) = z+r$ for all $z\in\mathbb{H}^2$.  We may thus simply enumerate the points of $\cals_0$ in order of increasing real part, choosing an arbitrary $s\in\cals_0$ to be $s_0$.  By $g$-invariance there are no points of $\cals$ in the interval between $s_k = g(s_0)$ and $g(s_1)$, so since $g$ preserves order or real parts it follows that $g(s_1) = s_{k+1}$.  An induction argument gives $g(s_i) = s_{i+k}$ for all $i$.

For fixed $i$ and any $n\in\mathbb{Z}$, let $r_n = \frac{1}{2}(\Re s_n - \Re s_i)$.  The geodesic arc through $s$ and $s_i$ is contained in the Euclidean circle in $\mathbb{C}$ with center at $\Re s_i+ r_n\in\mathbb{R}$ and radius $\sqrt{r_n^2+1}$.  It follows from $g$-invariance that $\Re s_n\to\infty$ as $n\to\infty$.  This implies that the geodesic arcs from $s_i$ to the $s_n$ intersect the vertical line through $s_{i+1}$ at a sequence of points whose imaginary parts go to infinity.  Hence by convexity the closed convex hull $C$ of $\cals_0$ contains the entire geodesic ray $[s_{i+1},\infty)$.

Since the above holds for any $i$ it is not hard to see that $C$ contains $\bigcup_i T_i$, which further is clearly a non-overlapping union.  For any $x\in\mathbb{H}^2$ outside this union there is some $i$ such that $\Re s_i\leq \Re x\leq\Re s_{i+1}$, and the geodesic arc joining $s_i$ to $s_{i+1}$ separates $x$ from $T_i$ in the region $\{\Re s_i\leq \Re z\leq\Re s_{i+1}\}$.  It follows that the geodesic $\gamma_i$ of $\mathbb{H}^2$ containing $s_i$ and $s_{i+1}$ separates $x$ from $\cals_0$, so since $x$ was arbitrary $C = \bigcup_i T_i$.\end{proof}

\begin{lemma}  Let $F$ be a complete, non-compact hyperbolic surface of finite area with universal cover $\pi\co\mathbb{H}^2\to F$, and for finite $\cals\subset F$ let $\widetilde{\cals} = \pi^{-1}(\cals)$.  If $C_T$ is a non-compact centered dual two-cell and $\Delta(e_0,v_{\infty})\subset C_T$ (recall Definition \ref{tree cells}) then there is a Delaunay two-cell $C_{v_{\infty}}$ invariant under a parabolic element of $\pi_1 F$ fixing $v_{\infty}$, with $\Delta(e_0,v_{\infty})\subset C_{v_{\infty}}$ such that the geometric dual $\gamma$ to $e_0$ is an edge of $C_{v_{\infty}}$.

For a Delaunay two-cell $C$ invariant under a parabolic subgroup $\Gamma$ of $\pi_1 F$ fixing some $v_{\infty}\in S_{\infty}$, the intersection of $\mathit{int}\, C$ with the centered dual complex, if non-empty, is a $\Gamma$-invariant union of the form $\bigcup\Delta(e,v_{\infty})-\gamma$, where $e$ (with geometric dual $\gamma$) ranges over the set of non-centered non-compact Voronoi edges with ideal endpoint $v$.\end{lemma}

\begin{proof}  For $C_T$ and $\Delta(e_0,v_{\infty})$ as above, we recall from Lemma \ref{exclusive horocycle} that the endpoints of the geometric dual $\gamma$ to $e_0$ are contained in a unique horocycle $S$ with ideal point $v_{\infty}$, and the horoball $B$ bounded by $S$ satisfies $B\cap\cals = S\cap\cals$.  Theorem \ref{pseudo EP} therefore implies that the closed convex hull of $S\cap\cals$ in $\mathbb{H}^2$ is a Delaunay two-cell $C_{v_{\infty}}$ invariant under a parabolic subgroup of $\pi_1 F$ fixing $v_{\infty}$.  The decomposition of Lemma \ref{inclusive horocycle} includes $\Delta(e_0,v_{\infty})$.

For a Delaunay two-cell $C$ invariant under a parabolic subgroup $\Gamma$ of $\pi_1 F$ fixing some $v_{\infty}\in S_{\infty}$, Corollary \ref{it's a decomp!} implies in particular that the intersection of $C$ with the centered dual complex of $\widetilde{\cals}$ is $\Gamma$-invariant.  Since the centered dual is a union of cyclic Delaunay cells and triangles of the form $\Delta(e,v)$ as above (recall Definition \ref{centered dual}), $\mathit{int}\, C$ intersects only the ideal triangles.  For any such $\Delta(e,v)$ the lemma's first assertion implies that $v = v_{\infty}$, that $\Delta(e,v)\subset C$, and the geometric dual $\gamma$ to $e$ is an edge of $C$.\end{proof}

\begin{proposition}\label{augmented centered dual} For a complete, oriented, non-compact hyperbolic surface $F$ of finite area with locally isometric universal cover $\pi\co\mathbb{H}^2\to F$ and a finite set $\cals\subset F$,  there is a \mbox{\rm centered dual} \mbox{\rm complex plus} of $\widetilde{\cals}\doteq\pi^{-1}(\cals)$ with underlying space:
$$\bar{\mathbb{H}}^2\doteq\mathbb{H}^2\cup\{v\in S_{\infty}\,|\,g(v) = v\ \mbox{for some parabolic}\ g\in\pi_1 F\}$$  
Its vertex and edge sets include those of the centered dual and also, for each parabolic fixed point $v\in S_{\infty}$ that is not the endpoint of a non-centered Voronoi edge, the vertex $v$ and an edge $[\bs,v]$ for each vertex $\bs$ of the corresponding Delaunay two-cell $C_v$ (as in Theorem \ref{pseudo EP}).  The two-cells consist of:\begin{itemize}
\item  all centered dual two-cells; and
\item  for each Delaunay two-cell $C$ that is invariant under a parabolic subgroup $\Gamma$ of $\pi_1 F$ fixing $v\in S_{\infty}$, and each edge $\gamma$ of $C$ that does not intersect the interior of a centered dual two-cell, $T\cup\{v\}$ where $T$ is the horocyclic ideal triangle spanned by $v$ and $\gamma$.\end{itemize}
For each cell $C$ of the centered dual plus, $\pi$ is embedding on $\mathit{int}\,C$ and $\pi|_{C\cap\mathbb{H}^2}$ extends to a continuous map $C\to \bar{F}$, where $\bar{F}$ is the closed surface obtained by adding one point to each cusp of $F$.  The images determine a cell decomposition of $\bar{F}$.\end{proposition}

Proposition \ref{augmented centered dual} follows directly from the prior results of this section, recalling that each geometric dual two-cell is contained in a centered dual two-cell by Definition \ref{centered dual}.

\subsection{Proof of the main theorem}\label{proof of main app}  For closed hyperbolic surfaces, the upper bound of Theorem \ref{main app} is assertion 1) of the main theorem of \cite{Bavard}.  We reproduce this result below as Lemma \ref{Boroczky}, and prove it along the same lines using B\"or\"oczky's Theorem.  We first saw this kind of argument in the related result \cite[Corollary 3.5]{KM}.

Though it is not noted in \cite{Bavard}, this strategy gives bounds for all surfaces.  They are not sharp in the non-compact case.  To obtain sharp bounds we use the centered dual machine.

\begin{lemma}\label{Boroczky}  For a complete, oriented hyperbolic surface $F$ of finite area and any $x\in F$, $\mathit{injrad}_x F \leq r_{\chi}$ for $r_{\chi} > 0$ defined by the equation $3\alpha(r_{\chi}) = \pi/(1-\chi(F))$, where $\alpha(r)$ is defined in Theorem \ref{main app} and $\chi(F)$ is the Euler characteristic of $F$.\end{lemma}

\begin{proof}  Let $\pi\co\mathbb{H}^2\to F$ be the universal cover.  For an open disk $D$ of radius $r$ embedded in $F$, $\pi^{-1}(D)$ is a packing of $\mathbb{H}^2$ by radius-$r$ disks that is invariant under the $\pi_1 F$-action by covering transformations of $\mathbb{H}^2$.

For such a disk $D$ let $s\in F$ be the center of $D$, and let $\cals=\pi^{-1}(s)\subset\mathbb{H}^2$.  $\cals$ is locally finite at $\pi_1 F$-invariant, so its Voronoi tessellation is as well.  The $\pi_1 F$-action is transitive on $\cals$, so also is on Voronoi $2$-cells.  Thus any fixed two-cell is a fundamental domain for the $\pi_1 F$-action, hence by the Gauss-Bonnet theorem has area $-2\pi\chi(F)$.

Each component $\widetilde{D}$ of $\pi^{-1}(D)$ is contained in a Voronoi $2$-cell $\widetilde{V}$, and the \textit{local density} of $\pi^{-1}(D)$ at $\widetilde{D}$ is by definition the ratio $\mathrm{Area}(\widetilde{D})/\mathrm{Area}(\widetilde{V})$.  Since $\mathrm{Area}(\widetilde{D}) = 2\pi(\cosh r-1)$, Bor\"oczky's theorem \cite{Bor} asserts the following bound:
$$ \frac{\mathrm{Area}(\widetilde{D})}{\mathrm{Area}(\widetilde{V})} = \frac{2\pi(\cosh r -1)}{-2\pi\chi(F)} \leq \frac{3\alpha(r)(\cosh r -1)}{\pi - 3\alpha(r)}$$
The quantity on the right-hand side above is interpreted as follows: it is the ratio of the area of intersection of an equilateral triangle $T$ that has all side lengths $2r$ with the union of radius $r$ disks centered at its vertices, divided by the area of $T$.  Solving the inequality above yields $\pi\leq 3\alpha(r)(1-\chi(F))$, and the desired bound follows since $\alpha$ decreases with $r$.
\end{proof}

\begin{theorem}\label{main app}\mainapp\end{theorem}

\begin{proof} If $n=0$ (i.e.~$F$ is closed), the equation defining $r_{g,n}$ simplifies to $(2g-1)3\alpha(r_{g,0}) = \pi$.  This bound is supplied by Lemma \ref{Boroczky}, so we assume below that $F$ has at least one cusp.

Let $\pi\co\mathbb{H}^2\to F$ be a locally isometric universal cover, fix $x\in F$, and let $\widetilde{\cals}=\pi^{-1}\{x\}$.  Let $\{C_1,\hdots,C_k\}$ be a complete set of representatives in $\bar{\mathbb{H}}^2$ for $\pi_1 F$-orbits of two-cells of the centered dual plus.  By Proposition \ref{augmented centered dual} these project to the two-cells of a cell decomposition of $\bar{F}$, which is obtained from $F$ by compactifying each cusp with a single point.  Their interiors project homeomorphically to $F$, so by the Gauss--Bonnet theorem:
$$  \mathrm{Area}(C_1) + \hdots + \mathrm{Area}(C_k) = -2\pi\chi(F)  $$
(We should technically replace each $C_i$ above by $C_i\cap\mathbb{H}^2$ above.)  Since the centered dual plus has a vertex at each parabolic fixed point of $\pi_1 F$,  its projection has a vertex at each point of $\bar{F}-F$, in addition to the vertex at $x$, for a total of $n+1$.  Each edge of the projection either begins and ends at $x$ or joins $x$ to a point of $\bar{F}-F$.  Edges in the former category have length at least $d\doteq 2\mathit{injrad}_x F$, and those in the latter intersect $F$ in infinite-length arcs.

Each point of $\bar{F}-F$ is contained in at least one cell of the centered dual plus.  Each horocyclic ideal triangle (including all those not in the centered dual, recall Proposition \ref{augmented centered dual}) has area   at least $D_0(\infty,d,\infty)$ by Proposition \ref{horocyclic defects}.  For a cell $C_i$ with $n_i \geq 4$ edges, Theorem \ref{main for noncompact} asserts:
$$ \mathrm{Area}(C_i) \geq D_0(\infty,b_0(d,d),\infty) + (n_i-3)D_0(b_0(d,d),d,d)$$

Each non-triangular cell $C_i$ of the centered dual plus that is entirely contained in $F$ is compact and hence satisfies the bound of Theorem \ref{main} (cf.~Proposition \ref{centered main}):
$$ \mathrm{Area}(C_i) \geq (n_i-2) D_0(b_0(d,d),d,d) $$
A triangular such cell $C_i$ satisfies $\mathrm{Area}(C_i) \geq D_0(d,d,d)$ by Corollary \ref{monotonicity}.

Since $F$ has genus $g$ and $n$ cusps, $\bar{F}$ is closed of genus $g$, and the projection of the centered dual plus is a cell decomposition with vertex set $\{x\}\cup (\bar{F}-F)$ of order $n+1$.  It satisfies the Euler characteristic identity $v-e+f = \chi(\bar{F})$, where $v$, $e$, and $f$ are the numbers of vertices, edges, and faces, respectively.  Substituting $n+1$ for $v$ and $2-2g$ for $\chi(\bar{F})$ yields:
$$ e - f = (n+1) - (2-2g) = 1 - \chi(F) $$
After renumbering if necessary, there exists $k_0 \leq k$ such that $C_i$ has an ideal vertex if and only if $i\leq k_0$. Each such $C_i$ has only one ideal vertex, so $k_0\geq n$ since each of the $n$ points of $\bar{F}-F$ is in the projection of such a cell.  We will apply the area inequalities recorded above, together with the following:
$$ D_0(\infty,b_0(d,d),\infty) > D_0(\infty,d,\infty) > D_0(b_0(d,d),d,d) > D_0(d,d,d) $$
These follow respectively from Proposition \ref{horocyclic defects}, Corollary \ref{big ideals}, and Corollary \ref{monotonicity}.  Together with the above they imply that for $i\leq k_0$,
$$ \mathrm{Area}(C_i) \geq D_0(\infty,d,\infty) + (n_i-3)D_0(d,d,d), $$
with equality holding if and only if $C_i$ is a horocyclic ideal triangle with finite side of length $d$.  For $k_0< i\leq k$ we have
$$ \mathrm{Area}(C_i) \geq (n_i-2) D_0(d,d,d), $$
with equality again holding if and only if $C_i$ is a triangle with all sides of length $d$.  Applying these inequalities and the Gauss--Bonnet theorem yields:\begin{align*}
  -2\pi\chi(F) & \geq k_0\cdot D_0(\infty,d,\infty) + \left(\sum_{i=1}^k (n_i-2) - k_0\right)\cdot D_0(d,d,d) \\
    & \geq n\cdot D_0(\infty,d,\infty) + \left(\sum_{i=1}^k n_i - 2k - n\right)\cdot D_0(d,d,d)\end{align*}
Equality holds here if and only if $k_0 = n$, i.e. every ideal point of $\bar{F}$ is in a unique $C_i$, and every Delaunay edge has length $d$.  The sum of $n_i$ counts each edge of the centered dual plus twice, so $\sum_{i=1}^k n_i - 2k = 2e-2f = 2-2\chi(F))$.  Moreover, $D_0(\infty,d,\infty) = \pi-2\beta(r)$ and $D_0(d,d,d) = \pi-3\alpha(r)$, where $r = d/2$.  It is not hard to check that $\alpha$ and $\beta$ are strictly decreasing functions of $r$, so substituting above yields the desired inequality.

The final assertion follows from the fact that a surface triangulated by equilateral and horocyclic triangles with all (finite) side lengths equal to $d$ has its isometry class determined by the combinatorics of the triangulation, so for a fixed finite collection of such triangles there are only finitely many possibilities.\end{proof}

\subsection{Some examples}\label{examples}

Below we describe a closed, oriented hyperbolic surface $F$ of genus $g$ with maximal injectivity radius $r_g$.  The same examples were constructed in \cite{Bavard}, but we give an alternate approach that extends easily to the non-compact case.

\begin{example}\label{closed examples}  Fix $g\geq 2$ and let $r_g = r_{g,0}$ from Theorem \ref{main app}.  Substituting into the defining equation for $r_{g,0}$ there and solving for $\alpha(r_g)$ yields:
$$ \alpha(r_g) = \frac{\pi}{3(2g-1)} $$
Let $T_1,\hdots,T_{4g-2}$ be a collection of equilateral triangles, each with all vertex angles equal to $\alpha(r_g)$, arranged in $\mathbb{H}^2$ sharing a vertex $v$ so that for $1\leq i < j\leq 4g-6$, $T_i\cap T_j$ is an edge of each, if $j=i+1$, or else $v$.  Then $P_g \doteq T_1\cup T_2\cup\hdots\cup T_{4g-2}$ is a $4g$-gon with all edge lengths equal and vertex angles that sum to
$$ (4g-2)\cdot 3\cdot\frac{\pi}{3(2g-1)} = 2\pi $$
Label the edges of $P$ as $a_1,b_1,c_1,d_1,a_2,b_2,\hdots,d_{g-1},a_g,b_g,c_g,d_g$ in cyclic order.  For each $i$ let $f_i$ be the orientation-preserving hyperbolic isometry such that $f_i(a_i) = \bar{c}_i$ and $c_i = f(P)\cap P$, and let $g_i$ have $g_i(b_i) = \bar{d}_i$ and $d_i = g_i(P)\cap P$.  Here the bar indicates that when $a_i$, $b_i$, $c_i$ and $d_i$ are given the boundary orientation from $P$, $f_i|_{a_i}$ and $g_i|_{b_i}$ reverse orientation.

One easily shows that the edge-pairing of $P$ described above has a single quotient vertex, so since the vertex angles of $P$ sum to $2\pi$ the Poincar\'{e} polygon theorem implies that the group $G=\langle f_1,g_1,\hdots,f_g,g_g\rangle$ acts properly discontinuously on $\mathbb{H}^2$ with fundamental domain $P$.  

For $r\leq r_g$, if the open metric disk $D_r(v)$ of radius $r$ centered at $v$ does not embed in $F\doteq \mathbb{H}^2/G$ under the quotient map $\pi\co\mathbb{H}^2\to F$ then $D_r(v)$ intersects a translate $D_r(w)$ for some $w\in\pi^{-1}(\pi(v))$.  For each vertex $w$ of each $T_i$, $T_i$ contains the entire sector of the open metric disk $D_r(w)$ that it determines; see \FullSector.  Moreover, disks of radius $r$ centered at distinct vertices of $T_i$ do not meet.  It follows that $D_r(v)\cap D_r(w) = \emptyset$ for any distinct vertices $v$ and $w$ of $P$, and that $P$ contains the full sector of any such $D_r(w)$ that it determines.  We therefore find that $D_r(v)$ embeds in $F$, upon noting that $\pi^{-1}(\pi(v))$ is the set of vertices of $G$-translates of $P$.  Thus $F$ has injectivity radius $r_g$ at $P$.\end{example}

\begin{example}\label{noncompact examples}  Fix $g\geq 0$ and $n\geq 1$ (excluding $(g,n) = (0,1)$ or $(0,2)$), and take $r_{g,n}$ as in Theorem \ref{main app}.  Let $T_1,\hdots,T_{4g+n-2}$ be  equilateral triangles with side lengths $2r_{g,n}$, arranged as in Example \ref{closed examples} so that their union is a $(4g+n)$-gon $P_0$.  Label the edges of $P_0$ in cyclic fashion as:
$$e_1,\hdots,e_n,a_1,b_1,c_1,d_1,a_2,\hdots,d_{g-1},a_g,b_g,c_g,d_g, $$
so that $v = e_1\cap d_g$. Then append horocyclic ideal triangles $S_1,\hdots,S_n$, each with finite side length $2r_{g,n}$, to $P_0$ so that $S_i\cap P_0 = e_i$ for each $i$.  Let $P = P_0\cup\bigcup S_i$.

For $1\leq i\leq n$, let $c_i$ be the parabolic isometry fixing the ideal point of $S_i$ and taking one of its sides to the other, and for $1\leq i\leq g$ let $f_i(a_i) = \bar{c}_i$ and $g_i(b_i)=\bar{d}_i$ as in Example \ref{closed examples}.  As in that example, each vertex of $P$ is equivalent to $v$ under the resulting edge-pairing.  By definition of $\alpha$ and $\beta$, the vertex angles of $P$ sum to:
$$ 3(4g+n)\alpha(r_{g,n}) + 2n\beta(r_{g,n}) = 2\pi $$
Therefore the Poincar\'e polygon theorem implies that $G = \langle c_1,\hdots,c_n,f_1,g_1,f_2,\hdots,f_g,g_g\rangle$ acts properly discontinuously on $\mathbb{H}^2$ with fundamental domain $P$ and quotient $F = \mathbb{H}^2/G$ a complete hyperbolic surface.

Inspecting the edge pairing one finds that $F$ has $n$ cusps.  Its area is equal to that of $P$, $(4g+2n-2)\pi - 2\pi = 2\pi(2g-2+n)$, so by the Gauss--Bonnet theorem $F$ has genus $g$.  We claim that $F$ has injectivity radius $r_{g,n}$ at the projection of $v$; the argument is completely analogous to that of Example \ref{closed examples}.\end{example}

\begin{corollary}\label{special monotonicity}  For any $r>0$, the function $x\mapsto D_0(2r,x,x)$ is continuous and increasing on $[2r,\infty]$.  In particular,
$$ \pi - 3\alpha(r) = D_0(2r,2r,2r) < D_0(2r,\infty,\infty) = \pi -2\beta(r);$$
whence $2\beta(r) < 3\alpha(r)$.\end{corollary}

\begin{proof}  For $2r\leq x <\infty$ it is clear that $(2r,x,x)\in\calc_3$, so the function above is continuous and increasing on $[2r,\infty)$ (recall Corollary \ref{monotonicity}).  It is also continuous at $\infty$: for $D_R(d,J)$ as defined in \FullSector, that result implies that $D_0(2r,x,x) = D_R(d,J)$ with $R=0$, $d = 2r$, and $J = x$.  Then \cite[Lemma 6.7]{DeB_cyclic_geom} implies the result (recall the explicit definition of $D_0(x,\infty)$ in Proposition \ref{horocyclic defects}.)\end{proof}

\begin{example}\label{closed sequence}  For fixed $g\geq 2$, Corollary \ref{special monotonicity} and the definition of $r_{g-1,2}$ imply:
$$ (4g-2)3\alpha(r_{g-1,2}) = (4g)3\alpha(r_{g-1,2}) + 4\beta(r_{g-1,2}) + (6\alpha(r_{g-1,2})-4\beta(r_{g-1,2})) > 2\pi $$
(Recall Theorem \ref{main app}.)  Hence $r_{g-1,2} < r_{g,0}$, since $\alpha(r)$ decreases in $r$.  The inequality above also holds for any $r\in (r_{g-1,2},r_{g,0})$, and some rearrangement yields:
$$ (2g-2)2\pi > (4g-2)(\pi - 3\alpha(r)) = (4g-2)D_0(2r,2r,2r) $$
On the other hand, for such $r$ we also have $(4g-4)3\alpha(r) + 4\beta(r) < 2\pi$, so an analogous rearrangement yields:
$$ (2g-2)2\pi < (4g-4)(\pi-3\alpha(r)) + 2(\pi-2\beta(r)) = (4g-4)D_0(2r,2r,2r) + 2D_0(2r,\infty,\infty) $$
Thus by Corollary \ref{special monotonicity} and the intermediate value theorem there exists $x\in(2r,\infty)$ with\begin{align}\label{satisfaction}
  (2g-2)2\pi = (4g-4)D_0(2r,2r,2r) + 2D_0(2r,x,x)\end{align}
We arrange triangles $T_1,T_2,\hdots,T_{4g-2}$ as in Example \ref{closed examples}; however in this case only the last $4g-4$ are equilateral, each with sides of length $2r$.  We take $T_1$ and $T_2$ isosceles, each with one side of length $2r$ and others of length $x$, arranged so that $b_1$ and $T_2\cap T_3$ have length $2r$ and $a_1$, $T_1\cap T_2$, and $c_1$ have length $x$.  Here the sides of $P = \bigcup T_i$ are cyclically labeled as in Example \ref{closed examples}, starting at $v$, so that in particular $a_1$ and $b_1$ are sides of $T_1$, and $T_2\cap\partial P = c_1\sqcup\{v\}$.

As in Example \ref{closed examples}, for each $i$ there exists $f_i$ taking $a_i$ to $\bar{c}_i$, and $g_i$ taking $b_i$ to $\bar{d}_i$.  The collection $\{f_i,g_i\}_{i=1}^g$ is an edge-pairing on $P$ with a single quotient vertex.  The sum of all vertex angles of $P$ is the sum over $i$ of the vertex angle sum of each $T_i$.  Since the area of $T_i$ is the difference between $\pi$ and its vertex angle sum, the choice of $x$ and the equation above imply that the vertex angle sum of $P$ is $2\pi$.  Therefore by the Poincar\'e polygon theorem $G = \langle f_i,g_i\,|\, 1\leq i\leq g\rangle$ acts discontinuously on $\mathbb{H}^2$ with fundamental domain $P$.

The proof that $F = \mathbb{H}^2/G$ has injectivity radius $r$ at the projection of $v$ follows that of Example \ref{closed examples}.  The only additional note is that $T_1$ and $T_2$, being isosceles, are still centered, so the conclusions \FullSector\ apply to them as well.  

We claim that the minimal injectivity radius of $F$ approaches infinity as $r\to r_{g-1,2}^+$.  This follows from two sub-claims: first, that the solution $x$ to equation (\ref{satisfaction}) above goes to infinity as $r\to r_{g-1,2}$, and second, that the arc in $T_1$ joining points halfway up its sides with length $x$ has length approaching $0$.  Toward the second, a hyperbolic trigonometric calculation shows that the length $d$ of this arc satisfies:
$$  \cosh d = 1 + \frac{\cosh(2r)  -1}{2\cosh x + 2} $$
The second sub-claim therefore follows from the first.  The point $p$ at the midpoint of $a_1$ has distance at most $2d$ from its image in $c_1$, so $F$ has injectivity radius at most $d$ at the projection of $p$, and the claim follows.

It remains to show the first sub-claim.  Toward this end, let us recall from Theorem \ref{main app} that $r_{g-1,2}$ is defined by the equation:
$$ (2g-2)2\pi = (4g-4)D_0(2r_{g-1,2},2r_{g-1,2},2r_{g-1,2}) + 2D_0(2r_{g-1,2},\infty,\infty) $$
Therefore by Corollary \ref{special monotonicity}, for any fixed $x_0$ with $2r_{g-1,2}< x_0 < \infty$ we have:
$$ (2g-2)2\pi > (4g-4)D_0(2r_{g-1,2},2r_{g-1,2},2r_{g-1,2}) + 2D_0(2r_{g-1,2},x_0,x_0) $$
For $r$ near $r_{g-1,2}$ the function $r\mapsto (4g-4)D_0(2r,2r,2r) + D_0(2r,x_0,x_0)$ is continuous, so there exists $\epsilon > 0$ such that the inequality above holds with $r_{g-1,2}$ replaced with any $r\in(r_{g-1,2},r_{g-1,2}+\epsilon)$.  Since $D_0(2r,x,x)$ increases in $x$ this implies for any such $r$ that the solution to (\ref{satisfaction}) lies outside $[2r,x_0]$.  This proves the sub-claim.
\end{example}

\bibliographystyle{plain}
\bibliography{voronoi2}
\end{document}